\documentclass[12pt]{article}

\usepackage{amsfonts}
\usepackage{amscd}
\usepackage{amsmath}
\usepackage{epsfig}
\usepackage{amsthm}
\usepackage{amssymb}
\usepackage{amsbsy}
\usepackage{latexsym}
\usepackage{eucal}
\usepackage{mathrsfs,bm}
\usepackage{mathtools}
\usepackage{tikz}
\usepackage{pgfplots}
\usepgfplotslibrary{fillbetween}
\usetikzlibrary{shadings,intersections}
\usepackage{tkz-euclide}
\pgfdeclarelayer{bg}
\pgfsetlayers{bg,main}
\usepackage[colorlinks]{hyperref}
\usepackage{xcolor}
\usepackage{soul}
\usepackage{multicol}
\usetikzlibrary {arrows.meta}

\pgfplotsset{compat=1.14}

 \newtheorem{theorem}{Theorem}[section]

\newtheorem{proposition}[theorem]{Proposition}

\newtheorem{remark}[theorem]{Remark}
\newtheorem{lemma}[theorem]{Lemma}
\newtheorem{example}[theorem]{Example}

\def\be{\begin{equation}}
\def\ee{\end{equation}}
\def\ben{\begin{displaymath}}
\def\een{\end{displaymath}}
\def\baa{\begin{eqnarray}}
\def\eaa{\end{eqnarray}}

\def\ba{\begin{aligned}}
\def\ea{\end{aligned}}
\renewcommand{\leq}{\leqslant}
\renewcommand{\geq}{\geqslant}

\newcommand{\ord}{\operatorname{ord}}

\makeatletter
\@addtoreset{equation}{section}
\makeatother

\begin{document}
\title{Triangulations of singular constant curvature spheres via Belyi functions and determinants of Laplacians}

\author{Victor Kalvin \footnote{ Department of Mathematics, Dawson College, 3040 Sherbrooke St. W, Montreal, Quebec H3Z 1A4, Canada;  { E-mail:  vkalvin@gmail.com}}}

\date{}
\maketitle

\begin{abstract} 
We study the zeta-regularized spectral determinant of the Friedrichs Laplacians on the singular spheres obtained by cutting and glueing copies of constant curvature (hyperbolic, spherical, or flat) double triangle. The determinant is explicitly expressed in terms of the corresponding Belyi functions and the determinant of the Friedrichs Laplacian on the double triangle. The latter determinant was found in a closed explicit form in  ArXiv:2112.02771~\cite{KalvinLast}. 
 In examples we consider the cyclic, dihedral, tetrahedral, octahedral, and icosahedral triangulations, and find the determinant for the corresponding spherical, Euclidean, and hyperbolic Platonic surfaces. These surfaces correspond to stationary points of the determinant. 
\end{abstract}

\section{Introduction}\label{Intro}
We study the zeta-regularized spectral determinant  of the Friedrichs scalar Laplacian on the surfaces constructed by cutting and gluing copies of a constant curvature  $S^2$-like double triangle. We call these surfaces glued or  triangulated. The geometric (combinatorial) cutting and gluing scheme is described in terms of the corresponding Belyi function~\cite{Belyi}. Or, equivalently, in terms of {\it dessins d'enfants}~\cite{Gr}. 

The Belyi function maps the (source) Riemann surface to the (target) Riemann sphere. Any constant curvature double triangle is isometric to the target Riemann sphere with  explicitly constructed conformal metric with three conical singularities~\cite{KalvinLast}. 
The glued surface is isometric to the source Riemann surface equipped with the pullback of the  explicit conformal metric by the Belyi function. This provides us with a geodesic triangulation and an
explicit uniformization of the glued surface.  In particular, in the case of flat right isosceles  $S^2$-like double triangle, we get the square-tiled flat surfaces, see e.g.~\cite{Sh,Zorich} and references therein.

 In other words,
for studying the spectral determinant we suggest using the natural decomposition via triangulation based on the celebrated results of Belyi, Grothendieck, Shabat and Voevodskii. 
The price for this is that we first need to study the spectral determinant on surfaces with conical singularities~\cite{KalvinLast,KalvinCCM,KalvinJFA,KalvinJGA}. 

The main idea is to explicitly express the spectral determinant of the glued surface in terms of the corresponding Belyi function and the spectral determinant of the constant curvature $S^2$-like double triangle.  The latter determinant was found in a closed explicit form in~\cite{KalvinLast}.

Let us recall that not all Riemann surfaces can be glued/triangulated in this manner. But for the most interesting surfaces, like the Fermat curve, the Bolza curve, the  Hurwitz surfaces, including Klein's quartic, the Platonic surfaces, etc., it is certainly possible, thanks to the famous Belyi theorem~\cite{Belyi,KW,ShVo}.  
Due to the highest possible number of authomorphisms~\cite{ShVo,harts}, some of the triangulated surfaces (equipped with the natural smooth hyperbolic metric) are expected to be stationary points of the spectral determinant. To the best of our knowledge, no closed explicit expression for the spectral determinant of any of these surfaces is known yet.

 In this paper, we study the spectral determinant of the genus zero triangulated surfaces. For the triangulations of singular constant curvature   (hyperbolic, spherical, or flat) spheres via Belyi functions, we explicitly express the spectral determinant in terms of the Belyi maps, and the determinant of the constant curvature $S^2$-like double triangle. In particular, with each bicolored plane tree~\cite{BZ,Bishop} we can naturally associate an infinite family of constant curvature singular spheres, and then the corresponding infinite family of explicit spectral invariants, that is the family of the corresponding spectral determinants. For some closely related geometric constructions and invariants see e.g.~\cite{HogRiv,Riv1,Riv2,Sprin,thurston}. 
 
  Let us stress that due to conical singularities on the cuts, none of the presently known BFK-type gluing formulae~\cite{BFK} can be used. We rely on a completely different approach relating the determinants of Laplacians on the target Riemann sphere and on the ramified covering via anomaly formulae for the determinants~\cite{KalvinLast,KalvinJFA}.

As a byproduct, our approach allows for explicit evaluation of the Liouville action~\cite{CMS,Z-Z,T-Z}, which can be of independent interest, cf.~\cite{P-T-T,ParkTeo,T-T}. It would be interesting to check if the explicit expressions can be reproduced by using conformal blocks~\cite{Z-Z}. 
  
  As is known~\cite{T-Z}, the Liouville action generates the famous accessory parameters  as their common antiderivative.  We explicitly express the accessory parameters of the triangulated singular constant curvature spheres  in terms of the Belyi functions and the orders of conical singularities. We also find the derivative of the Liouville action with respect to the order of a conical singularity in the spirit of~\cite{Z-Z}.   As one can expect, this allows us to show that the five Platonic solids are special also in the context of this paper: their surfaces correspond to stationary points of the determinant.

Recall that for smooth metrics on Riemann surfaces extrema of spectral determinants were studied in a series of papers by Osgood, Phillips, and Sarnak~\cite{OPS,OPS.5,OPS1}; see also~\cite{Kim} for an extension of their results, and~\cite{EM} for recent closely related results in the four-dimensional case.

We illustrate the main results of this paper with a number of examples: We consider the cyclic, dihedral, tetrahedral, octahedral, and icosahedral triangulations~\cite{Klein}, and find explicit expressions for the spectral determinant of the corresponding spherical, Euclidean, and hyperbolic Platonic surfaces.   In particular, we explicitly evaluate the determinants of the regular hyperbolic octahedron with vertices of angle $\pi$ (a double cover is the Bolza curve) and the regular hyperbolic icosahedron corresponding to the tessellation by $(2,3,7)$-triangles (related to Klein's quartic~\cite{SSW,KW,ShVo}); see Example~\ref{EOctahedron} and Example~\ref{EIcosahedron}.

 As the angles of the conical points of a Platonic surface go to zero (and the area remains fixed), the determinant grows without any bound, cf. Fig~\ref{Platonic}. In the limit, the conical points turn into cusps and we obtain an ideal polyhedron~\cite{Judge,Judge2,KimWilkin}. The spectrum of the corresponding Laplacians is no longer discrete.

 This paper is organized as follows.  Section~\ref{S2} contains preliminaries and the main results of the paper (Theorem~\ref{main}). Namely, after giving some preliminaries on constant curvature $S^2$-like double triangles and their spectral determinants in Subsection~\ref{SS2.1}, we describe the geometric setting of this paper and formulate the main results in Theorem~\ref{main} of Subsection~\ref{SS2.2}.
 
  Section~\ref{PFmain} is entirely devoted to the proof of  Theorem~\ref{main}: In Subsection~\ref{SS3.1.} we prove Proposition~\ref{PBdet}, which is a preliminary version of Theorem~\ref{main}.  In Subsection~\ref{BFET} we refine the result of  Proposition~\ref{PBdet} by using the natural Euclidean equilateral triangulation~\cite{ShVo,VoSh}. This completes the proof of Theorem~\ref{main}.
 
 Section~\ref{Unif} is devoted to the uniformization, the accessory parameters, the Liouville action, and stationary points of the spectral determinant.
 
Finally, Section~\ref{ExAppl} contains illustrating examples and applications.  

\section{Preliminaries and main results}\label{S2}

\subsection{Double triangles}\label{PrelimTriangulation}\label{SS2.1}
Consider the involutional tessellation of the standard round sphere $x_1^2+x_2^2+x_3^2=1$ in $\Bbb R^3$ by two congruent $(1,1,1)$-triangles; as usual, by a $(k,\ell,m)$-triangle we mean a geodesic triangle (spherical, Euclidean, or hyperbolic) with internal angles $\pi/k$, $\pi/\ell$, and $\pi/m$. In particular,  the $(1,1,1)$-triangles are hemispheres. Let the standard sphere in $\Bbb R^3$   be identified with the Riemann sphere $\overline{\Bbb C}_z=\Bbb C\cup\infty$ by means of the stereographic projection. Let the $(1,1,1)$-triangles correspond to the upper $\Im z\geq 0$ and lower  $\Im z\leq 0$ half-planes. Without loss of generality, we can assume that three marked points on the sphere (the vertices of the congruent $(1,1,1)$-triangles) have the coordinates  $z=0$, $z=1$, and $z=\infty$. 
 
 Let us endow the Riemann sphere with a unique unit area constant curvature metric having conical singularities of order $\beta_0$, $\beta_1$, and $\beta_\infty$ at the marked points $z=0$, $z=1$, and $z=\infty$ correspondingly. For simplicity, we assume that the orders $\beta_j$ are in the interval $(-1,0]$, or, equivalently, that the cone angles $2\pi(\beta_j+1)$ are positive and do not exceed $2\pi$; here $j\in\{0,1,\infty\}$.   We denote the metric by $m_{\pmb\beta}$ and say  that it  represents the divisor 
 $$
 \pmb\beta=\beta_0\cdot 0+\beta_1\cdot 1+\beta_\infty\cdot \infty
 $$
of degree $|\pmb \beta|=\beta_0+\beta_1+\beta_\infty$. 

 The metric $m_{\pmb\beta}$ can be explicitly constructed as the pullback of the Gaussian curvature $2\pi(|\pmb\beta|+2)$ model metric
 \be\label{mm}
 \frac{4|dw|^2}{(1+2\pi(|\pmb\beta|+2)|w|^2)^2}
 \ee
 by an appropriately normalized Schwarz triangle function $z\mapsto w=w_{\pmb\beta}(z)$, see~\cite[Appendix]{KalvinLast}.   Note that the conditions $\beta_j -|\pmb \beta|/2>0$ are necessary and sufficient for the existence of the metric. 
 
The Schwarz triangle function $w_{\pmb\beta}$ maps the upper half-plane $\Im z>0$ into a geodesic  (light-coloured) triangle with internal angle $\pi(\beta_0+1)$ at the vertex $w_{\pmb\beta}(0)=0$, $\pi(\beta_1+1)$ at the vertex $w_{\pmb\beta}(1)\in\Bbb R$, and $\pi(\beta_\infty+1)$ at the vertex $w_{\pmb\beta}(i\infty)$. The analytic continuation of $w_{\pmb\beta}$ (from the upper half-plane $\Im z>0$ through the interval $(0,1)$ of the real axis) maps the lower half-plane $\Im  z<0$ into the (dark-coloured) reflection of the geodesic triangle in the side joining the points $w_{\pmb\beta}(0)$ and $w_{\pmb\beta}(1)$. Thus we obtain a geodesic with respect to the model metric~\eqref{mm} bicolored double triangle, which is hyperbolic for $|\pmb\beta|<-2$  (cf. Fig.~\ref{TemplateHyp}), spherical for  $|\pmb\beta|>-2$ (cf. Fig.~\ref{TemplateSph}), and Euclidean for  $|\pmb\beta|=-2$ (cf. Fig.~\ref{TemplateEuc}).
 \begin{figure}[h]
\centering\begin{tikzpicture}[scale=1.25]
\draw[gray!50,fill=gray!50!blue!30]  (6.75,0)--(0,0) -- (5,2) arc (-165:-112:3cm);
\draw[gray!50,fill=gray!80!blue!80]  (0,0) -- (5,-2) arc (165:112:3cm);
\draw[black,solid,thick] (6.75,0)node[anchor=west]{$w_{\pmb \beta}(1)$}--(0,0)node[anchor=north east]{$w_{\pmb \beta}(0)$} -- (5,2)node[anchor=south]{$w_{\pmb \beta}( i \infty)$} arc (-165:-112:3cm) ; 
\draw[black,solid,thick]  (0,0) -- (5,-2)node[anchor=north]{$w_{\pmb \beta}(-i \infty)$} arc (165:112:3cm); 

\filldraw[black] (0,0) circle (1pt);\filldraw[black] (6.75,0) circle (1pt);\filldraw[black] (5,2) circle (1pt);\filldraw[black] (5,-2) circle (1pt);

\draw[black] (1.5,0) arc (0:33:1cm);
\draw[black] (2.2,0)node[anchor=south]{$\pi(\beta_0+1)$};

\draw[black] (5.2,1.48) arc (-73:-146:.7cm);
\draw[black] (4.5,1.48)node[anchor=north]{$\pi(\beta_\infty+1)$};

\draw[black] (5.5,0) arc (180:138:1cm);
\draw[black] (4.8,0)node[anchor=south]{$\pi(\beta_1+1)$};
\end{tikzpicture}
\caption{Hyperbolic double triangle.}
\label{TemplateHyp}
\end{figure}
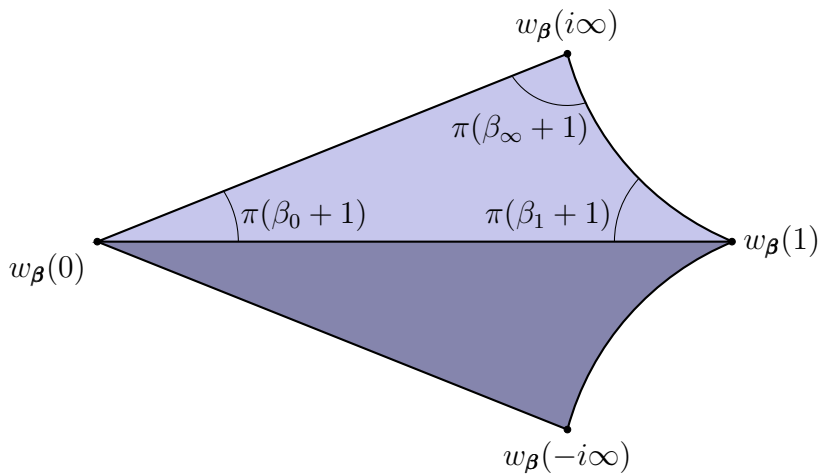

\begin{figure}[h]
\centering\begin{tikzpicture}[scale=1.25]
\draw[gray!50,fill=gray!60!blue!30]  (4.98,2.01)--(0,0) --(6.75,0)arc (180-165:180-112:3cm) ;
\draw[gray!50,fill=gray!80!blue!80]  (4.98,-2.01)--(0,0) --(6.75,0)arc (-180+165:-180+112:3cm) ;
\draw[black,solid,thick] (4.98,2.01)node[anchor=south]{$w_{\pmb \beta}(i \infty)$}--(0,0)node[anchor=north east]{$w_{\pmb \beta}(0)$} --(6.75,0) node[anchor=west]{$w_{\pmb \beta}(1)$}arc (180-165:180-112:3cm)   ; 
\draw[black,solid,thick]  (4.98,-2.01)node[anchor=north]{$w_{\pmb \beta}(-i \infty)$}--(0,0) --(6.75,0)arc (-180+165:-180+112:3cm);

\filldraw[black] (0,0) circle (1.5pt);\filldraw[black] (6.75,0) circle (1.5pt);\filldraw[black] (5,2) circle (1.5pt);\filldraw[black] (5,-2) circle (1.5pt);

\draw[black] (1.5,0) arc (0:33:1cm);
\draw[black] (2.2,0)node[anchor=south]{$\pi(\beta_0+1)$};

\draw[black] (5.55,1.7) arc (-50:-135:.9cm);
\draw[black] (4.9,1.55)node[anchor=north]{$\pi(\beta_\infty+1)$};

\draw[black] (5.9,0) arc (180:122:1cm);
\draw[black] (5.2,0)node[anchor=south]{$\pi(\beta_1+1)$};
\end{tikzpicture}
\caption{Spherical double triangle.}
\label{TemplateSph}
\end{figure}
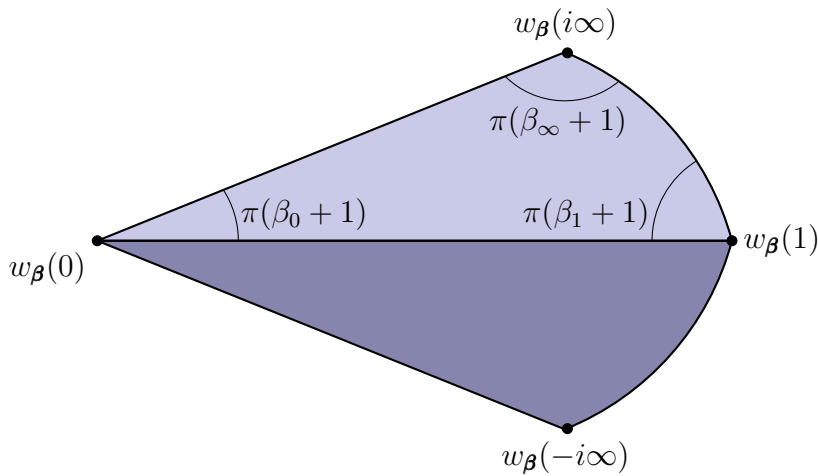

\begin{figure}[h]
\centering\begin{tikzpicture}[scale=1.25]
\draw[gray!50,fill=gray!60!blue!30] (0,0) -- (5,2)--(7,0)--(0,0);
\draw[gray!50,fill=gray!80!blue!80] (0,0) -- (7 ,0) -- (5,-2)--(0,0); 
\draw[black,solid,thick] (0,0)node[anchor=north east]{$w_{\pmb \beta}(0)$} -- (5,2)node[anchor=south]{$w_{\pmb \beta}(i \infty)$}--(7,0)node[anchor=west]{$w_{\pmb \beta}(1)$}--(0,0); 
\draw[black,solid, thick]  (7 ,0) -- (5,-2)node[anchor=north]{$w_{\pmb \beta}(-i \infty)$}--(0,0); 

\filldraw[black] (0,0) circle (1.5pt);\filldraw[black] (7,0) circle (1.5pt);\filldraw[black] (5,2) circle (1.5pt);\filldraw[black] (5,-2) circle (1.5pt);

\draw[black] (1.5,0) arc (0:33:1cm);
\draw[black] (2.2,0)node[anchor=south]{$\pi(\beta_0+1)$};

\draw[black] (5.5,1.5) arc (-60:-138:1cm);
\draw[black] (4.9,1.4)node[anchor=north]{$\pi(\beta_\infty+1)$};

\draw[black] (6,0) arc (180:135:1cm);
\draw[black] (5.3,0)node[anchor=south]{$\pi(\beta_1+1)$};
\end{tikzpicture}
\caption{Euclidean  double  triangle.} 
\label{TemplateEuc}
\end{figure}
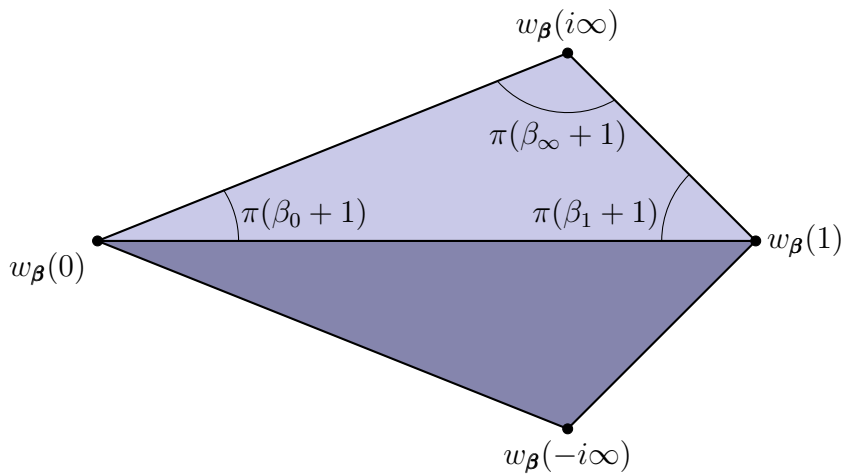

 To make it $S^2$-like, the double triangle is folded along the interval $[w_{\pmb\beta}(0), w_{\pmb\beta}(1)]\subset \Bbb R$, and then the corresponding sides of the light- and dark-coloured triangles are glued pairwise. Namely, the side joining $w_{\pmb\beta}(0)$ and $w_{\pmb\beta}(i\infty)$  is glued to the side joining $w_{\pmb\beta}(0)$ and $w_{\pmb\beta}(-i\infty)$, and then the side joining  $w_{\pmb\beta}(1)$ and $w_{\pmb\beta}(i\infty)$  is glued to the side joining $w_{\pmb\beta}(1)$ and $w_{\pmb\beta}(-i\infty)$. The resulting $S^2$-like bicolored double triangle is isometric to the genus zero constant curvature surface $(\overline {\Bbb C}_z, m_{\pmb\beta})$ with three conical singularities. The isometry is given by the Schwarz triangle function $w_{\pmb\beta}$.

 The potential $\phi$ of the conformal metric $m_{\pmb\beta}=e^{2\phi}|dz|^2$  satisfies the Liouville equation
\be\label{Liouv}
e^{-2\phi}(-4\partial_z\partial_{\bar z}\phi)={2\pi(|\pmb\beta|+2)} ,\quad z\in\Bbb C\setminus\{0,1\},
\ee
and obeys the asymptotics 
 \be\label{asympt_phi}
 \ba
 \phi(z)&=\beta_0\log|z|+\phi_0+o(1), \quad z\to 0,\quad  \phi_0=\Psi(\beta_0,\beta_1,\beta_\infty),   
 \\
 \phi(z)&=\beta_1\log|z-1|+\phi_1+o(1), \quad z\to 1,\quad  \phi_1=\Psi(\beta_1,\beta_0,\beta_\infty),    
 \\
 \phi(z)&=-(\beta_\infty+2)\log|z|+\phi_\infty+ o(1), \quad z\to\infty,\quad \phi_\infty= \Psi(\beta_\infty,\beta_1,\beta_0).
 \ea
 \ee
Here  $\Psi(\beta_1,\beta_2,\beta_3)=\Phi(\beta_1,\beta_2,\beta_3)+ \log 2$ with explicit function $\Phi$ from~\cite[Prop. A.2)]{KalvinLast}.

The Liouville equation indicates that outside of the marked points the Gaussian curvature of the metric $m_{\pmb\beta}$  is $2\pi(|\pmb\beta|+2)$, while the asymptotics~\eqref{asympt_phi} indicate that at the marked points $z=0$,$z=1$, and $z=\infty$ the metric has conical singularities of order $\beta_0$, $\beta_1$, and $\beta_\infty$ correspondingly~\cite{Troyanov}.

On the singular constant curvature surface $(\overline{\Bbb C}_z, m_{\pmb\beta})$ we consider the Laplace-Beltrami operator $\Delta_{\pmb\beta}=-e^{-2\phi}4\partial_z\partial_{\bar z}$ as an unbounded operator in the usual $L^2$-space. The operator is initially defined on the smooth functions supported outside of the conical singularities and not essentially selfadjoint. We take the Friedrichs selfadjoint extension,  which we still denote by $\Delta_{\pmb\beta}$ and call the Friedrichs Laplacian or simply Laplacian for short.  The spectrum of $\Delta_{\pmb\beta}$ consists of non-negative isolated eigenvalues of finite multiplicity, and the zeta-regularized spectral determinant  $\det \Delta_{\pmb\beta}$ of $\Delta_{\pmb\beta}$ can be introduced in the standard well-known way.

 In what follows it is important that the spectral zeta-regularized determinant of the Friedrichs Laplacian $\Delta_{\pmb\beta}$ on the constant curvature surface $(\overline{\Bbb C}_z, m_{\pmb\beta})$  with three conical singularities, or, equivalently, on the isometric $S^2$-like bicolored double triangle,  is an explicit function
\be\label{CalcVar}
(-1,0]^3\ni(\beta_0,\beta_1,\beta_\infty)\mapsto \det \Delta_{\pmb\beta}
\ee
 found in~\cite[Corollary~1.3]{KalvinLast}.

\subsection{Triangulations and determinants of Laplacians}\label{SS2.2}
Recall that a (non-constant) meromorphic function $f: X\to \overline{\Bbb C}_z$ on a compact Riemann surface $X$ is called a Belyi function, if it is ramified at most above three points~\cite{Belyi}. Any Belyi function can alternatively be described via the corresponding {\it dessin d'enfant},  which is usually defined as the graph formed  on $X$ by the preimages $f^{-1}([0,1])$ of the real line segment $[0,1]$ with black points placed at the preimages $f^{-1}(0)$ of zero and white points placed at the preimages $f^{-1}(1)$ of $1$,  e.g.~\cite{CIW,Gr,LZ,Sch,Wolfart}. Any meromorphic function on the Riemann sphere is a rational function. If, in addition, $f$ has only a single pole that is at infinity, then $f$ is a polynomial~\cite{BZ,Bishop,LZ}.

 Consider a  Belyi function 
$f: \overline {\Bbb C}_x\to  \overline {\Bbb C}_z$  ramified at most above the  marked points $z=0$, $z=1$, and $z=\infty$.  This defines a ramified (branched) covering and a bicolored triangulation of the Riemann sphere $\overline {\Bbb C}_x$, e.g.~\cite{VoSh,LZ}.  Namely, the function $f$ maps:  {\it i}) the sides of the bicolored triangles on $\overline {\Bbb C}_x$ to the line segments $(-\infty,0)$, $(0,1)$, and $(1,\infty)$ of the real axis; {\it ii}) the vertices of the triangles  to the marked points $0$, $1$, and  $\infty$; {\it iii}) the light-colored triangles to the upper half-plane $\Im z>0$, and the dark-colored triangles  to the lower half-plane $\Im z <0$.  The  number of bicolored double triangles on $ \overline {\Bbb C}_x$ is exactly the degree $\deg f=\max\{\deg P,\deg Q\}$ of $f$,  where $f(x)=P(x)/Q(x)$ with coprime polynomials  $P$ and $Q$. 

For example, the Belyi function 
\be\label{BFIcos}
f(x)=1728 \frac{x^5(x^{10}-11 x^5-1)^5}  {(x^{20}+228(x^{15}-x^{5})+494 x^{10}+1)^3}, \quad \deg f =60,
\ee
 defines the icosahedral triangulation~\cite{Klein1}, which corresponds to the tessellation of  the standard round sphere with bicolored spherical $(2,3,5)$-triangles in Fig.~\ref{Icos Triang Pic}, Fig.~\ref{TemplateSph}. As usual, we identify the Riemann sphere $\overline {\Bbb C}_x$ with the standard round sphere in $\Bbb R^3$ by means of  the stereographic projection.

\begin{figure}
 \centering\includegraphics[scale=.3]{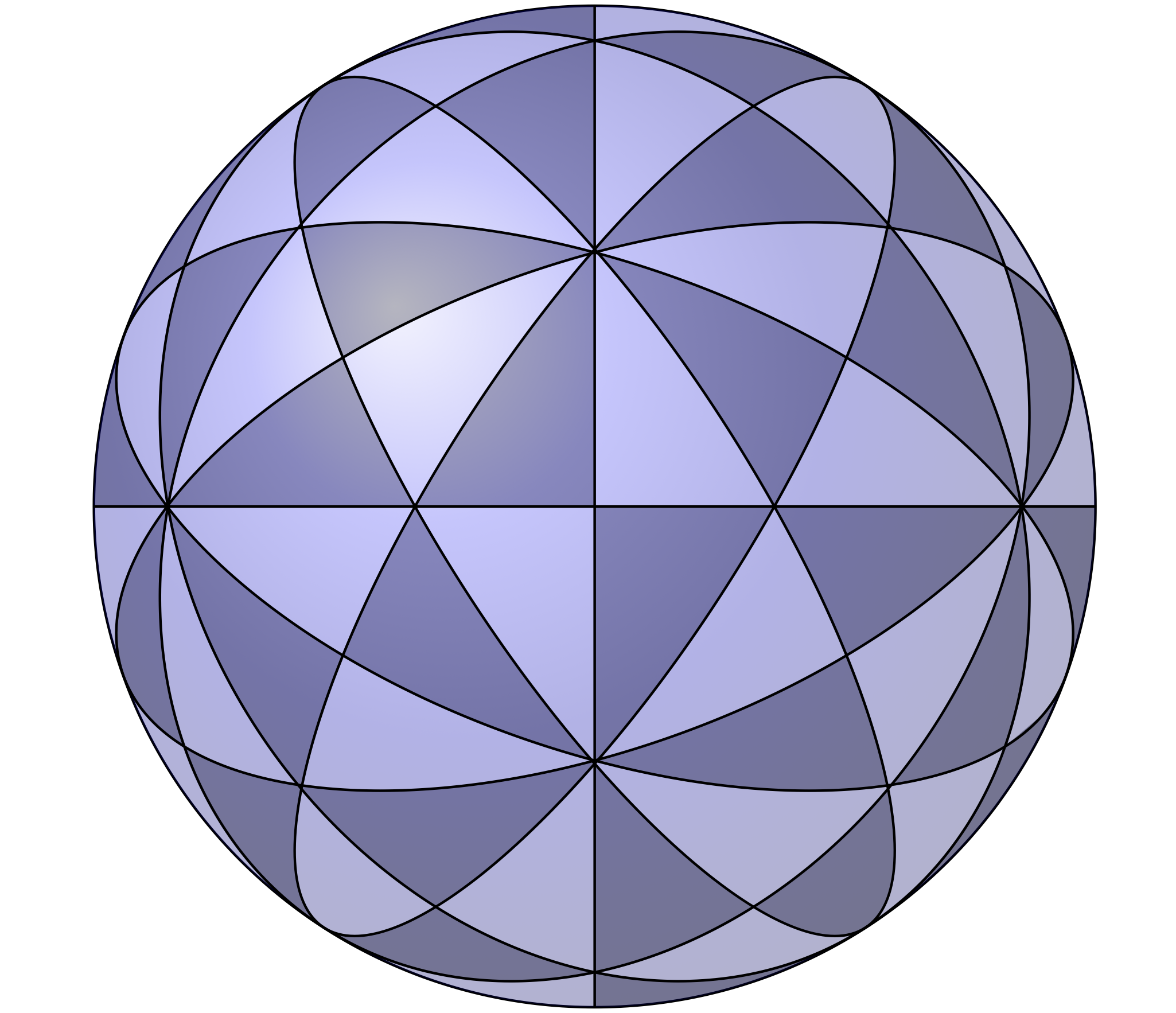}
\caption{Icosahedral triangulation}\label{Icos Triang Pic}
\end{figure}

Depending on the sign of $|\pmb\beta|+2$, the composition $w_{\pmb\beta}\circ f $ (of the Schwarz triangle function $w_{\pmb\beta}$ with a  Belyi function $f: \overline {\Bbb C}_x\to  \overline {\Bbb C}_z$) maps the light-coloured (resp. the dark-coloured) triangles defined by $f$ on $\overline {\Bbb C}_x$ to the light-coloured (resp. the dark-coloured) hyperbolic, spherical, or Euclidean triangle in  Fig.~\ref{TemplateHyp}, Fig.~\ref{TemplateSph}, and Fig.~\ref{TemplateEuc}.  The pullback of the model metric~\eqref{mm} by this composite function, or, equivalently, the pullback $f^*m_{\pmb\beta}$ of the metric $m_{\pmb \beta}$ by $f$,  is a singular metric  of Gaussian curvature $2\pi(|\pmb\beta|+2)$ on $\overline{\Bbb C}_x$. The triangulated surface $(\overline{\Bbb C}_x, f^*m_{\pmb\beta})$ is isometric to the one obtained by cutting and gluing $\deg f $ copies of the $S^2$-like bicolored double triangle in accordance with a  combinatorial scheme prescribed by $f$.  In particular, Fuchsian triangle groups are included into consideration~\cite{CIW}. 

Let us note that the above construction also works in the opposite (constructive) direction:  the combinatorial cutting and gluing schemes of bicolored double triangles can be described with the help of a (fixed) base star in $\overline{\Bbb C}_z$ with the terminal vertices at   $z=0,1,\infty$ and the constellations. For any constellation there exists a Belyi function defining the cutting and gluing scheme, e.g.~\cite{LZ}.

 For instance, one can preserve the gluing scheme of bicolored double triangles in Fig.~\ref{Icos Triang Pic}, and replace each light-coloured (resp. dark-coloured) spherical  $(2,3,5)$-triangle with the image of the upper (resp. lower) complex half-plane under the Schwarz triangle function  $w_{\pmb\beta}$ in  Fig.~\ref{TemplateHyp}, Fig.~\ref{TemplateSph}, or Fig.~\ref{TemplateEuc}. Then, as a result, we obtain a hyperbolic, Euclidean, or spherical singular surface isometric to $(\overline{\Bbb C}_x, f^*m_{\pmb\beta})$, where $f$ is the Belyi function~\eqref{BFIcos}. However, in general, it is not easy to find a Belyi function corresponding to a particular gluing scheme, even though for the regular dihedra and the surfaces of five Platonic solids they were first found by Schwarz and Klein~\cite{Klein}, see also~\cite{LZ,Margot Zvonkin}.

The main result of this paper is a closed explicit formula for the zeta-regularized spectral determinant of the Friedrichs Laplacian on the triangulated singular constant curvature spheres $(\overline{\Bbb C}_x, f^*m_{\pmb\beta})$. We express the determinant in terms of the Belyi function $f$ and the determinant~\eqref{CalcVar} of the corresponding $S^2$-like bicolored double triangle. Recall that the latter determinant is explicitly evaluated in~\cite{KalvinLast}.

Let us list the preimages of the marked points $\{0,1,\infty\}\subset\overline{\Bbb C}_z$ under $f$ as the marked points  $x_1,x_2,\dots,x_n$ on the Riemann sphere $\overline{\Bbb C}_x$; here $n=\deg f +2$. We shall always assume that $x_k=\infty$ for some $k\leq n$: this can always be achieved by replacing $f$ with equivalent Belyi function $f\circ \mu$, where $\mu$ is a  M\"obius transformation satisfying $\mu(\infty)=x_k$. The surfaces $(\overline{\Bbb C}_x, f^*m_{\pmb\beta})$ and $(\overline{\Bbb C}_x, (f\circ\mu)^*m_{\pmb\beta})$ are isometric, and, as a consequence, the corresponding spectral determinants are equal. 

Introduce the ramification divisor $$
\pmb f :=\sum_{k=1}^n \ord _k f \cdot x_k,
$$ 
where $ \ord _k f$ is the ramification order of the Belyi function $f$  at $x_k$.  
If $x_k$ is a pole of $f$, then its multiplicity coincides with $ \ord _k f +1$.  If  $f'(x_k)=0$, then $\ord_k f+1$ is  the order of zero at $x=x_k$ of the function  $x\mapsto f(x)- f(x_k)$.  If $x_k$ is not a pole of $f$ and $f'(x_k)\neq 0$, then the  ramification order  $ \ord _k f$  is zero.  One can also interpret $x_k$ as a vertex in the {\it dessin d'enfant} corresponding to $f$, then the number $ \ord _k f +1$ is its graph-theoretic degree, or, equivalently, the number of edges emanating from the vertex $x_k$.  
By the Riemann-Hurwitz formula  the degree $|\pmb f|:=\sum_{k=1}^n \ord _k f$ of the divisor $\pmb f$ satisfies $|\pmb f|=2\deg f-2$.
  
From the Liouville equation~\eqref{Liouv} and the asymptotics~\eqref{asympt_phi} it follows that  
 the potential 
$$
f^*\phi:=\phi\circ f+\log |f'|
$$
of the pullback metric  $ f^*m_{\pmb\beta}=e^{2 f^*\phi}|dx|^2$   satisfies the Liouville equation
\be\label{LUx}
 e^{-2f^*\phi}\bigl(-4\partial_x\partial_{\bar x} (f^*\phi)\bigr)= {2\pi(|\beta|+2)},\quad x\in \Bbb C\setminus \{x_1,\dots, x_n\},
\ee
and obeys the asymptotics
\be\label{asfphi}
\ba
(f^*\phi)(x)&=(f^*\beta)_k\log |x-x_k|+O(1), \quad x\to x_k\neq \infty,\\
(f^*\phi)(x)&=-((f^*\beta)_k+2)\log |x|+O(1),\quad x \to x_k=\infty,
\ea
\ee
where
\be\label{orders*}
 \left(f^*\pmb\beta\right)_k:=(\ord_k f+1)(\beta_{f(x_k)}+1)      -1.
\ee

Geometrically this means that outside of the points $x_1,\dots,x_k$ the metric $ f^*m_{\pmb\beta}$ is a regular metric of constant Gaussian curvature $2\pi(|\beta|+2)$, while at each point $x_k$ the metric has a conical singularity of order $(f^*\beta)_k$, or, equivalently, of angle $2\pi((f^*\beta)_k+1)$, see e.g.~\cite{Troyanov}.  That is because $\ord_k f+1$ vertices of bicolored double triangles meet together at $x_k$ to form the conical singularity. Each of the vertices makes a contribution of angle $2\pi(\beta_{f(x_k)}+1)$ into the angle of the conical singularity, cf.~\eqref{orders*}. Note that it could be so that for a point $x_k$ we have $(f^*\beta)_k=0$ (for example,  this is the case if $\ord_k f=1$ and $\beta_{f(x_k)}=-1/2$), then $x_k$ is also a regular point of the metric $ f^*m_{\pmb\beta}$. Clearly, $\beta_{f(x_k)}=\beta_\infty$ if $x_k$ is a pole of $f$, $\beta_{f(x_k)}=\beta_0$ if $x_k$ is a zero of $f$, and $\beta_{f(x_k)}=\beta_1$ if $f(x_k)=1$, where $\beta_j$ is the same as in~\eqref{asympt_phi}.

Now we are in a position to formulate the main results of this paper. 
\begin{theorem}[Spectral determinant of triangulated spheres]\label{main}
Let $\det \Delta_{\pmb\beta}$ stand for the explicit function~\eqref{CalcVar}, whose value at a point is the zeta-regularized spectral determinants of  the Friedrichs selfadjoint extension of the Laplace-Beltrami operator on the unit area $S^2$-like double triangle, or, equivalently, on  $(\overline {\Bbb C}_z,m_{\pmb\beta})$; see Sec.~\ref{PrelimTriangulation} and~\cite{KalvinLast}.

 Let $f: \overline {\Bbb C}_x \to \overline {\Bbb C}_z$ be a Belyi function unramified outside of the set $\{0,1,\infty\}$ and  such that $f(\infty)\in\{0,1,\infty\}$.  Denote by  $\ord_k f$  the ramification order of $f$ at $x_k$, where  $x_1,x_2,\dots,x_n\in \overline {\Bbb C}_x$ are the preimages of the points $0$, $1$, and $\infty$ under $f$.

  Consider the surface $(\overline {\Bbb C}_x, f^*m_{\pmb\beta})$ isometric to the one glued from $\deg f$ copies of the double triangle in accordance with a pattern defined by  $f$.  
Then for the zeta-regularized spectral determinant  $\det\Delta_{f^*\pmb\beta}$ of the Friedrichs selfadjoint extension of the Laplace-Beltrami operator on  $(\overline {\Bbb C}_x, f^*m_{\pmb\beta})$ we have
\be\label{S20STMNT}
\ba
\log\frac {\det\Delta_{f^*\pmb\beta}}{\deg f} 
 =& \deg f \cdot\log{\det\Delta_{\pmb\beta}} +\frac 1 6 \sum_{k} \left(\ord_k f+1-\frac 1 {\ord_k f+1}\right)\frac {\phi_{f(x_k)}}{\beta_{f(x_k)}+1}   
   \\
      & -\frac 1 6 \sum_{k}  \left( (f^*\pmb\beta)_k+1+\frac 1 { (f^*\pmb\beta)_k+1}  \right)\log (\ord_k f+1)
 \\
&  - \sum_{k}\Bigl(\mathcal C\left(\left(f^*\pmb\beta\right)_k\right)
  -(\ord_k f+1) \mathcal C(\beta_{f(x_k)})\Bigr) -  (\deg f-1){\bf C} +C_f,
  \ea
  \ee
 where $\phi_j$ and $\beta_j$ with $j\in \{0,1,\infty\}$ is the (explicit) uniformization data in~\eqref{asympt_phi}, and $(f^*\pmb\beta)_k$ is the same as in~\eqref{orders*}. The real-analytic function $(-\infty,0]\ni\beta\mapsto\mathcal C(\beta)$ is  defined by the equality
  \be\label{cbeta}
 \mathcal C(\beta)=2\zeta'_B(0;\beta+1,1,1)-2\zeta_R'(-1)-\frac {\beta^2}{6(\beta+1)}\log 2 -\frac \beta {12}+\frac 1 2 \log(\beta+1),
\ee
where $\zeta'_B$ and $\zeta_R'$ stand for the derivatives with respect to $s$ of the Barnes double zeta function $\zeta_B(s;\beta+1,1,1)$ and the Riemann zeta function $\zeta_R(s)$ respectively. 
For the constant ${\bf C}$ in~\eqref{S20STMNT} we have 
\be\label{bfC}
   {\bf C}=\frac 1 {6}-\frac 4 3 \log 2 -4\zeta_R'(-1) -\log\pi.
\ee 
Finally, the constant $C_f$ in~\eqref{S20STMNT} depends only on the Belyi function $f$. It is given by 
\be\label{C_F}
\ba
C_f & =\frac 1 {18} \sum_{ k:x_k\neq \infty}\sum_{\ \ \ell:x_k\neq x_\ell\neq \infty}\frac{(\ord_k f-2)(\ord_\ell f -2)}{\ord_k f+1} \log|x_k-x_\ell|
\\
&\qquad\qquad +\frac 1 6 \sum_{k}  \left( \frac{\ord_k f +1} 3 +\frac 3 { \ord_k f+1}  \right)\log (\ord_k f+1)
 \\
&\qquad\qquad +\frac 1 {6} \left( \deg f - \sum_{k}  \frac 3 { \ord_k  f+1}    \right)\log A_f,
  \ea
\ee
where 
\be\label{Exp_A}
A_f=\frac{|f(x)|^{-2/3}|f(x)-1|^{-2/3}|f'(x)|}{\prod_{k: x_k\neq\infty}|x-x_k|^{\frac 1 3 (\ord_k f-2)}}.
\ee
Note that $A_f$ is a scaling coefficient that does not depend on $x$ and can be easily evaluated for any particular Belyi function $f$. 
\end{theorem} 

As we show in the proof  of Theorem~\ref{main}, the expression~\eqref{C_F} for $C_f$ comes from the Euclidean equilateral triangulation defined on $\overline {\Bbb C}_x$ by $f$, cf.~\cite{ShVo,VoSh}. 
The readers whose interests are more in the results and applications than in the proof may wish to skip Section~\ref{PFmain} and proceed directly to Sections~\ref{Unif} and~\ref{ExAppl}.

\section{Proof of the main results}\label{PFmain}

\subsection{Equality~\eqref{S20STMNT} is valid with a constant  $C_f$}\label{SS3.1.}

In this subsection we prove that the representation for the spectral determinant~\eqref{S20STMNT}  is valid with some constant $C_f$ that does not depend on the metric $m_{\pmb\beta}$ on the target Riemann sphere $\overline{\Bbb C}_z$. As a byproduct we obtain an explicit formula for $C_f$ that is different from the one in~\eqref{C_F}. 
\begin{proposition}\label{PBdet}
The equality~\eqref{S20STMNT} in Theorem~\ref{main} is valid with a constant $C_f$ that depends only on the Belyi function $f$. Moreover,
  \be\label{C_f}
  \ba
  C_f=   \frac 1 6 \sum_{k: x_k\neq\infty, f(x_k)\neq\infty} \left(\frac{\ord_k f}{\ord_k f+1} \log|c_k|   + (\ord_k f+2) \log (\ord_k f+1)\right)
 \\
 + \frac 1 6 \sum_{k: x_k\neq\infty, f(x_k)=\infty} \left( \frac{\ord_k f+2}{\ord_k f+1} \log|c_k|-\ord_k f \log (\ord_k f+1) \right) 
 \\
 + \left. \frac 1 6 \left(- \frac{\ord_k f}{\ord_k f+1} \log|c_k| - (\ord_k f+2)\log (\ord_k f+1) \right)\right|_{k:x_k=\infty, f(\infty)=\infty} 
 \\
 +\left. \frac 1 6  \left( -\frac{\ord_k f+2}{\ord_k f+1} \log|c_k| + \ord_k f\log (\ord_k f+1)  \right)\right |_{k: x_k=\infty, f(\infty)\neq\infty},
  \ea
  \ee
where 
  \begin{itemize}
  \item $c_k$ stands for the first nonzero coefficient in the Taylor series of $f(x)-f(x_k)$ at $x_k$, if $x_k$ is not a pole of $f$;
 \item  $c_k$ stands for the first  nonzero coefficient of the Laurent series of $f$ at $x_k$, if $x_k$ is a pole of $f$.
 \end{itemize}
\end{proposition}

The proof of Proposition~\ref{PBdet} is preceded by Lemma~\ref{INT} below.  To formulate the lemma we need the following refined version of the asymptotics~\eqref{asfphi}  for the metric potential 
$f^*\phi$:
\be\label{asympPB1}
\ba
(f^*\phi) (x)=  \left(f^*\pmb\beta\right)_k  \log|x-x_k|    +(f^*\phi)_k +o(1),\quad x\to x_k\neq\infty,
\\
(f^*\phi) (x)= - \bigl( \left(f^*\pmb\beta\right)_k   +2\bigr)   \log|x|     +(f^*\phi)_k  +o(1) , \quad x\to x_k=\infty,
\ea
\ee
where  $\left(f^*\pmb\beta\right)_k$  is the same as in~\eqref{orders*}. Since $f^*\phi=\phi\circ f+\log |f'|$, it is not hard to see that the coefficients $(f^*\phi)_k$ in~\eqref{asympPB1} satisfy
\be\label{1}
\ba
(f^*\phi)_k=\phi_{f(x_k)}+(\beta_{f(x_k)}+1)\log|c_k|  +\log(\ord_k f+1), \quad \text{ if }\quad f(x_k)\neq\infty,
\\
(f^*\phi)_k=\phi_{f(x_k)}-(\beta_{f(x_k)}+1)\log|c_k|  +\log(\ord_k f+1), \quad \text{ if }\quad f(x_k)=\infty.
\ea
\ee
Here $\beta_j$ and $\phi_j$ with $j\in\{0,1,\infty\}$ are the same as in~\eqref{asympt_phi}, and $c_k$ is the same as in Proposition~\ref{PBdet}.

\begin{lemma} \label{INT} 
Recall that $m_{\pmb\beta}=e^{2\phi}|dz|^2$ and $f^*m_{\pmb\beta}=e^{f^*\phi}|dx|^2$. We have
$$
\ba
(|\pmb\beta|+2)\int_{\Bbb C}(f^*\phi) e^{2f^*\phi}\frac {dx\wedge d\bar x}{-2i}=(|\pmb\beta|+2)\deg f\int_{\Bbb C} \phi e^{2\phi}\frac {dz\wedge d\bar z}{-2i}
\\
-\sum_{k: f(x_k)\neq \infty}(f^*\phi)_k \ord_k f
+ \sum_{k: f(x_k)=\infty}  (f^*\phi)_k\ord_k f + \sum_{k}(f^*\pmb\beta)_k \log\left |(\ord_k f+1)c_k\right|
\\
+2 \sum_{k: x_k\neq\infty, f(x_k)=\infty}(f^*\phi)_k+\left. 2  \log\left |(\ord_k f+1)c_k\right|\right|_{k:x_k=\infty} -\left.2(f^*\phi)_k\right|_{k:x_k=\infty, f(\infty)\neq\infty},
\ea
$$
where $\left(f^*\pmb\beta\right)_k$ and  $(f^*\phi)_k$ are the coefficients in the asymptotics~\eqref{asympPB1} satisfying the equalities~\eqref{orders*} and~\eqref{1} respectively. 
\end{lemma}

\begin{proof}
We begin with the equality
 $$
\int_{\Bbb C} (f^*\phi) e^{2f^*\phi}\frac {dx\wedge d\bar x}{-2i}=\deg f \cdot\int_{\Bbb C} \phi e^{2\phi}\frac {dz\wedge d\bar z}{-2i}+\int_{\Bbb C} (\log |f'|) e^{2f^*\phi} \frac {dx\wedge d\bar x}{-2i}
$$
that easily follows from the relation $f^*\phi=\phi\circ f+\log |f'|$.
Thus, in order to prove the lemma, we only need to evaluate the last integral.

Thanks to the Liouville equation~\eqref{LUx} we have
\be\label{Int log f}
\ba
(|\pmb\beta|+2)\int_{\Bbb C} (\log |f'|) e^{2 f^*\phi} \frac {dx\wedge d\bar x}{-2i}=\lim_{\epsilon\to 0} \frac 1 {2\pi}\int_{\Bbb C^\epsilon}   (\log |f'|)\bigl (-4\partial_x\partial_{\bar x} (f^*\phi)\bigr) \frac {dx\wedge d\bar x}{-2i}
\\
=\lim_{\epsilon\to 0} \frac 1 {2\pi}\oint_{\partial \Bbb C^\epsilon} \Bigl((f^*\phi) \partial_{\vec n }(\log|f'|) -(\log|f'|) \partial_{\vec n } (f^*\phi)  \Bigr) \,  |d x|,
\ea
\ee
where $\Bbb C^\epsilon$ stands for the large disk $|x|>1/\epsilon$ with the  small disks $|x-x_k|<\epsilon$ encircling the marked points $x_k$ removed, and ${\vec n}$ is the unit outward normal. The last equality in~\eqref{Int log f} is valid because $x\mapsto \log|f'(x)|$ is a harmonic function on $\Bbb C^\epsilon$.

Notice that if $x_k$ is not a pole of  $f$, then we  have
\be\label{asymp_log_f1}
\ba
 \log|f'(x)|=\ord_k f\log |x-x_k|+\log\left |(\ord_k f+1)c_k\right|+o(1),\  x\to x_k\neq\infty,
\\
\log|f'(x)|=-(\ord_k f +2)\log|x| +\log\left |(\ord_k f+1)c_k\right|+o(1)  , \  x\to x_k=\infty.
\ea
\ee
Besides, if $x_k$ is a pole of   $f$, we have 
\be\label{asymp_log_f2}
\ba
  \log|f'(x)|=& -(\ord_k f +2)\log|x-x_k| 
  \\
  &\qquad \qquad+\log |(\ord_k f+1)c_k|+o(1),\quad  x\to x_k\neq \infty, 
\\
  \log|f'(x)|=& \ord_k f  \log|x|
 +\log |(\ord_k f+1)c_k| +o(1),\quad x\to x_k= \infty.
\ea
\ee

Below, for the derivatives along the unit outward normal vector ${\vec n}$ we use the equalities
$$
\partial_{\vec n}=-\partial_{|x-x_j|}\quad \text{on}\quad |x-x_j|=\epsilon;\quad \partial_{\vec n}=\partial_{|x|}\quad \text{on}\quad |x|=1/\epsilon.
$$

We calculate the integral in right hand side of~\eqref{Int log f}
 relying on~\eqref{asympPB1} together with~\eqref{asymp_log_f1} and~\eqref{asymp_log_f2}. As a result, we obtain
$$
\ba
\lim_{\epsilon\to 0}\frac 1 {2\pi}\oint_{\partial \Bbb C^\epsilon} \Bigl((f^*\phi) \partial_{\vec n }(\log|f'|) -(\log|f'|) \partial_{\vec n } (f^*\phi)  \Bigr)\,|d x|
\\
=\sum_{j=1}^n \frac 1 {2\pi}\lim_{\epsilon\to 0}\oint_{ |x-x_j|=\epsilon}\Bigl((f^*\phi) \partial_{\vec n }(\log|f'|) -(\log|f'|) \partial_{\vec n } (f^*\phi)  \Bigr)\,|d x|
\\
=\sum_{k: x_k\neq\infty, f(x_k)\neq \infty}\Bigl( -(f^*\phi)_k \ord_k f+(f^*\pmb\beta)_k \log\left |(\ord_k f+1)c_k\right|\Bigr)
\\
+ \sum_{k: x_k\neq\infty, f(x_k)=\infty}\Bigl (   (f^*\phi)_k(\ord_k f +2)+(f^*\pmb\beta)_k \log\left |(\ord_k f+1)c_k \right| \Bigr)
\\
+\Bigl( -(f^*\phi)_k  (\ord_kf +2)+(\log\left |(\ord_k f+1)c_n\right|) \bigl( (f^*\pmb\beta)_k  +2\bigr) \Bigr)\Bigr|_{k: x_k=\infty, f(\infty)\neq \infty}
\\
+\Bigl(    (f^*\phi)_k  \ord_k f  + ((f^*\pmb\beta)_k  +2)  \log\left |(\ord_k f+1)c_k\right|\Bigr)\Bigr|_{k: x_k=\infty, f(\infty)=\infty}.
\ea
$$
After regrouping the terms  in the right hand side, this implies the assertion of Lemma~\ref{INT}. 
\end{proof}

\begin{lemma}\label{ANOMALY}
For the determinant $\det\Delta_{f^*\pmb\beta}$ of the Friedrichs  Laplacian on $(\overline{\Bbb C}_x, f^*m_\beta)$ the anomaly formula 
\be\label{S20Anom}
\ba
\log \frac {\det\Delta_{f^*\pmb\beta}}{\deg f}   = -& \frac{|\beta|+2}{6}\int_{\Bbb C} f^*\phi e^{2f^*\phi}\,\frac {dx\wedge d\bar x} {-2i} 
+\frac  1 6 \sum_{k: x_k\neq \infty}\frac {\left(f^*\pmb\beta\right)_k}{\left(f^*\pmb\beta\right)_k+1} (f^*\phi)_k
\\&
   -\left. \frac 1 6
\frac{\left(f^*\pmb\beta\right)_k+2} {\left(f^*\pmb\beta\right)_k+1}(f^*\phi)_k\right |_{k:x_x=\infty}
  -\sum_{k}\mathcal C\bigl(\left(f^*\pmb\beta\right)_k\bigr)
 + {\bf C}
\ea
\ee
is valid. 
Here   $\mathcal C(\beta)$ is the function defined  in~\eqref{cbeta}, and $\bf C$ stands for the constant~\eqref{bfC}.
\end{lemma}

\begin{proof} The main argument of the proof is similar to the one in~\cite[Proof of Prop. 2.1]{KalvinLast}. 
For this reason, we skip the details that can be easily restored from the references~\cite{KalvinLast,KalvinCCM,KalvinJFA}. 

Let us first obtain an asymptotics for the determinant of the Friedrichs Dirichlet Laplacian  $\Delta^D_{f^*\pmb\beta}\!\!\restriction_{|x|\leq 1/\epsilon}$ on the disk $|x|\leq 1/\epsilon$ endowed with the metric $f^*m_{\pmb\beta}$ as $\epsilon\to 0^+$.
Denote by $\Delta^D_0 \!\!\restriction_{|x|\leq 1/\epsilon}$ the selfadjoint Dirichlet  Laplacian on the disk $|x|\leq 1/\epsilon$ equipped with the flat background metric $|dx|^2$. As is known~\cite{Weisberger}, for the determinant of the Dirichlet  Laplacian on the flat disk one has
$$
\log \det\Delta^D_0 \!\!\restriction_{|x|\leq 1/\epsilon}=\frac 1 3 \log\epsilon +\frac 1 3 \log 2-\frac 1 2\log 2 \pi -\frac 5 {12} -2\zeta_R'(-1).
$$
On the other hand, the Polyakov-Alvarez type formula from~\cite[Theorem~1.1.2]{KalvinJFA} reads
\be\label{PAdisk}
\ba
\log \frac {\det \Delta^D_{f^*\pmb\beta}\!\!\restriction_{|x|\leq 1/\epsilon}}{\det \Delta^D_0 \!\!\restriction_{|x|\leq 1/\epsilon}} =& -\frac {|f^*\pmb\beta|+2}{6\deg f }  \int_{\Bbb C} (f^*\phi) e^{2f^*\phi}\frac {dx\wedge d\bar x}{-2i} 
\\
& -\frac 1 {12\pi}\oint_{|x|=1/\epsilon}(f^*\phi)\partial_{\vec n}(f^*\phi)\,|dx|-\frac \epsilon {6\pi}\oint_{|x|=1/\epsilon} (f^*\phi)\, |dx|
\\
&
-\frac 1 {4\pi} \oint_{|x|=1/\epsilon} \partial_{\vec n}(f^*\phi)\, |dx|+\frac 1 6 \sum_{k: x_k\neq\infty} \frac {\left(f^*\pmb\beta\right)_k}{\left(f^*\pmb\beta\right)_k+1} (f^*\phi)_k
\\
&-\sum_{k: x_k\neq \infty} \mathcal C\bigl(\left(f^*\pmb\beta\right)_k\bigr),
\ea
\ee
where $|f^*\pmb\beta|=\sum_k  \left(f^*\pmb\beta\right)_k$ is the degree of the divisor $f^*\pmb\beta=\sum_k  \left(f^*\pmb\beta\right)_k\cdot x_k$. 

By the Gauss-Bonnet theorem~\cite{Troyanov}, the product of the (regularized) Gaussian curvature by the total area of the singular sphere $(\overline{\Bbb C}_x, f^*m_{\pmb\beta})$ equals  $2\pi(|f^*\pmb\beta|+2)$.  Since the (regularized) Gaussian curvature is the one of  $m_{\pmb\beta}$, and the total area is $\deg f$, we conclude that  $$
\frac{|f^*\pmb\beta|+2}{ \deg f}=|\pmb\beta|+2.
$$

In~\eqref{PAdisk} we also replace the integrals along the circle $|x|=1/\epsilon$  with their asymptotics.
 As a result,   the Polyakov-Alvarez type formula~\eqref{PAdisk} implies
$$
\ba
\log {\det \Delta^D_{f^*\pmb\beta}\!\!\restriction_{|x|\leq 1/\epsilon}}=-\frac {|\pmb\beta|+2}{6}  \int_{\Bbb C}(f^*\phi) e^{2f^*\phi}\frac {dx\wedge d\bar x}{-2i} +\frac 1 6 \sum_{k: x_k\neq\infty} \frac {\left(f^*\pmb\beta\right)_k}{\left(f^*\pmb\beta\right)_k+1} (f^*\phi)_k
\\
-\sum_{k:x_k\neq \infty} \mathcal C\bigl(  \left(f^*\pmb\beta\right)_k\bigr )
+\frac 1 6 \left(f^*\pmb\beta\right)_k \left((f^*\phi)_k+3\right)|_{k:x_k=\infty} 
+\frac 1 3 \log 2-\frac 1 2\log 2 \pi +\frac 7 {12} 
\\
-2\zeta_R'(-1)
+   \frac 1 6\left( \left(\left(f^*\pmb\beta\right)_k+2\right)^2 -2\left(\left(f^*\pmb\beta\right)_k+1\right)\right)|_{k:x_k=\infty} \log \epsilon + o(1),\ \epsilon\to 0^+.
\ea
$$

In a similar way, one can also obtain the asymptotics for the determinant of theFriedrichs Dirichlet Laplacian on the cap $|x|\geq 1/\epsilon$ of the singular sphere $(\overline{\Bbb C}_x, f^*m_\beta)$:
$$
\ba
\log \det\Delta^D_{f^*\pmb\beta}\!\!\restriction_{|x|\geq 1/\epsilon}=-\frac 1 6 ((\left(f^*\pmb\beta\right)_k+1)^2+1)\log \epsilon-\frac 1 6\left(\left(f^*\pmb\beta\right)_k+1+\frac 1 {\left(f^*\pmb\beta\right)_k+1}\right)(f^*\phi)_k
\\
-\mathcal C\bigl(\left(f^*\pmb\beta\right)_k\bigr) -2\zeta_R'(-1)-\frac 5 {12} +\frac 1 3 \log 2 -\frac 1 2 \log2\pi -\frac {\left(f^*\pmb\beta\right)_k} 2 +o(1),\quad \epsilon\to 0^+,
\ea
$$
where $k$ is such that $x_k=\infty$. In total, we have
\be\label{FinalAsympt}
\ba
\log {\det \Delta^D_{f^*\pmb\beta}\!\!\restriction_{|x|\leq 1/\epsilon}}+  \log \det\Delta^D_{f^*\pmb\beta}\!\!\restriction_{|x|\geq 1/\epsilon} =-\frac {|\pmb\beta|+2}{6}  \int_{\Bbb C}(f^*\phi) e^{2f^*\phi}\frac {dx\wedge d\bar x}{-2i} 
\\
+\frac 1 6 \sum_{k: x_k\neq\infty} \frac {\left(f^*\pmb\beta\right)_k}{\left(f^*\pmb\beta\right)_k+1} (f^*\phi)_k
 -\frac 1 6\left(1+\frac 1 {\left(f^*\pmb\beta\right)_k+1}\right) (f^*\phi)_k
 \\ -\sum_{k} \mathcal C\bigl(\left(f^*\pmb\beta\right)_k\bigr)+{\bf C}+\log 2 + o(1),\quad \epsilon\to0^+.
\ea
\ee

Now we  cut the singular sphere $\left(\overline{\Bbb C}_x, f^*m_{\pmb\beta}\right)$ along the circle $|x|=1/\epsilon$ with the help of the BFK formula~\cite[Theorem B$^*$]{BFK} to obtain
\be\label{BFKf}
\ba
\log \det \Delta_{f^*\pmb\beta}=&\log \deg f -\log 2 \\&+ \lim_{\epsilon\to 0^+}\left(\log \det \Delta^D_{f^*\pmb\beta}\!\!\restriction_{|x|\leq 1/\epsilon}+\log \det \Delta^D_{f^*\pmb\beta}\!\!\restriction_{|x|\geq 1/\epsilon} \right).
\ea
\ee
Here $\deg f$ is the total area of the singular sphere, and  the number $-\log 2$  is the value of  the difference
$$
\log \det \mathcal N(\epsilon)-\log \oint_{|x|=1/\epsilon}e^{f^*\phi}\,|dx|,
$$
 where $ \det \mathcal N(\epsilon)$ is the determinant of the Neumann jump operator on the cut $|x|=1/\epsilon$, see~\cite[Proof of Prop. 2.1]{KalvinLast} or~\cite{KalvinCCM}. As demonstrated in~\cite{KalvinJFA}, the BFK formula~\eqref{BFKf} remains valid in spite of the fact that the metric  $f^*m_{\pmb\beta}$ is singular.

In the limit, the asymptotics~\eqref{FinalAsympt} together with the BFK formula~\eqref{BFKf} implies the anomaly formula~\eqref{S20Anom}. This completes the proof of Lemma~\ref{ANOMALY}. 
\end{proof}

\begin{proof}[Proof of Proposition~\ref{PBdet}]
Thanks to  Lemma~\ref{INT}, the integral in the right-hand side of the anomaly formula in~Lemma~\ref{ANOMALY} reduces to an integral on the target sphere $(\overline{\Bbb C}_z,m_{\pmb\beta})$. Our next purpose is to express that integral in terms of the determinant of Laplacian $\Delta_{\pmb\beta}$. With this aim in mind, we write  down the anomaly formula
\begin{equation}\label{DetDelta}
\begin{aligned}
\log {\det\Delta_{\pmb\beta}} 
 = -\frac{|\pmb\beta|+2}{6}\int_{\Bbb C} \phi e^{2\phi}\,\frac {dz\wedge d\bar z} {-2i} 
  +&\frac  1 6 \sum_{j=0,1}\frac {\beta_j}{\beta_j+1}\phi_j 
  \\  
 & -\frac 1 6\frac {\beta_\infty+2} {\beta_\infty+1}\phi_\infty -\sum_{j\in\{0,1,\infty\}}\mathcal C(\beta_j)
+{\bf C},
 \end{aligned}
\end{equation}
which is~\eqref{S20Anom} in the particular case of $f\equiv 1$. Now we change the index of summation from $j\in\{0,1,\infty\}$ to $k=1,2,\dots,n$ as follows:
\begin{equation}\label{S20.2}
\begin{aligned}
\log
  {\det\Delta_{\pmb\beta}} 
 =-\frac{|\pmb\beta|+2}{6}\int_{\Bbb C} \phi e^{2\phi}\,\frac {dz\wedge d\bar z} {-2i} 
  +\frac  1 6 \sum_{k: f(x_k)\neq \infty}  \frac {\ord_k f+1}{\deg f} \frac {\beta_{f(x_k)}}{\beta_{f(x_k)}+1}\phi_{f(x_k)}
  \\-\sum_{k}  \frac {\ord_k f+1}{\deg f}\mathcal C(\beta_{f(x_k)})
  -\frac 1 6  \sum_{k:f(x_k)=\infty} \frac {\ord_k f+1}{\deg f}  \frac {\beta_{f(x_k)}+2} {\beta_{f(x_k)}+1}\phi_{f(x_k)}
  +{\bf C}.
\end{aligned}
\end{equation}
Here we rely on the obvious equalities 
\be\label{OEQ}
\sum_{k: f(x_k)=j} (\ord_k f+1)=\deg f\ \text{ for any fixed } j\in\{0,1,\infty\}.
\ee

The equality~\eqref{S20.2} together with the result of  Lemma~\ref{INT} allows us to express the integral in the anomaly formula~\eqref{S20Anom} in terms of $\det\Delta_{\pmb\beta}$ and the explicit uniformization data in~\eqref{asympt_phi}. In the remaining part of the proof, we show that as a  result, we arrive at the equality~\eqref{S20STMNT} with  $C_f$ given by~\eqref{C_f}.

It is  easy to see that proceeding as discussed above,  we get the terms 
\be\label{EasyTerms}
\deg f \cdot\log\det\Delta_{\pmb \beta}- \sum_{k}\left(\mathcal C\left(\left(f^*\pmb\beta\right)_k\right)
  -(\ord_k f+1) \mathcal C(\beta_{f(x_k)})\right) -  (\deg f-1)  {\bf C}  
\ee
in the right hand side of~\eqref{S20STMNT}, we omit the details.  

It is considerably harder to show that we also obtain the other terms in the right-hand sides of~\eqref{S20STMNT} and~\eqref{C_f}.	 With this aim in mind, we separately consider  the following four possibilities: 
\begin{multicols}{2}
\begin{itemize}
\item $x_k\neq\infty$ is not a pole of $f$, 
\item $x_k\neq\infty$ is a pole of $f$,
\item  $x_k=\infty$ is a not pole of $f$, and 
\item   $x_k=\infty$ is a pole of $f$. 
\end{itemize}
\end{multicols}
 
If $x_k\neq\infty$ is not a pole of  $f$, then, in addition to the terms listed in~\eqref{EasyTerms}, we get $\frac 1 6$ times
$$
\ba
{
\frac {\left(f^*\pmb\beta\right)_k}{\left(f^*\pmb\beta\right)_k+1} (f^*\phi)_k} -\Bigl( -(f^*\phi)_k \ord_k f+(f^*\pmb\beta)_k \log\left |(\ord_k f+1)c_k\right|\Bigr)\\-{%\color{blue}
(\ord_k f+1) \frac {\beta_{f(x_k)}}{\beta_{f(x_k)}+1}\phi_{f(x_k)}}
\\
=\left(\ord_k f+1-\frac 1 {\ord_k f +1}\right)\frac{\phi_{f(x_k)}}{\beta_{f(x_k)}+1}     - \left( (f^*\pmb\beta)_k+1+\frac 1 { (f^*\pmb\beta)_k+1}  \right)\log (\ord_k f+1)
\\
  +  \left( \frac{\ord_k f}{\ord_k f+1} \log|c_k| + (\ord_k f+2)\log (\ord_k f+1)\right).
\ea
$$
Here the first two terms in the right-hand side contribute to the first and the second sums in the right-hand side of~\eqref{S20STMNT} correspondingly. The last term goes to $C_f$, cf.~\eqref{C_f}.

 If  $x_k \neq  \infty$ is a pole of  $f$, then we get the terms listed in~\eqref{EasyTerms} and also $\frac 1 6$ times
$$
\ba
\frac {\left(f^*\pmb\beta\right)_k}{\left(f^*\pmb\beta\right)_k+1} (f^*\phi)_k -\Bigl (   (f^*\phi)_k(\ord_k f +2)+(f^*\pmb\beta)_k  \log\left |(\ord_k f+1)c_k\right|\Bigr)
\\
 +(\ord_k f+1) \frac {\beta_{f(x_k)}+2} {\beta_{f(x_k)}+1}\phi_{f(x_k)}
\\
=\left(\ord_k f+1-\frac 1 {\ord_k f +1}\right)\frac{\phi_{f(x_k)}}{\beta_{f(x_k)}+1}     - \left( (f^*\pmb\beta)_k+1+\frac 1 { (f^*\pmb\beta)_k+1}  \right)\log (\ord_k f+1)
\\
+  \left( \frac{\ord_k f+2}{\ord_k f+1} \log|c_k|- \ord_k f \log (\ord_k f+1)  \right).
\ea$$
Here again, the first two terms in the right-hand side contribute to the first and the second sums in the right-hand side of~\eqref{S20STMNT}. The last term goes to $C_f$, cf.~\eqref{C_f}.

Similarly, if $x_k =  \infty$  is not a pole of  $ f$,  we get the terms listed in~\eqref{EasyTerms} and $\frac 1 6$ times
$$
\ba
-\frac{\left(f^*\pmb\beta\right)_k+2} {\left(f^*\pmb\beta\right)_k+1}(f^*\phi)_k - \Bigl( -(f^*\phi)_k (\ord_k f +2)+  ((f^*\pmb\beta)_k+2)  (\log\left |(\ord_k f+1)c_k\right|)   \Bigr)
\\
 -(\ord_k f+1) \frac {\beta_{f(x_k)}}{\beta_{f(x_k)}+1}\phi_{f(x_k)}
\\
=\left( \ord_k f+1-\frac 1 {\ord_k f +1}\right)\frac {\beta_{f(x_k)}}{\beta_{f(x_k)}+1}\phi_{f(x_k)} \qquad\qquad \qquad \qquad\qquad \qquad \qquad\qquad 
\\
  - \left( (f^*\pmb\beta)_k+1+\frac 1 { (f^*\pmb\beta)_k+1}  \right)\log (\ord_k f+1) \qquad\qquad \qquad\qquad
\\
 +\left( -   \frac{\ord_k f+2}{\ord_k f+1} \log|c_k|+ \ord_k f\log (\ord_k f+1)\right) .
\ea
$$
Here, as before, the first two terms in the right-hand side are in agreement with~\eqref{S20STMNT}, and the last term goes to $C_f$, cf.~\eqref{C_f}.

Finally, if  $x_k = \infty$ is a pole of   $f$, we get the terms listed in~\eqref{EasyTerms} and $\frac 1 6$ times
$$
\ba
-\frac{\left(f^*\pmb\beta\right)_k+2} {\left(f^*\pmb\beta\right)_k+1}(f^*\phi)_k - \Bigl(    (f^*\phi)_k \ord_k f  +\bigl( (f^*\pmb\beta)_k +2\bigr) \log\left |(\ord_k f+1)c_k\right|\Bigr)\\+(\ord_k f+1) \frac {\beta_{f(x_k)}+2} {\beta_{f(x_k)}+1}\phi_{f(x_k)} 
\\
=\left(\ord_k f+1-\frac 1 {\ord_k f +1}\right)\frac{\phi_{f(x_k)}}{\beta_{f(x_k)}+1}     - \left( (f^*\pmb\beta)_k+1+\frac 1 { (f^*\pmb\beta)_k+1}  \right)\log (\ord_k f+1)
\\
+\left(-   \frac{\ord_k f}{\ord_k f+1} \log|c_k|- (\ord_k f+2)\log (\ord_k f+1)  \right),
\ea
$$
which is in agreement with~\eqref{S20STMNT} and~\eqref{C_f}. This completes the proof of Theorem~\ref{PBdet}.
\end{proof}

\subsection{Euclidean equilateral triangulation}\label{BFET}

By Proposition~\ref{PBdet}  the constant $C_f$ does not depend on the metric $m_{\pmb\beta}$. In this section, we make a particular choice of $m_{\pmb\beta}$ that significantly simplifies the calculation of the constant $C_f$. As we show in the proof of Proposition~\ref{B_C_f} below, it makes sense to consider the  Euclidean equilateral triangulation naturally associated with Belyi function $f$, cf.~\cite{ShVo,VoSh}.

\begin{proposition}\label{B_C_f} The constant $C_f$ in~\eqref{C_f} satisfies~\eqref{C_F} with $A_f$ from~\eqref{Exp_A}; see also Remark~\ref{AvsCn} at the end of this subsection. 
\end{proposition}

\begin{proof} 
 Here we consider the pull-back  by $f$
of the flat singular metric 
$$
m_{\pmb \beta}=c^2 |z|^{-4/3}|z-1|^{-4/3}|dz|^2,\quad \pmb\beta=\left(-\frac 2 3 \right)\cdot 0+\left(-\frac 2 3 \right)\cdot 1+\left(-\frac 2 3 \right)\cdot\infty,
$$
where $c>0$ is a  scaling coefficient that 
 normalizes the area of the metric to one. Note that the surface $(\overline{\Bbb C}_z,m_{\pmb \beta})$ is isometric to two congruent Euclidean equilateral triangles glued together along their sides~\cite{Troyanov,KalvinJGA}. In terms of~\cite{ShVo}, the corresponding bicolored double triangle in Fig.~\ref{TemplateEuc} (with $\beta_0=\beta_1=\beta_\infty=-2/3$) is  a ``butterfly''  that puts the wings together to become $S^2$-like, see also~\cite{VoSh}.

Relying on anomaly formulae we can obtain expressions for the determinants of Laplacians on the base $(\overline{\Bbb C}_z, m_{\pmb \beta})$ and on the ramified covering $(\overline{\Bbb C}_x, f^*m_{\pmb \beta})$. These determinants satisfy the relation~\eqref{S20STMNT} from Theorem~\ref{PBdet}. As we show below, this leads to the equalities~\eqref{C_F} and~\eqref{Exp_A} for $C_f$ and $A_f$ correspondingly, and thus proves the assertion.

The potential $\phi$ of the metric $m_{\pmb\beta}=e^{2\phi}|dz|^2$ has the asymptotics
$$
\phi(x)=-\frac 2 3 \log|z|+\log c+o(1),\quad|z|\to 0;\quad 
\phi(x)=-\frac 2 3 \log|z-1|+\log c +o(1),\quad|z|\to 1;
$$
$$
\phi(x)=-\frac 4 3\log|z|+\log c +o(1),\quad |z|\to \infty. 
$$
This asymptotics is particularly simple, because for the coefficients $\phi_j$  we have $\phi_j=\log c$ (instead of a cumbersome expression with a lot of Gamma functions for $\phi_j$ in~\eqref{asympt_phi}, see~\eqref{eqn_Psi}), and the weights $\beta_j$ of the marked points are $-2/3$.  Moreover, thanks to our special choice of $m_{\pmb\beta}$,  for the Laplacian $\Delta_{\pmb\beta}$ on $(\overline{\Bbb C}_z, m_{\pmb \beta})$ the right hand side of the anomaly formula~\eqref{DetDelta} takes the following particularly  simple form: 
$$
\log {\det\Delta_{\pmb\beta}}=-\frac 4 3 \log c  -3\mathcal C\left(-\frac 2 3\right)+{\bf C}.
$$
This together with the relation~\eqref{S20STMNT} proved in Proposition~\ref{PBdet} gives
\be\label{auxFeb26}
\ba
\log\frac {\det\Delta_{f^*\pmb\beta}}{\deg f } 
 = \deg f\cdot \left(-\frac 4 3\log c +{\bf C}\right)+\frac 1 2  \left(  3\deg f - \sum_{k}\frac 1 {\ord_k f+1}\right) {\log c} 
 \\
   -\frac 1 6 \sum_{k}  \left( \frac{\ord_k f +1} 3 +\frac 3 { \ord_k  f+1}  \right)\log (\ord_k f+1)
 \\
  - \sum_{k}\mathcal C\left(\frac{\ord_k f -2} 3\right)-  (\deg f-1){\bf C}+C_f.
  \ea
\ee

One can think of the surface $(\overline{\Bbb C}_x,f^*m_{\pmb \beta})$ as of the $\deg f$ copies of the flat bicolored double triangle glued along the edges in a way prescribed by $f$.  For the pull-back metric we obtain
\be\label{pb1}
f^*m_{\pmb\beta}=c^2 |f(x)|^{-4/3}|f(x)-1|^{-4/3}|f'(x)|^2 \, |dx|^2. 
\ee
As is well known~\cite{Troyanov}, the flat metric $f^*m_{\pmb\beta}$ can equivalently be written in the standard form 
\be\label{pb2}
f^*m_{\pmb\beta}=c^2 A_f^2  \prod_{k: x_k\neq \infty }|x-x_k|^{\frac 2 3 (\ord_k f-2)}\,|dx|^2.
\ee
Now the representation~\eqref{Exp_A} for the scaling coefficient $A_f$ is an immediate consequence of the equalities~\eqref{pb1} and~\eqref{pb2}.

Clearly, the metric potential $f^*\phi$ of $f^*m_{\pmb\beta}=e^{2f^*\phi}|dx|^2$ obeys the asymptotics
\be\label{a_s_m}
\ba
(f^*\phi)(x)= & \frac{\ord_k f-2} 3 \log|x-x_k|+ \log  cA_f
\\
&  + \sum_{\ell: x_k\neq x_\ell\neq\infty}  \frac {\ord_\ell f-2} 3 \log |x_k-x_\ell|+o(1), \quad x\to x_k\neq \infty,
\ea
\ee

\be\label{asympA}
(f^*\phi)(x)= -\left( \frac {\ord_k f-2}{3} +2\right) \log|x|+ \log cA_f +o(1), \quad x\to x_k=\infty.
\ee  
Therefore, for the Laplacian $\widetilde \Delta_{f^*\pmb\beta}$ induced by the scaled flat  metric
 $$
 (cA_f)^{-2} \cdot f^*m_{\pmb\beta}=\prod_{k: x_k\neq \infty}|x-x_k|^{\frac 2 3 (\ord_j f-2)}\,|dx|^2
 $$  
of total area $\deg f\cdot (cA_f)^{-2}$,  the anomaly formula from~\cite[Prop. 3.3]{KalvinJFA} gives 
\be\label{fp-lA}
\ba
\log\frac {\det\widetilde \Delta_{f^*\pmb\beta}}{\deg f  \cdot (cA_f)^{-2}} 
 =\frac 1 {18}\sum_{ k:x_k\neq \infty}\sum_{\ \ \ell:x_k\neq x_\ell\neq \infty}\frac{(\ord_k f-2)(\ord_\ell f -2)}{\ord_k f+1} \log|x_k-x_\ell|
 \\
 -\sum_{k=1}^n\mathcal C\left(  \frac{\ord_k f-2} 3\right)+{\bf C}.
 \ea
 \ee

  The standard rescaling property of the determinant reads
$$
\log \frac{\det\Delta_{f^*\pmb\beta}}{\deg f }=\log\frac { \det\widetilde \Delta_{f^*\pmb\beta}}{\deg f \cdot (cA_f)^{-2}}-2(\zeta(0)+1)\cdot \log cA_f, 
$$
where
$$
\zeta(0)=-\frac 1 {12}\sum_{k} \left(  \frac{\ord_k f +1} 3 -\frac 3 { \ord_k  f+1}    \right)-1
$$
is the value of the spectral zeta function of $\widetilde \Delta_{f^*\pmb\beta}$ at zero; for details we refer to~\cite[Section 1.2]{KalvinJFA}.

This together with~\eqref{OEQ} allows one to rewrite the equality~\eqref{fp-lA} in the form
$$
\ba
\log\frac {\det\Delta_{f^*\pmb\beta}}{\deg f } 
 =\frac 1 {18}\sum_{ k:x_k\neq \infty}\sum_{\ \ \ell:x_k\neq x_\ell\neq \infty}\frac{(\ord_k f-2)(\ord_\ell f -2)}{\ord_k f+1} \log|x_k-x_\ell|
 \\
+\frac 1 {6}\left(  \deg f-   \sum_{k=1}^n \frac 3 { \ord_k  f+1}    \right)\log cA_f -\sum_{k=1}^n\mathcal C\left(  \frac{\ord_k f-2} 3\right)+{\bf C}.
 \ea
 $$
Substituting this into the left-hand side of~\eqref{auxFeb26},  we finally arrive at~\eqref{C_F}. This completes the proof. 
\end{proof}

\begin{remark}\label{AvsCn}  Let $c_k$ be the coefficient of Taylor or Laurent series from Proposition~\ref{PBdet}. Let $A_f$ be the constant defined in~\eqref{Exp_A}. Then
the asymptotics~\eqref{asympA} together with~\eqref{asympPB1} and~\eqref{1} implies
$$
  A_f=\left\{
\begin{array}{cc}
  (\ord_k f+1)|c_k|^{1/3}, &  \text{ where } k \text{ is such that } x_k=\infty \text{ and } f(\infty)\neq\infty,   \\
   (\ord_k f+1)|c_k|^{-1/3}, &   \text{ where } k \text{ is such that } x_k=\infty \text{ and } f(\infty)=\infty.
\end{array}
\right.
 $$
\end{remark}

\section{Uniformization, Liouville action, and stationary points of the   determinant}\label{Unif}

 Consider a  constant curvature sphere $(\overline{\Bbb C}_x,e^{2\varphi} |dx|^2)$ with conical singularities of order $\beta_k$ located at $x_k\in \overline{\Bbb C}_x$. The parameters $x_1,\dots,x_n$ are called moduli. By the Gauss-Bonnet theorem~\cite{Troyanov}, the (regularized) Gaussian curvature $K$ of   the singular sphere $(\overline{\Bbb C}_x,e^{2\phi} |dx|^2)$ satisfies the equality  $K=2\pi(|\pmb\beta|+2)/S_\varphi$, where $|\pmb\beta|=\sum_k \beta_k$ is the degree of the divisor $\pmb\beta=\sum \beta_k\cdot x_k$, and 
\be\label{Area}
S_\varphi=\int_{\Bbb C} e^{2\varphi}\frac {dx\wedge d\bar x }{-2i}
\ee
is the total surface area of $(\overline{\Bbb C}_x,e^{2\varphi} |dx|^2)$. 

The potential $\varphi$ of the metric $e^{2\varphi} |dx|^2$ is a solution to the Liouville equation
\be\label{LiouvilleEq}
e^{-2\varphi}(-4\partial_x\partial_{\bar x} \varphi)=K, \quad x\in\Bbb C\setminus\{x_1,x_2,\dots,x_n\},
\ee
having the following asymptotics
\be\label{ascoeff_k}
\ba
\varphi(x)=\beta_k\log |x-x_k|+\varphi_k+o(1), \quad x\to x_k\neq\infty,
\\
\varphi(x)=-(\beta_k+2)\log |x|+\varphi_k+o(1), \quad x\to x_k=\infty.
\ea
\ee
The metric potential $\varphi$ and the coefficients $\varphi_k$ in the asymptotics depend on  the  divisor  $\pmb\beta$, i.e. on the moduli $x_k$ and the orders $\beta_k$ of the conical singularities.

 Introduce  the classical stress-energy tensor $T_{\varphi}:=2(\partial_x^2\varphi- (\partial_x \varphi)^2)$ of the Liouville field theory.
 The stress-energy tensor is a meromorphic function on $\overline{\Bbb C}$  satisfying
\be\label{SET}
\ba
T_\varphi(x)=& \sum_{k: x_k\neq \infty}\left(\frac {{ s}_k}{2(x-x_k)^2}  +\frac { h_k}{x-x_k} \right),\\
 T_\varphi(x)=& \frac {{s}_k}{2x^2}+\frac { h_k}{x^3}+O(x^{-4})\text{ as } x\to x_k=\infty.
\ea
\ee
 Here $ s_k=-\beta_k(2+\beta_k)$ are the weights of the second order poles, and $h_k$ are the famous accessory parameters, e.g.~\cite{He2,Kra,T-Z} and references therein.  Note that the meromorphic quadratic differential $T_\varphi dx^2$ is a uniformizing projective connection compatible with the divisor $\pmb\beta$~\cite{Troyanov,TroyanovSp}. 
 
Recall that one of the approaches to the uniformization consists of finding appropriate values of the accessory parameters $ h_k$, and two appropriately normalized linearly independent solutions  $u_1$ and $u_2$ to the Fuchsian differential equation
 $$
 \partial_x^2 u+\frac 1 2 T_{\varphi} u=0.
 $$
 Then the metric potential $\varphi$ can be found in the form
 $$
 \varphi=\log 2 +\log |\partial_xw|-\log (1+K|w|^2),
 $$
where  $w=u_1/u_2$ is an analytic in $\Bbb C\setminus\{x_1,\dots,x_n\}$ function called the developing map. It satisfies the Schwarzian differential equation $\{w,x\}=T_\varphi(x)$,  where $\{w,x\}=\frac{2w' w'''-3w''^2}{2w'^2}$ is the Schwarzian derivative. However, the accessory parameters can be determined in some special cases only, and, in general, they remain elusive.

It turns out that in the geometric setting of this paper, the accessory parameters can be found explicitly in terms of the Belyi function $f$ and the orders $\beta_0$, $\beta_1$, and $\beta_\infty$ of the conical singularities of the metric $m_{\pmb\beta}=e^{2\phi}|dz|^2$. 
Indeed, consider the constant curvature unit area singular  sphere $(\overline{\Bbb C}_z, m_{\pmb\beta})$, see Sec.~\ref{PrelimTriangulation}. 
For the corresponding stress-energy tensor we have
$$
T_\phi(z)=\frac {\mathfrak s_0}{2z^2}+\frac { \mathfrak h_0}{z}+\frac {\mathfrak s_1}{2(z-1)^2}+\frac {\mathfrak h_1}{z-1},\quad 
T_\phi(z)=\frac {\mathfrak s_\infty}{2z^2}+\frac{\mathfrak h_\infty}{z^3}+O(z^{-4})\quad\text{as}\quad z\to\infty,
$$
where 
\be\label{weight_s}
\mathfrak s_k=-\beta_k(2+\beta_k), \quad k\in\{0,1,\infty\}.
\ee
The accessory parameters
\be\label{S_AP}
\mathfrak h_0=-\mathfrak h_1=\frac{{\mathfrak s}_0+{\mathfrak s}_1-{\mathfrak s}_\infty}{2}, \quad
\mathfrak h_\infty=\frac{{\mathfrak s}_1+{\mathfrak s}_\infty-{\mathfrak s}_0}{2}
\ee
 were first found by Schwarz.  The stress-energy tensors
satisfy the relation
\be\label{relation1727}
T_{f^*\phi}=(T_\phi\circ f) (f')^2+\{f,x\}.
\ee
As a consequence, we obtain the following simple result:
\begin{lemma}[Accessory parameters]
\label{AccPar} Let $f: \overline {\Bbb C}_x \to \overline {\Bbb C}_z$ be a Belyi function unramified outside of the set $\{0,1,\infty\}$ and  such that $f(\infty)\in\{0,1,\infty\}$. Then the stress-energy tensor $T_{f^*\phi}$ of the pull back metric $f^*m_\beta=e^{f^*\phi}|dx|^2$ of $m_\beta$  by $f$  
satisfies the relations~\eqref{SET}, where $x_1,\dots,x_n$ with $n=\deg f +2$ are the preimages of the points $\{0,1,\infty\}\subset\overline{\Bbb C}_z$ under $f$, the weights of the second order poles in~\eqref{SET} are given by the equalities
$$
s_k=-(f^*\pmb\beta)_k ((f^*\pmb\beta)_k +2)
$$
with orders $(f^*\pmb\beta)_k$ from~\eqref{orders*}, and the accessory parameters $h_k$ can be found  as follows:

1. If  $x_k$ is not a pole of $f$, then the accessory parameter $h_k$ satisfies 
\be\label{AP1}
h_k=
\left\{
\begin{array}{cc}
   \frac {d_k}{c_k}{\mathfrak s}_{f(x_k)}+  c_k  \mathfrak h_{f(x_k)}&    \text{ for }\quad  \ord_k f =0, \\
-    \frac {d_k}{c_k}  \frac {   (f^*\beta)_k ((f^*\beta)_k+2) }{ \ord_k f+1}    &    \text{ for }\quad \ord_k f >0,
\end{array}
\right.
\ee
where $\mathfrak s_k$ and $\mathfrak h_k$  are the same as in~\eqref{weight_s} and~\eqref{S_AP}, and   the coefficients $c_k$ and $d_k$ are those from the first or the second  expansion
$$
\ba
f(x)-f(x_k) & =c_k x^{-\ord_k f-1}\left(1+\frac{d_k}{c_k} \frac 1 x + O(x^{-2})\right), \quad x\to x_k=\infty,
\\
f(x)-f(x_k) & =c_k(x-x_k)^{\ord_k f+1}\left(1+\frac {d_k}{c_k}(x-x_k)+O\left(   (x-x_k)^{2}\right)\right),\quad x\to x_k\neq \infty,
\ea
$$
depending on whether $x_k=\infty$ or $x_k\neq\infty$,

2.    If  $x_k$ is a pole of $f$, then the accessory parameter $h_k$ satisfies  
\be\label{AP2}
h_k=
\left\{
\begin{array}{cc}
  -\frac {d_k}{c_k}{\mathfrak s}_{\infty}+\frac{1}{ c_k} \mathfrak h_{\infty} &    \text{ for }\quad  \ord_k f =0, \\
\frac {d_k}{c_k}  \frac {   (f^*\beta)_k ((f^*\beta)_k+2) }{ \ord_k f+1}   &    \text{ for }\quad \ord_k f >0,
\end{array}
\right.
\ee
where  $ \mathfrak s_\infty$ and $\mathfrak h_\infty$ are the same  in~\eqref{weight_s} and~\eqref{S_AP}, and  the coefficients $c_k$ and $d_k$ are those from the first or the second expansion
$$
\ba
f(x)& =c_k x^{\ord_k f+1}\left(1+\frac {d_k}{c_k} \frac 1 x+ O(x^{-2})\right), \quad x\to x_k=\infty,
\\
f(x)&={c_k}{(x-x_k)^{-\ord_k f-1}}\left(1+\frac  {d_k}{c_k} (x-x_k)+O((x-x_k)^{2})\right),\quad x\to x_k\neq \infty,
\ea
$$
depending on whether $x_k=\infty$ or $x_k\neq\infty$.  
\end{lemma}

\begin{proof} If $f(x_k)\neq \infty$ and $x_k= \infty$, then
$$
f(x)-f(x_k)=c_k x^{-\ord_k f-1} \left(1+\frac {d_k}{c_k} \frac 1 x + O(x^{ -2})\right), \quad x\to x_k=\infty,
$$
and the asymptotics can be differentiated. As a consequence, for the contributions into
$$
(T_\phi\circ f) (f')^2 (x)=\left(\frac {-\beta_{f(x_k)}(2+\beta_{f(x_k)})}{2(f(x)-f(x_k))^2}+\frac {\mathfrak h_{f(x_k)}}{f(x)-f(x_k)}+O(1)\right)(f'(x))^2
$$
we obtain
$$
\frac {(f'(x))^2}{(f(x)-f(x_k))^2}=\frac {(\ord_k f+1)^2} {x^{2}}+2\frac {d_k}{c_k} \frac {\ord_k f+1}  {x^3} +O(x^{-4}),
$$
$$
\frac {h_{f(x_k)}}{f(x)-f(x_k)}(f'(x))^2=h_{f(x_k)}{c_k} (\ord_k f+1)^2x^{-\ord_k f -3}+O(x^{-4}).
$$
Besides, for the Schwarzian derivative, we get
$$
\{f,x\}= - \frac 1 2 (\ord_k f) (\ord_k f+2)\left(    \frac {1}{x^{2} }   + \frac {2d_k}{c_k( \ord_k f+1)}   \frac {1} {x^3} +O(x^{-4}) \right),\quad x\to x_k=\infty.
$$
These together with~\eqref{relation1727} imply
$$
T_{f^*\phi}(x)=-\frac { (f^*\beta)_k( (f^*\beta)_k+2)}{2x^2}+\frac {h_k}{x^3}+O(x^{-4})\quad\text{as}\quad x\to x_k=\infty,
$$
 where the accessory parameter $h_k$ satisfies~\eqref{AP1}.

Similarly, if $f(x_k)=\infty$ and $x_k= \infty$, then starting with the asymptotics
$$
f(x)=c_k x^{\ord_k f+1}\left(1+\frac {d_k}{c_k} \frac 1 x+ O(x^{-2})\right), \quad x\to x_k=\infty,
$$
we arrive at 
$$
T_{f^*\phi}(x)=-\frac { (f^*\beta)_k( (f^*\beta)_k+2)}{2x^2}+\frac {h_k}{x^3}+O(x^{-4})\quad\text{as}\quad x\to x_k=\infty,
$$
 where the accessory parameter $h_k$ satisfies~\eqref{AP2}. 

The cases $f(x_k)=\infty$ with $x_k\neq \infty$, and $f(x_k)\neq \infty$ with $x_k\neq \infty$ are  similar, we omit the details. 
\end{proof}

 \begin{remark} In particular, by using Lemma~\ref{AccPar} with $\pmb \beta=\beta_0\cdot 0+\beta_1\cdot 1+\left(-\frac 2 3 \right)\cdot \infty$ and $f(x)=x^3$, we recover the values of the accessory parameters found in~\cite[Sec. 4.13]{Kra} for a constant curvature Hyperbolic or Euclidean metric with four conical singularities located at the vertices of a regular tetrahedron. 
  \end{remark}

Introduce the Liouville action 
\be\label{LA}
\mathcal S_{\pmb \beta}[\varphi]=2\pi(|\pmb\beta|+2) \left (\frac 1 {S_\varphi} \int_{\Bbb C} \varphi e^{2\varphi}\frac  {dx\wedge d\bar x}{-2 i}   -1 \right) +2\pi\sum_{k}\beta_k\varphi_k+4\pi \varphi_k|_{k: x_k=\infty},
\ee
where $S_\varphi$ is the total area of the singular sphere $(\overline{\Bbb C}_x, e^{2\varphi}|dx|^2)$, and  $\varphi_k$ are the coefficients in the asymptotics~\eqref{ascoeff_k}. As is demonstrated in~\cite{KalvinLast}, this new definition of the Liouville action is in agreement, for instance, with that in~\cite {CMS,Z-Z,T-Z}. It is not hard to show that the Liouville equation~\eqref{LiouvilleEq} is the Euler-Lagrange equation for the Liouville action functional $\psi\mapsto \mathcal S_{\pmb \beta}[\psi]$.

\begin{remark}
In the geometric setting of this paper we have $\varphi=f^*\phi$, and the anomaly formula from Lemma~\ref{ANOMALY} can equivalently be written as follows:
$$
\log \frac {\det\Delta_{f^*\pmb\beta}}{\deg f}   =\frac {|f^*\pmb\beta|+2}{6} -\frac 1 {12 \pi}\left(\mathcal S_{f^*\pmb\beta} [f^*\phi]-\pi\log \mathcal H_{f^*\pmb\beta} [f^*\phi]\right)-\sum_{k} \mathcal C((f^*\beta)_k) +{\bf C}.
$$
Here  the functional $\mathcal H_{f^*\pmb\beta} [f^*\phi]$ is defined explicitly via the equality
$$
\mathcal H_{f^*\pmb\beta}   [f^*\phi]= \exp\left(2\sum_k \frac {(f^*\beta)_k((f^*\beta)_k+2)}{(f^*\beta)_k+1} (f^*\phi)_k\right)
$$ 
with $(f^*\beta)_k$ from~\eqref{orders*} and $(f^*\phi)_k$ from~\eqref{1}.
Thus, as a consequence of  the explicit expression for the determinant of Laplacian in Theorem~\ref{main} and the anomaly formula in Lemma~\ref{ANOMALY}, one also obtains an explicit expression for the Liouville action~$\mathcal S_{f^*\pmb\beta} [f^*\phi]$, cf.~\cite{P-T-T,ParkTeo,T-T}. It would be interesting to check if this result can be reproduced by using conformal blocks~\cite{Z-Z}. 
\end{remark}

In the remaining part of this section for simplicity, we assume that the orders $\beta_k$ of the conical singularities meet the condition  $|\pmb\beta|\leq -2$, i.e. we exclude form consideration the spherical metrics. This allows us to differentiate the hyperbolic metric potential and the corresponding Liouville action with respect to $x_k$ and $\beta_k$ relying on known regularity results, e.g.~\cite{KimWilkin,Kra,SchuTrap1,SchuTrap2,T-Z}. In the Euclidean case, we have $|\pmb\beta|=-2$, and the metrics can be written explicitly~\cite{TroyanovSSC}, which immediately justifies the differentiation. Let us also note that there are good grounds to believe~\cite{Judge,Judge2,KalvinLast,KimWilkin,TroyanovSp} that the potential $\varphi$ of a constant curvature metric is necessarily a real-analytic function of the orders of conical singularities on the existence and uniqueness set $$\{\beta_k\in(0,1): \beta_k-|\pmb\beta|/2>0, k=1,\dots,n\},$$ and the results below remain valid on that set.

 Next, we show that the Liouville action $\mathcal S_{\pmb \beta}[\varphi]$ generates the accessory parameters $h_k$ as their common antiderivative.  
 \begin{lemma}[After P. Zograf and  L. Takhtajan] Assume that  $\beta_k\in(0,1)$,  $k=1,\dots, n$, and  $|\pmb\beta|\leq -2$. Let  $\varphi$ be a (unique) solution to the Liouville equation~\eqref{LiouvilleEq}   satisfying the area condition~\eqref{Area} with some fixed $S_\varphi>0$, and having the asymptotics~\eqref{ascoeff_k}.  Then the Liouville action~\eqref{LA} meets the equalities
\be\label{TZeqn}
-\frac 1 {2\pi}{\partial_{x_k} \mathcal S_{\pmb \beta}[\varphi]}=h_k, \quad k=1,\dots, n, 
\ee
where $x_1,x_2,\dots, x_n$ are the moduli, and  $h_1,h_2,\dots,h_n$ are the accessory parameters.

Note that in the geometric setting of this paper, we have $\varphi=f^*\phi$. Thus the moduli $x_k$, $k=1,\dots,\deg f +2 $,  are the preimages of the points $\{0,1,\infty\}$ under $f$, and the accessory parameters  $h_k$  are those found in Lemma~\ref{AccPar}.
\end{lemma}
\begin{proof} As is shown in~\cite{KalvinLast}, the expression in the right hand side of~\eqref{LA} is an equivalent regularization of the Liouville action introduced in~\cite{CMS,Z-Z,T-Z}. Hence, in the case of $|\pmb\beta|<-2$ and  $K=-1$  the assertion of the lemma is just a reformulation of the result proven in~\cite{CMS,T-Z}, see also~\cite{T-Z1} for the first proof of Polyakov's conjecture~\eqref{TZeqn}. Next we show that in the (hyperbolic) case $|\pmb\beta|<-2$ the assertion remains valid for any $S_\varphi>0$ and $K=2\pi(|\pmb\beta|+2)/S_\varphi<0$.

Consider a (unique) metric potential $\varphi$ such that $|\pmb\beta|<-2$ and  $K=-1$. Clearly,  $S_\varphi=-2\pi(|\pmb\beta|+2)$, and for any $C>0$ we have the following transformation laws for the surface area, the Liouville action, and the stress-energy tensor:
$$
S_{\varphi+\log C}=C^2 S_\varphi, \quad \mathcal S_{\pmb \beta}[\varphi+\log C]=\mathcal S_{\pmb \beta}[\varphi]  +4\pi(|\pmb\beta|+2)\log C,\quad T_{\varphi+\log C}=T_\varphi.
$$
This implies that the rescaling $\varphi\mapsto \varphi+\log C$ multiplies the total area by $C^2$, but does not affect the equalities~\eqref{TZeqn}.  
Thus, in the case $|\pmb\beta|<-2$, the assertion of lemma is valid for any fixed $S_\varphi>0$.

In the case $|\pmb\beta|=-2$  the integral term in~\eqref{LA} disappears and the metric $e^{2\varphi}|dx|^2$ is flat. As is known~\cite{TroyanovSSC}, up to a rescaling  $\varphi\mapsto \varphi+\log C$,  the potential $\varphi$ can be written explicitly in the form
\be\label{ECSC}
\varphi(\pmb\beta)=\sum_{k: x_k\neq \infty}\beta_k\log|x-x_k|,\quad |\pmb\beta|=-2.
\ee
As a result, the equality~\eqref{TZeqn} follows from a simple direct computation. 
\end{proof}

 It may also be possible to prove Polyakov's conjecture~\eqref{TZeqn} for the spherical case $|\pmb\beta|>-2$  along the lines of~\cite{CMS,T-Z}, however this goes out of the scope of this paper.

\begin{lemma}[After A. Zamolodchikov and Al. Zamolodchikov]\label{Z^2} Assume that $\beta_k\in (-1,0)$ and $|\pmb\beta|\leq -2$.  Let $\varphi$ stand for a (unique) solution to the Liouville equation~\eqref{LiouvilleEq}   satisfying the area condition~\eqref{Area} with a fixed $S_\varphi>0$, and having the asymptotics~\eqref{ascoeff_k}. Then for any fixed configuration $x_1,\dots,x_n$  the Liouville action~\eqref{LA} satisfies
\be\label{ZZeqn}
-\frac 1 {2\pi} {\partial_{\beta_k} \mathcal S_{\pmb \beta}[\varphi]}=1-2 \varphi_k,
\ee
where $\varphi_k$ is the coefficient  in the corresponding asymptotics~\eqref{ascoeff_k}.

In the geometric setting of this paper, we have $\varphi=f^*\phi$ and  $\varphi_k=(f^*\phi)_k$, see~\eqref{1}.
\end{lemma}
\begin{proof} In the case $|\pmb\beta|<-2$,  the proof  essentially repeats the one in~\cite[Proof of Lemma 3.1]{KalvinLast}, where   the differentiation with respect to $\beta_k$ is now justified by the results~\cite%[Prop. 2.5]
{KimWilkin,SchuTrap1,SchuTrap2} on the regularity of $\beta_k\mapsto \varphi(\pmb\beta)$ for the hyperbolic metric $e^{2\varphi}|dx|^2$ on the Riemann sphere. Indeed, one  need only notice that the index $j=k$ in~\cite[Proof of Lemma 3.1]{KalvinLast} now runs from $1$ to $n$, and the region $\Bbb C_R$ is defined as follows:
$$
\Bbb C_R:=\{x\in\Bbb C: |x|\leq R,\  |x-x_k|\geq 1/R,\ k=1,\dots n \}.
$$

In the case $|\pmb\beta|=-2$  the integral term in~\eqref{LA} disappears, and the equality~\eqref{ZZeqn} can be verified by a direct computation, cf.~\eqref{ECSC}. We omit the details.
\end{proof}

Our choice of examples in Section~\ref{ExAppl} is partially motivated by the following result:
\begin{theorem} \label{ExtDet} The (hyperbolic or flat) surfaces of five Platonic solids and the regular constant curvature dihedra are critical points of the spectral determinant on the conical  metrics of fixed area and fixed Gaussian curvature.

More precisely: Consider the divisors $\pmb\beta=\sum_{k} \beta_k\cdot x_k$  of degree $|\pmb\beta|\leq -2$ with distinct marked points $x_1,\dots,x_n$ and weights $\beta_k\in(-1,0)$. 
Then for any fixed $S>0$ and any divisor $\pmb\beta$ there exists a unique metric  $e^{2\varphi}|dx|^2$  on $\overline{\Bbb C}$  of  total area $S$, Gaussian curvature $K=2\pi(|\pmb\beta|+2)/S$, and  representing the divisor $\pmb\beta$. 
Consider the spectral determinant  $\det\Delta_{\pmb\beta}$ of the surface $(\overline{\Bbb C}_x,e^{2\varphi}|dx|^2)$ as a function on the configuration space
$$
\mathcal Z_n(S,K)=\left\{\pmb\beta =\sum_{k\leq n} \beta_k\cdot x_k:  x_j\neq x_k\in \overline{\Bbb C}\text{ for } j\neq k, \beta_k\in(-1,0), 2\pi (|\pmb\beta|+2)=S K \right\}
$$
with some fixed values  $S>0$, $K\leq 0$, and $n\geq 3$.
\begin{enumerate}
\item If  $\pmb\beta_0\in \mathcal Z_n(S,K) $ is a divisor such that the corresponding surface $(\overline{\Bbb C}_x,e^{2\varphi}|dx|^2)$ is isometric to the one of a Platonic solid, then
$\pmb\beta_0$ is a stationary point of the function 
\be\label{ext_det}
\mathcal Z_n(S,K)\ni\pmb\beta\mapsto \det\Delta_{\pmb\beta},
\ee
where $n$ is the number of vertices of the Platonic solid. 
\item  If  $\pmb\beta_0\in \mathcal Z_n(S,K) $ is a divisor such that the corresponding surface $(\overline{\Bbb C}_x,e^{2\varphi}|dx|^2)$ is isometric to the regular dihedron with $n$ vertices, then
$\pmb\beta_0$ is a stationary point of the function~\eqref{ext_det} with the corresponding value of $n$. 
\end{enumerate}
\end{theorem}
\begin{proof}  
Consider the potential $\varphi(x;\beta_1,\beta_2,\beta_3,\beta_4)$  of a (unique) unit area constant curvature metric $e^{2\varphi}|d x|^2$ representing the divisor
$$
\pmb\beta=\beta_1\cdot 0 +\beta_2\cdot (-1)+\beta_3\cdot e^{i\frac {\pi} 3}+\beta_4\cdot e^{-i\frac{\pi} 3}. 
$$
Recall that the Gauss-Bonnet theorem~\cite{Troyanov} implies that the (regularized) Gaussian curvature of  the surface $(\overline{\Bbb C}_x, e^{2\varphi}|dx|^2)$ equals $2\pi(|\pmb\beta|+2)$. The four marked points in the divisor $\pmb\beta$ are in an equi-anharmonic position. In particular, if the orders of the conical singularities satisfy $\beta_k=|\pmb\beta|/4$, then the surface $(\overline{\Bbb C}_x, e^{2\varphi}|dx|^2)$ is isometric to the one of a unit area regular tetrahedron of Gaussian curvature $2\pi(|\pmb\beta|+2)\leq 0$.  

Notice that  $\varphi(\bar x;\beta_1,\beta_2,\beta_3,\beta_4)$ is the potential of a (unique) unit area Gaussian curvature $2\pi(|\pmb\beta|+2)$ metric 
representing the divisor 
$$
\pmb\beta=\beta_1\cdot 0 +\beta_2\cdot (-1)+\beta_4\cdot e^{i\frac {\pi} 3}+\beta_3\cdot e^{-i\frac{\pi} 3}.
$$
Similarly, the potential $\varphi(e^{i\frac{2\pi}3} x;\beta_1,\beta_2,\beta_3,\beta_4)$ 
corresponds to  the divisor 
$$
\pmb\beta=\beta_1\cdot 0 +\beta_4\cdot (-1)+\beta_2\cdot e^{i\frac{\pi} 3}+\beta_3\cdot e^{-i\frac{\pi} 3};
$$
$\varphi(e^{-i\frac{2\pi}3} x;\beta_1,\beta_2,\beta_3,\beta_4)$ 
corresponds to  the divisor 
$$
\pmb\beta=\beta_1\cdot 0+\beta_3\cdot (-1) +\beta_4\cdot e^{i\frac \pi 3 }+\beta_2\cdot e^{-i\frac{\pi} 3};
$$
and $\varphi(\frac {x+1}{2x-1};\beta_1,\beta_2,\beta_3,\beta_4)+\log3-2\log|2x-1|$
corresponds to  the divisor 
$$
\pmb\beta=\beta_2\cdot 0+\beta_1\cdot (-1) +\beta_4\cdot e^{i\frac{\pi} 3} +\beta_3\cdot e^{-i\frac {\pi} 3}.
$$

As a consequence of these symmetries, we have
$$
\varphi(x;\beta_1,\beta_2,\beta_3,\beta_4)=\varphi(\bar x;\beta_1,\beta_2,\beta_4,\beta_3)   =\varphi(e^{i\frac{2\pi}3} x;\beta_1,\beta_3,\beta_4,\beta_2)=\varphi(e^{-i\frac{2\pi}3} x;\beta_1,\beta_4,\beta_2,\beta_3)
$$
$$
=\varphi\left(\frac {x+1}{2x-1};\beta_2,\beta_1,\beta_4,\beta_3\right)+\log3-2\log|2x-1|.
$$
For the coefficients $\varphi_k=\varphi_k(\beta_1,\beta_2,\beta_3,\beta_4)$ in  the asymptotics~\eqref{ascoeff_k}
 the latter equalities imply
\be\label{STAR}
\ba
(\varphi_1-\varphi_\ell)|_{\beta_k=|\pmb\beta|/4}=(|\pmb\beta|/4+1)\log 3,\quad \ell=2,3,4,
\\
\sum_{j=1}^4 ( \partial_{\beta_1}\varphi_j -\partial_{\beta_\ell}\varphi_j )|_{\beta_k=|\pmb\beta|/4}=\log3,\quad\ell=2,3,4.
\ea
\ee

 Denote the Friedrichs Laplacian on $(\overline{\Bbb C}_x,e^{2\varphi}|dx|^2)$ by $\Delta_{\pmb\beta}$. As is proven in~\cite[Sec. 2]{KalvinLast}, the spectral determinant $\det \Delta_{\pmb\beta}$ satisfies the anomaly formula
 \be\label{AnFla}
 \log \det\Delta_{\pmb\beta}=-\frac{|\pmb\beta|+1} 6 -\frac 1 {12\pi}\left(\mathcal S_{\pmb\beta}[\varphi]-\pi\log \mathcal H_{\pmb\beta}[\varphi]\right) -\sum_{k=1}^4\mathcal C(\beta_k)+\bf C,
 \ee
where $\mathcal S_{\pmb\beta}[\varphi]$ is the Liouville  action~\eqref{LA} and 
 \be\label{F_H}
\mathcal H_{\pmb\beta}   [\varphi]= \exp\left\{2\sum_{k=1}^4 \left(\beta_k+1-\frac 1 {\beta_k+1} \right)\varphi_k\right\}.
\ee

Since $|\pmb\beta|$ is fixed, we can set, for example, $\beta_1=|\pmb\beta|-\sum_{k>1}\beta_k$, and consider the determinant $\det\Delta_{\pmb\beta}$ as a function of $(\beta_2,\beta_3,\beta_4)$.  Then, thanks to the anomaly formula~\eqref{AnFla}, Lemma~\ref{Z^2}, and the equality~\eqref{F_H}, we have
$$
\ba
\partial_{\beta_\ell}  \left( \log \det\Delta_{\pmb\beta}|_{\beta_1=|\pmb\beta|-\sum_{k>1}\beta_k}\right)|_{\beta_k=\frac{|\pmb\beta|}4}& =\frac 1 3 (\varphi_1-\varphi_\ell)|_{\beta_k=\frac{|\pmb\beta|}4} 
\\ & - \frac 1 6 \left( 1+\left(\frac{|\pmb\beta|}4+1\right)^{-2}   \right)(\varphi_1-\varphi_\ell)|_{\beta_k=\frac{|\pmb\beta|}4}
\\
& - \frac 1 6 \left(\frac {|\pmb\beta|}4+1-\frac 1 {\frac{|\pmb\beta|}4+1}\right) \sum_{j=1}^4 ( \partial_{\beta_1}\varphi_j -\partial_{\beta_\ell}\varphi_j ) |_{\beta_k=\frac {|\pmb\beta|}4}.
\ea
$$
Here the right-hand side is equal to zero because of~\eqref{STAR}.

 Now we are in a  position to study the determinant under a small perturbation of the coordinate of a vertex. Let us consider the potential $\varphi(x)$  of a (unique) unit area constant curvature metric $e^{2\varphi}|d x|^2$ representing the divisor
$$
\pmb\beta=\frac{|\pmb\beta|} 4 \cdot h +\frac{|\pmb\beta|} 4\cdot (-1)+\frac{|\pmb\beta|} 4\cdot e^{i\frac {\pi} 3}+\frac{|\pmb\beta|} 4\cdot e^{-i\frac{\pi} 3}. 
$$
Here $h$ is a small complex number. In the case $h= 0$ the surface $(\overline{\Bbb C}_x, e^{2\varphi}|dx|^2)$ is isometric to the one of a unit area constant curvature regular tetrahedron.

Consider, for example, the rotation $x\mapsto e^{i\frac{2\pi}3} x$. Notice that  $ \chi(x)  :=\varphi( e^{i\frac{2\pi}3} x)$ is the potential of a (unique) unit area Gaussian curvature $2\pi(|\pmb\beta|+2)$ metric 
representing the divisor 
$$
{\pmb\gamma}=\frac{|\pmb\beta|} 4 \cdot (e^{-i\frac{2\pi}3}h) +\frac{|\pmb\beta|} 4\cdot (-1)+\frac{|\pmb\beta|} 4\cdot e^{i\frac {\pi} 3}+\frac{|\pmb\beta|} 4\cdot e^{-i\frac{\pi} 3}.
$$
The surfaces $({\overline {\Bbb C}}_x, e^{2\varphi}|dx|^2)$ and $({\overline {\Bbb C}}_x, e^{2 \chi}|dx|^2)$ are isometric, the isometry is given by the rotation. 
As a consequence,  $\det \Delta_{\pmb\beta}=\det \Delta_{\pmb\gamma}$. Equating the directional derivative of $\det \Delta_{\pmb\beta}$ along $h$ with the one along $e^{i\frac{2\pi}3}h$, we immediately conclude that 
  $$
  \partial_{\Re h}\det\Delta_{\pmb\beta}=\partial_{\Im h} \det\Delta_{\pmb\beta}=0.
  $$ 
  
It remains to note that the determinants of the Laplacians $\Delta_{\pmb\beta}$  and $\Delta_{\pmb\beta}^{S_\varphi}=\frac 1 {S_\varphi} \Delta_{\pmb\beta}$ satisfy the standard rescaling property
$$
\log\det\Delta_{\pmb\beta}^{S_\varphi}=\log \det \Delta_{\pmb\beta} +\zeta_{\pmb\beta} (0) \log S_{\varphi},
$$
  where the value $\zeta_{\pmb\beta} (0) $ of the spectral zeta function at zero~\cite{KalvinJFA} does not depend on the moduli $x_1,\dots,x_4$ and satisfies
  $$
  \zeta_{\pmb\beta} (0) =\frac{|\pmb\beta|+2} 6-\frac 1 {12}\sum_{k}\left(\beta_k+1-\frac{1}{\beta_k+1}\right)-1, \quad \partial_{\beta_\ell} \left(\zeta_{\pmb\beta}(0)|_{\beta_1=|\pmb\beta|-\sum_{k>1}\beta_k}\right)|_{\beta_k=\frac{|\pmb\beta|}4}=0.
  $$
 Due to the invariance of the spectral determinant under the M\"obius transformations, this completes the proof of the first assertion.

For the octahedron, cube, dodecahedron, icosahedron, and dihedra there are more symmetries to consider, but the idea and the steps of the proof remain exactly the same.   We omit the details. The case of constant curvature (flat, spherical, or hyperbolic) metrics with three conical singularities is studied in~\cite{KalvinLast}.
\end{proof}

As is well-known, starting from four punctures on the $2$-sphere, explicit construction of the general uniformization map is an open long-standing problem. In this paper, we rely on the uniformization via  Belyi functions. There is another straightforward special case that deserves to be mentioned.

\begin{remark} {\rm In the case of a divisor
$$
\pmb\lambda=\left(-1/ 2 \right)\cdot 0+\left(-1/ 2 \right)\cdot 1 +\left(-1/ 2 \right)\cdot\lambda +\left(-1/ 2 \right)\cdot\infty, \quad \lambda\in\Bbb C,
$$
the corresponding constant curvature unit area metric $m_{\pmb\lambda}$ with three or four conical singularities of angle $\pi$ can be written explicitly, e.g.~\cite{BE,Kra,TroyanovSSC}.

 Recall that by using a suitable M\"obius transformation we can always normalize the marked points so that any three of them are at $0,1,\infty$. As we permute the marked  points $0,1,\lambda,\infty$ by M\"obius transformations so that three of them are still $0,1,\infty$, the fourth point is one of the following six:
\begin{equation}\label{37.1}
\lambda,\quad \frac 1 \lambda,\quad 1-\lambda,\quad \frac 1 {1-\lambda},\quad \frac \lambda{\lambda-1},\quad \frac{\lambda-1}{\lambda}.
\end{equation}
In general, these six points are distinct. The exceptions are the following three cases:
\begin{itemize}
\item $\lambda=0$ or $\lambda=1$ or $\lambda=\infty$. In this case, the set~\eqref{37.1}  contains only three distinct numbers: $0,1,\infty$. This case is studied in~\cite{KalvinLast}, we do not discuss it here. 

\item Harmonic position of four points:  $\lambda=-1$ or $\lambda=1/2$ or $\lambda=2$. The set~\eqref{37.1}  contains only three distinct numbers: $-1,1/2,2$. The surface $(\overline{\Bbb C}_x, m_{\pmb\lambda})$ is isometric to a unit area flat regular dihedron with four conical singularities of angle $\pi$. 
\item Equi-anharmonic position of four points:  $\lambda=\frac{1+i\sqrt{3}}{2}$ or $\lambda=\frac{1-i\sqrt{3}}{2}$. The set~\eqref{37.1} contains only the numbers $\frac{1\pm i\sqrt{3}}{2}$. The surface $(\overline{\Bbb C}_x, m_{\pmb\lambda})$ is isometric to the surface of a unit area regular Euclidean tetrahedron. 
\end{itemize}
In general, for $\lambda\neq 0,1,\infty$, the metric is flat, and  we have
$$
m_{\pmb\lambda}=c^2_{\lambda}|x|^{-1}|x-1|^{-1}|x-\lambda|^{-1}   |dx|^2,
$$
where $c^2_{\lambda}$ is a scaling factor that guarantees that  the surface $(\overline{\Bbb C}_x,m_{\pmb\lambda})$ is of unit area, see~\cite{TroyanovSSC}.  For the spectral  determinant of the Friedrichs Laplacian $\Delta_{\lambda}$ on the flat surface $(\overline{\Bbb C}_x,m_{\pmb\lambda})$  the anomaly formula~\eqref{AnFla} gives
$$
\log\det\Delta_\lambda=- \log c_{\lambda} +\frac 1 6 (\log |\lambda|+\log|\lambda-1|)
-4\mathcal C(-1/2)+{\bf C},
$$
where  ${\bf  C}$ is the same as in~\eqref{bfC}. Besides, thanks to~\cite[Appendix]{KalvinJFA}, we have
\be\label{C(-1/2)}
\mathcal C\left(-\frac 1 2\right)=-\zeta_R'(-1)-\frac 1 6 \log 2 +\frac 1 {24}.
\ee
As is well-known, e.g.~\cite[Sec. 2.9]{Clem},  the scaling factor $c^2_\lambda$ satisfies
$$
c^{-2}_{\lambda}=\int_{\Bbb C}    \frac{dx\wedge d\bar x}{-2i |x|\cdot|x-1|\cdot|x-\lambda|} =8|k|\left(K' \overline{K} +\overline{K'}K\right), \quad \lambda=\frac{(k+1)^2}{4k},
$$ 
where $K=K(k)$  is the complete elliptic integral of the first kind, and  $K'=K(\sqrt{1-k^2})$. 

In total, in terms of  $\tau=i K'/K$, we get
$$
K=\frac \pi 2 \vartheta_3^2(0|\tau),\quad K'=-i\tau K,\quad k=\frac{\vartheta_2^2(0|\tau)}{\vartheta_3^2(0|\tau)},
$$
$$
\det\Delta_\lambda= \frac {2^{2/3}} \pi |1-k^2|^{1/3} |k|^{1/6}\sqrt{\Im\tau} |K|=\sqrt{\Im\tau}|\eta(\tau/2)|^2.
$$
Here $\eta(\tau/2)$ is the Dedekind eta function. The last equality is a consequence of the identities 
$$
 2\eta^3(\tau)=\vartheta_2(0|\tau) \vartheta_3(0|\tau) \vartheta_4(0|\tau),\quad \eta^2(\tau/2)=\vartheta_4(0|\tau)\eta(\tau),\quad  \vartheta^4_3(0|\tau)= \vartheta^4_2(0|\tau)+ \vartheta^4_4(0|\tau).
$$
By analyzing the expression $\sqrt{\Im\tau}|\eta(\tau/2)|^2$, it is not hard to see that there are only two stationary points: $\tau=2i$ is a saddle point, and $\tau=2e^{2\pi i/3}$ is the unique absolute maximum  of the determinant $\Bbb C\setminus\{0,1\}\ni\lambda\mapsto \det\Delta_{\lambda}$, cf.~\cite[Sec. 4]{OPS}. The case $\tau=2i$ (resp. $\tau=2e^{i\pi/3}$) corresponds to a harmonic (resp. to an equi-anharmonic) position of four points in the divisor $\pmb\lambda$, cf.~Theorem~\ref{ExtDet}.

The reader may also find it interesting that in~\cite[Sec. 3.5.2.]{KK06} the authors demonstrate that the determinant $\det\Delta_\lambda$ is maximal for some (not equi-anharmonic nor harmonic)  positions of four points on the unit circle.
}\end{remark}

\section{Examples and applications}\label{ExAppl}

\subsection{Determinant  for triangulations by plane trees   }
By  Riemann's existence theorem,  the planar bicolored trees are in one-to-one correspondence with the (classes of equivalence of) Shabat polynomials~\cite{BZ,LZ}, see also~\cite{Bishop}. Recall that a Shabat polynomial, also known as a generalized Chebyshev polynomial, is a polynomial with at most two critical values. Thanks to Theorem~\ref{main}, to each bicolored plane tree we can associate a family of spectral invariants $\det \Delta_{f^*\pmb\beta}$. Indeed, a Belyi function $f: \overline{\Bbb C}_x\to \overline{\Bbb C}_z$ (in this case it is a Shabat polynomial) only prescribes a certain gluing scheme of the bicolored double triangles. We can still make any suitable choice of the angles of those triangles, or, equivalently, of the orders $\beta_0$, $\beta_1$, and $\beta_\infty$ of three conical singularities of the constant curvature metric $e^{2\phi}|dz|^2$ on the target Riemann sphere $ \overline{\Bbb C}_z$.
 
As an example, consider  the  Shabat polynomial
$$
f(x)=x^\ell,\quad \ell\in\Bbb N. 
$$
The ramification divisor is 
 $$
 \pmb f=(\ell-1)\cdot 0+(\ell-1)\cdot \infty,\quad |\pmb f|=2\ell-2,
 $$ 
  where $x=0$ is the only point with $f'(x)=0$, and $x=\infty$ is the only pole of  $f$.
     The corresponding bicolored tree is the inverse image of the line segment $[0,1]$ under $f$, see~Fig.~\ref{Snowflake}.
   \begin{figure}[h]
   \centering 
   \begin{tikzpicture}[scale=.7]
\draw[black,solid,thick] (5,0)node[anchor=west]{$1$}--(0,0)node[anchor=north east]{} -- (4.33,2.5)node[anchor=south]{$e^{\frac {2\pi i} \ell}$}; 
\draw[black,solid,thick]  (0,0) -- (4.33,-2.5)node[anchor=north]{$e^{-\frac{2\pi i }\ell }$};
\draw[black,solid,thick]  (0,0) -- (2.5,4.33)node[anchor=south]{$e^{\frac{4\pi i }\ell }$};
\draw[black,solid,thick]  (0,0) -- (2.5,-4.33)node[anchor=north]{$e^{-\frac{4\pi i }\ell }$};
 \filldraw[black] (0,0) circle (3.5pt);  
 \filldraw[white](4.33,2.5) circle (3.5pt);     \draw[black,thick] (4.33,2.5) circle (3.5pt);
 \filldraw[white] (5,0) circle (3.5pt); \draw[black,thick] (5,0) circle (3.5pt); 
\filldraw[white] (4.33, -2.5) circle (3.5pt); \draw[black,thick] (4.33, -2.5) circle (3.5pt);
\filldraw[white] (2.5,4.33) circle (3.5pt);\draw[black, thick] (2.5,4.33) circle (3.5pt);
\filldraw[white] (2.5,-4.33) circle (3.5pt); \draw[black,thick] (2.5,-4.33) circle (3.5pt);
\draw[black,dashed]  (0,0) -- (0,5);    \filldraw[white](0,5) circle (3.5pt);     \draw[black] (0,5) circle (3.5pt);
\draw[black,dashed]  (0,0) -- (0,-5);    \filldraw[white](0,-5) circle (3.5pt);     \draw[black] (0,-5) circle (3.5pt);
\draw[black,dashed]  (0,0) -- (-2.5,4.33);\filldraw[white] (-2.5,4.33) circle (3.5pt);\draw[black] (-2.5,4.33) circle (3.5pt);
\draw[black,dashed]  (0,0) -- (-2.5,-4.33);\filldraw[white] (-2.5,-4.33) circle (3.5pt);\draw[black] (-2.5,-4.33) circle (3.5pt);
   \draw[black,dashed]  (0,0) -- (-5,0); \filldraw[white] (-5,0) circle (3.5pt);\draw[black] (-5,0) circle (3.5pt);
   \draw[black,dashed]  (0,0) -- (-4.33,2.5);\filldraw[white] (-4.33,2.5) circle (3.5pt);\draw[black] (-4.33,2.5) circle (3.5pt);
 \draw[black,dashed]  (0,0) -- (-4.33,-2.5);\filldraw[white] (-4.33,-2.5) circle (3.5pt);\draw[black] (-4.33,-2.5) circle (3.5pt);
  \draw[arrows = {-Stealth[length=10pt, inset=2pt]},dotted,thick] (1.25,2.165) arc (60:300:2.5 and 2.5);
\end{tikzpicture}
\caption{Dessin d'enfant representing the polynomial $f(x)=x^\ell$}
\label{Snowflake}
\end{figure}
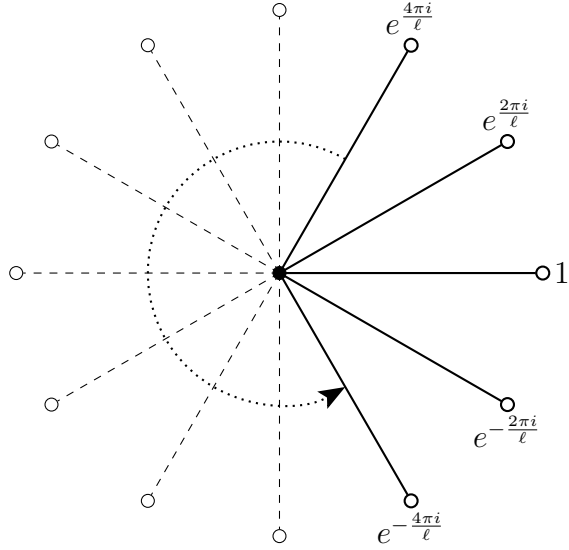
The black colored point is the preimage of  the point $z=0$, and  the white colored points are the $\ell$ preimages  $x=\sqrt[\ell]{1}$ of $z=1$.  
This describes the cyclic triangulation of the Riemann sphere, or, equivalently, the tessellation of the standard round sphere with $(\ell,\infty,\ell)$ bicolored double triangles,   see  Fig.~\ref{Cyclic triangulation}.

Clearly, the first non-zero coefficient in the Taylor expansion of $f-f(0)$ at zero is $c_1=1$, and the first non-zero coefficient  in the Laurent expansion of $f$ at infinity  is $c_2=1$. Hence,   the equality~\eqref{C_f}  immediately implies
  \be\label{calcCf}
  C_f=\frac 1 6 \left(\frac {\ell-1}{\ell}\log|c_1|+(\ell+1)\log\ell\right) +\frac 1 6 \left(-\frac {\ell-1}\ell\log |c_2| -(\ell+1)\log\ell\right)=0,
  \ee
 where $C_f$ is the constant from Theorem~\ref{PBdet}.

The pullback of the divisor $\pmb \beta=  \beta_0\cdot 0+\beta_1\cdot 1+\beta_\infty \cdot \infty$ by $f(x)=x^\ell$ is the divisor $ f^*{\pmb \beta}=\sum_k \left(f^*\pmb\beta\right)_k\cdot x_k$. For the latter one we have
 $$
 f^*{\pmb \beta}=(\ell(\beta_0+1)-1)\cdot0    +\beta_1\cdot\{\sqrt[\ell]{1}\}+(\ell(\beta_\infty+1)-1)\cdot\infty, 
 $$
 where $\{\sqrt[\ell]{1}\}$ stands for the set of $\ell$ radicals $\sqrt[\ell]{1}$ in $\Bbb C_x$ (those are the white colored points of the ''snowflake`` in Fig.~\ref{Snowflake}).
  The notation $\beta_1\cdot\{\sqrt[\ell]{1}\}$  means that each element of the set $\{\sqrt[\ell]{1}\}$ is a marked point of weight $\beta_1$.

  \begin{figure}[h]
 \centering\includegraphics[scale=.3]{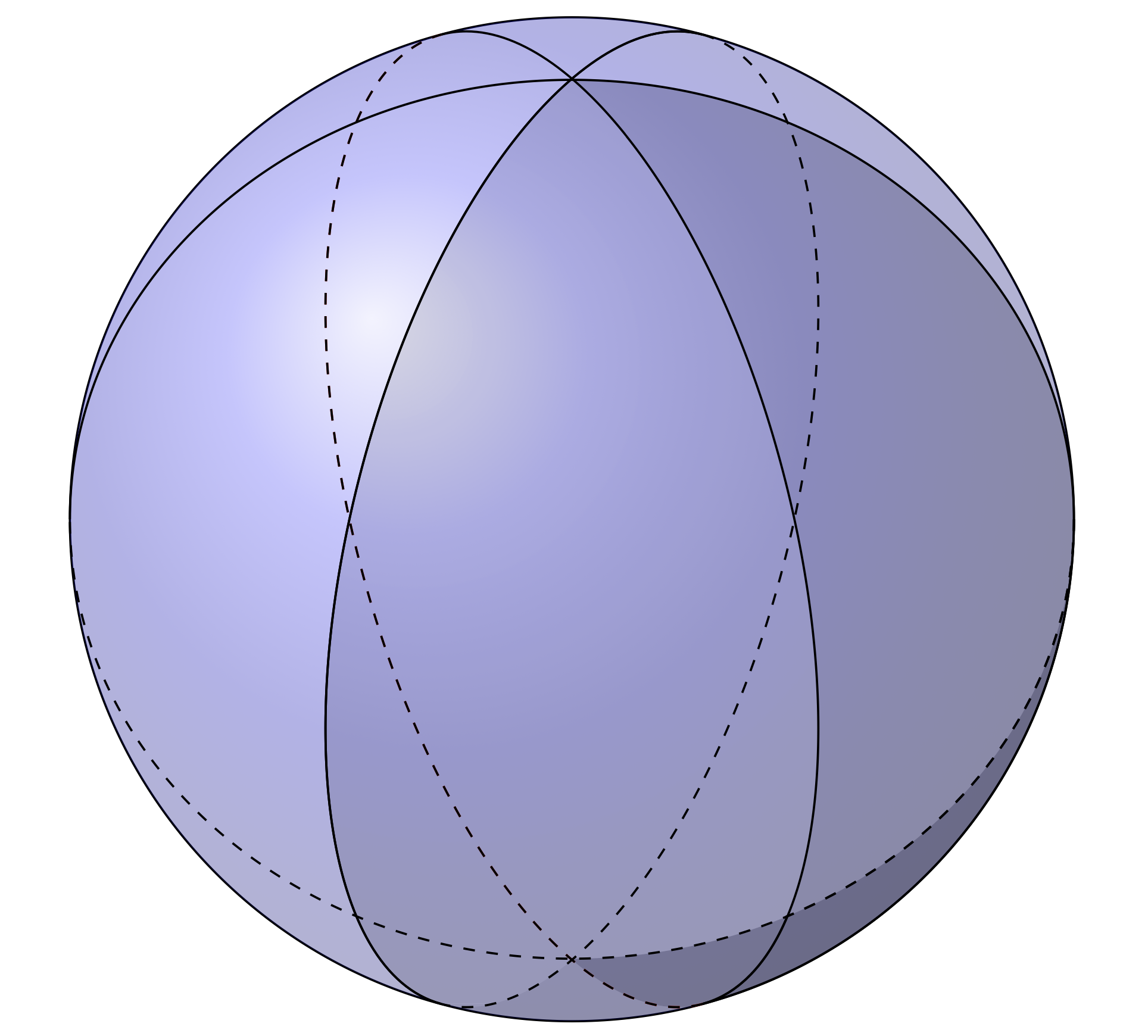}
 \caption{Cyclic triangulation}
\label{Cyclic triangulation}
\end{figure}

\begin{theorem}[Cyclic triangulation]\label{CTr}  Let $m_{\pmb\beta}$ be the unit area Gaussian curvature $2\pi(|\pmb\beta|+1)$ metric of $S^2$-like double triangle, see Section~\ref{PrelimTriangulation}. Let $f(x)=x^\ell$ with $\ell\in\Bbb N$, cf. Fig~\ref{Cyclic triangulation}. Then for  the  zeta-regularized  spectral determinant of the Friedrichs  Laplacian $\Delta_{f^*\pmb\beta}$ corresponding to the area $\ell$ pullback metric $f^*m_{\pmb\beta}=e^{2f^*\phi}|dx|^2$  we have 
$$
\ba
\log\frac { \det \Delta_{f^*\pmb\beta}} \ell = &  \ell \Bigl(\log \det \Delta_{\pmb\beta} + \mathcal C(\beta_0)   + \mathcal C(\beta_\infty)-{\bf C}\Bigr)
\\
&+\frac 1 6  \left(\ell-\frac 1 {\ell}\right)\left(\frac {\Psi(\beta_0,\beta_1,\beta_\infty)}{\beta_0+1}  +\frac {\Psi(\beta_\infty,\beta_1,\beta_0)}{\beta_\infty+1}  \right) 
\\
    &    -\frac 1 6   \left( \ell(\beta_0+\beta_\infty+2)+\frac 1 { \ell(\beta_0+1)}  +\frac 1 {\ell(\beta_\infty+1)} \right)\log \ell
 \\
 &  - \mathcal C\left(  \ell(\beta_0+1)-1 \right)-\mathcal C\left(   \ell(\beta_\infty+1)-1\right)+ {\bf C},
\ea
$$
where the right hand side is an explicit function of $\ell\in\Bbb N$ and $(\beta_0,\beta_1,\beta_\infty)\in(-1,0]^3$. 

 Here $\beta\mapsto \mathcal C(\beta)$ is the function~\eqref{cbeta}, $\bf C$ is the constant introduced in~\eqref{bfC}, and the function 
 \be\label{eqn_Psi}
\begin{aligned}
\Psi(\beta_0,\beta_1,\beta_\infty)& = \log \frac{\Gamma(-\beta_0)}{\Gamma(1+\beta_0)}
\\
&+\frac 1 2 \log \frac{   \Gamma\left(2+\frac {|\pmb \beta|} 2\right)  \Gamma\left(\beta_0-\frac {|\pmb \beta|} 2\right)
 \Gamma\left(1+\frac{|\pmb \beta|} 2-\beta_1\right)  \Gamma\left(1+\frac {|\pmb \beta|} 2-\beta_\infty\right) }{  \pi  \Gamma\left(-\frac {|\pmb \beta|} 2\right)  \Gamma\left(1+\frac {|\pmb \beta|} 2-\beta_0\right)    \Gamma\left(\beta_1-\frac {|\pmb \beta|} 2\right) \Gamma\left(\beta_\infty-\frac {|\pmb \beta|} 2\right) }
 \end{aligned}
\ee
 is the one from~\eqref{asympt_phi}.  Recall  that  $\det \Delta_{\pmb\beta}$ stands for an explicit function~\eqref{CalcVar}, whose values are the determinants of the unit area $S^2$-like double triangles of Gaussian curvature $2\pi(|\pmb\beta|+2)$.
\end{theorem}
\begin{proof} Recall that  $\Psi(\beta_1,\beta_2,\beta_3)=\Phi(\beta_1,\beta_2,\beta_3)+ \log 2$ with explicit function $\Phi$ from~\cite[Prop. A.2)]{KalvinLast}. This implies~\eqref{eqn_Psi}, where $\Gamma$ stands for the Gamma function.

 Now the assertion is an immediate consequence of Theorem~\ref{main} together with the asymptotics~\eqref{asympt_phi}, and the calculation of $C_f$ in~\eqref{calcCf}. 
\end{proof}

\begin{example}[Dihedra]{\rm
    For  the determinant of the Gaussian curvature $2\pi(\beta+2/\ell)$  area $\ell$ regular dihedron  with $\ell$ conical singularities of order  $\beta$ we obtain 
\be\label{DDIH}
\ba
\log { \det \Delta^\ell_{Dihedron}} = &  \ell\log \det \Delta_{\pmb\beta} + 2\ell\mathcal C\left(\frac 1 \ell -1\right)  
\\
&+\frac {\ell^2- 1 } 3    \Psi\left(\frac 1 \ell -1,\beta,\frac 1 \ell -1\right)    +\frac 1 3     \log \ell+ (1-\ell){\bf C}.
\ea
\ee
This is a direct consequence of Theorem~\ref{CTr}, where we take 
$$
\pmb\beta=\left(\frac 1 \ell -1 \right)\cdot 0+\beta\cdot 1+ \left(\frac 1 \ell -1 \right)\cdot \infty.
$$

In particular, when $\beta=-2/\ell$, we obtain the determinant of the flat regular dihedron of area $\ell$. In the case $\beta=0$, we obtain a surface isometric to the round sphere in $\Bbb R^3$ of area $\ell$ and its determinant.
Finally, as $\beta\to -1^+$ the determinant increases without any bound in accordance with the asymptotics
$$
\ba
\log { \det \Delta^\ell_{Dihedron}} =\frac \ell {12}   \left(  -2\log(\beta+1) +\log \left(1-\frac 2 \ell\right)+\log 2\pi -2 +24\zeta_R'(-1)   \right)\frac 1 {\beta+1}
\\
 -\frac \ell  2 \log(\beta+1) +O(1)
 \ea
$$
of the right-hand side in~\eqref{DDIH}. In the limit $\beta=-1$ we obtain a surface of Gaussian curvature $2\pi(2/\ell-1)$ with $\ell$ cusps. The spectrum of the corresponding Laplacian is no longer discrete~\cite{Judge,Judge2}.

}
\end{example}

\begin{example}[Tetrahedron]{\rm Here we find the spectral determinant of Laplacian for a constant curvature regular tetrahedron of total area $4\pi$ with (four) conical singularities of order $\beta$. Or, equivalently, the spectral determinant of the Platonic surface of Gaussian curvature $2\beta+1$ with four faces.    Up to a rescaling,  this is a particular case of Theorem~\ref{CTr}, that corresponds to the choice $\ell=3$ and 
$$
\pmb\beta=\frac {\beta -2}3\cdot 0+\beta \cdot 1 + \left (-\frac 2 3\right) \cdot\infty.
$$ 
Here and in the remaining part of this section we use the standard rescaling property~\cite[Sec. 1.2]{KalvinJFA} of the determinant 
$$
 \log \det \Delta^{4\pi}_{f^*\pmb\beta}=\log \det \Delta_{f^*\pmb\beta}-\zeta_{f^*\pmb\beta}(0) \log \frac {4\pi}{\deg f}
$$
in order to normalize the total area to $4\pi$, where
$$
\zeta_{f^*\pmb\beta}(0)=\frac {|f^*\pmb\beta|+2}{6}-\frac 1{12}\sum_k \left((f^*\pmb\beta)_k+1-\frac 1 {(f^*\pmb\beta)_k+1}\right)-1.
$$
As a result, in the case $\beta=0$, when all conical singularities disappear, we obtain a surface isometric to the standard unit sphere $x_1^2+x_2^2+x_3^2=1$ in $\Bbb R^3$ and its determinant
$$
\log { \det \Delta} =1/2-4\zeta'_R(-1).
$$

In general,  for the constant curvature regular tetrahedron of total area $4\pi$ with conical singularities of order $\beta$,  Theorem~\ref{CTr} gives
\be\label{DetTetr1}
\ba
\log { \det \Delta^{4\pi}_{Tetrahedron}} &  =   3 \left(\log \det \Delta_{\pmb\beta} + \mathcal C\left(\frac{\beta-2}3\right)   +\mathcal C\left(-\frac 2   3\right)   \right)
\\&
+\frac 4 3  \left(\frac  1 {\beta+1}   \Psi\left(\frac{\beta-2}3,\beta,-\frac 2 3\right) +{\Psi\left(-\frac 2 3,\beta,\frac{\beta-2}3\right)}  \right)  
\\&
   -\frac 1 3   \left( \beta-3+\frac 1 { \beta+1}  \right)\left(\log (4\pi) -\frac 1 2 \log 3\right)    - \mathcal C\left( \beta \right)   -2 {\bf C},
\ea
\ee
where the right hand side is an explicit function of $\beta$. A graph of this function is depicted in Fig.~\ref{Platonic} as a solid line.

\begin{figure}[h]
 \centering\includegraphics[scale=.8]{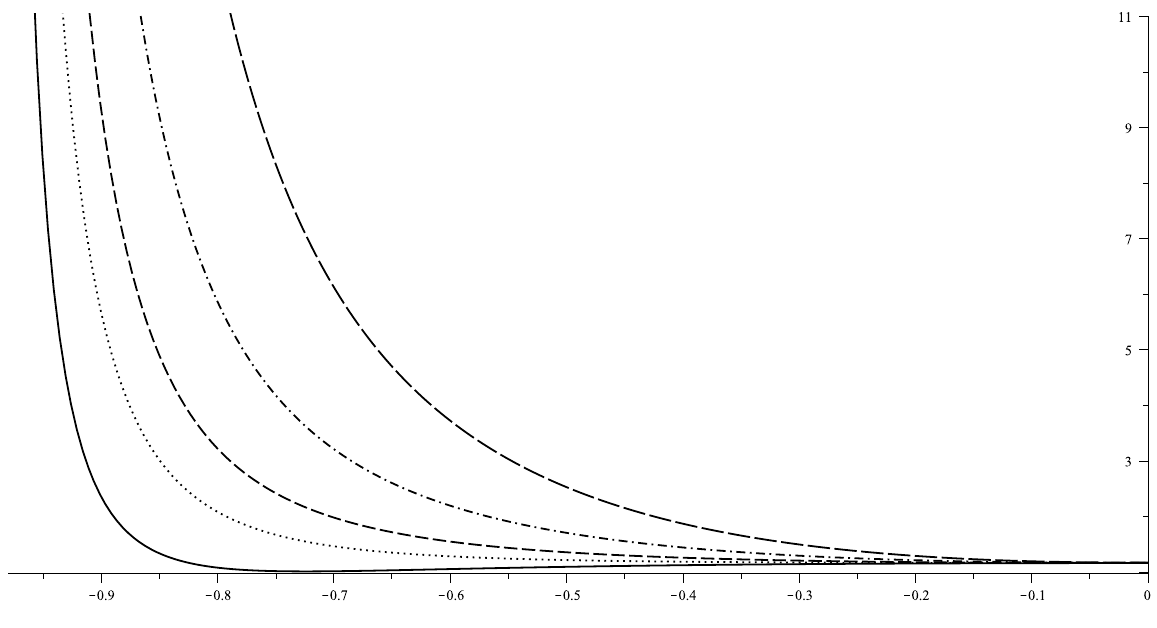}
 \caption{Graphs of the logarithm of the spectral determinant of Laplacian on the surfaces of Platonic solids of area $4\pi$ as a function of the order $\beta\in(-1,0]$ of the conical singularities: {\bf a.} Regular Tetrahedron of Gaussian curvature $2\beta+1$ (solid line), {\bf b.} Regular Octahedron of Gaussian curvature $3\beta+1$ (dotted line), {\bf c.} Regular Cube of Gaussian curvature $4\beta+1$ (dashed line), {\bf d.} Regular Icosahedron of Gaussian curvature $6\beta+1$ (dash-dotted line), {\bf e.} Regular Dodecahedron  of Gaussian curvature $10\beta+1$ (long-dashed line). 
 The point $(0, 1/2-4\zeta'_R(-1))\approx (0,1.16)$ on the graphs corresponds to the logarithm of the spectral determinant of the unit round sphere in $\Bbb R^3$. As $\beta\to -1^+$, the determinants increase without any bound in accordance with the asymptotics~\eqref{IdealTetrahedron},~\eqref{IdealOctahedron},~\eqref{IdealCube},~\eqref{IdealIcosahedron}, and~\eqref{IdealDodecahedron}. In the limit $\beta=-1$,  the conical singularities turn into cusps, and one obtains the ideal Platonic surfaces; the spectrum of the corresponding Laplacians is no longer discrete.
  }
 \label{Platonic}
\end{figure}

In particular, in the case $\beta=-1/2$ we obtain a surface  $(\overline{\Bbb C}_x, f^*m_{\pmb\beta})$ isometric to the surface of a Euclidean regular tetrahedron. Note that $\mathcal C(-1/2)$ can be evaluated as in~\eqref{C(-1/2)}.
 As a result, the  formula~\eqref{DetTetr1} for the determinant reduces to
$$
\log { \det \Delta^{4\pi}_{Tetrahedron}}|_{\beta=-1/2}=  \log \frac 4 3  - 3\log \Gamma\left(\frac 2 3\right) +\frac 3 2 \log \pi. 
$$
Alternatively, the latter expression for the determinant can be obtained by applying the partially heuristic Aurell-Salomonson formula~\cite{AS2} to the explicitly evaluated pullback  of the  flat metric 
$$
m_{\pmb\beta}=c^2 |z|^{-1}|z-1|^{-1}  {|dz|^2}$$
by $f(x)=x^3$, where $c$ is a scaling coefficient that normalizes the total area of $m_{\pmb\beta}$ to $4\pi/3$; for a rigorous proof of the Aurell-Salomonson formula we refer to~\cite[Sec. 3.2]{KalvinJFA}.

Finally,  as $\beta\to -1^+$, the determinant grows without any bound in accordance with the asymptotics
\be\label{IdealTetrahedron}
\log { \det \Delta^{4\pi}_{Tetrahedron}} =\left (  -\frac 2 3 \log (\beta+1) -\frac 2 3   +8 \zeta_R'(-1)  \right)\frac 1 {\beta+1} -2  \log (\beta+1)+ O(1) 
\ee
of the right hand side in~\eqref{DetTetr1}, cf. Fig~\ref{Platonic}.
 In the limit $\beta=-1$  we get a surface isometric to an ideal tetrahedron: a surface of Gaussian curvature $-2$ with four cusps;  the spectrum of the  Laplacian on an ideal tetrahedron is not discrete~\cite{Judge,Judge2}. 
}
\end{example}

\begin{example}[Octahedron]\label{EOctahedron}{\rm  Here we find the determinant of Laplacian for a constant curvature regular octahedron of area $4\pi$ with  (six) conical singularities of order $\beta$. Or, equivalently, the spectral determinant of the Platonic surface of Gaussian curvature $3\beta+1$ with  eight faces.  In Theorem~\ref{CTr}  we substitute $\ell=4$ and 
\be\label{ocDiv}
\pmb\beta=\frac {\beta-3} 4\cdot 0+\beta \cdot 1 +\frac {\beta-3} 4 \cdot\infty. 
\ee
As a result, after an appropriate rescaling,  we obtain
\be\label{DetOct1}
\ba
\log  \det \Delta^{4\pi}_{Octahedron}   =      4\log \det \Delta_{\pmb\beta}        -  \left( \beta+1+\frac 1 {\beta+1} \right)\left(\frac 2 3 \log 2 +\frac 1 2 \log \pi\right) 
\\
+\frac 5 3 \log \pi+\frac 5 {\beta +1} \Psi\left(\frac {\beta-3} 4,\beta,\frac {\beta-3} 4\right)
 +2\log 2+8 \mathcal C\left(\frac {\beta-3} 4\right)   - 2\mathcal C\left(  \beta \right)-3{\bf C},
\ea
\ee
where the right hand side is an explicit function of $\beta$. A graph of this function is depicted in Fig.~\ref{Platonic} as a dotted line.

In the case  $\beta=0$ we obtain a surface isometric to the standard unit sphere in $\Bbb R^3$ and a representation for its determinant in terms of the determinant of spherical $(4,\infty,4)$ double triangle, see also Example~\ref{Spindles} below.

In the case $\beta=-1/3$ we obtain a surface  $(\overline{\Bbb C}_x, f^*m_{\pmb\beta})$ isometric to the (flat) surface of Euclidean regular octahedron.
 The formula~\eqref{DetOct1} for the determinant reduces to
$$
\log  \det \Delta^{4\pi}_{Octahedron} |_{\beta=-1/3}=6\zeta_R'(-1) + \frac {35} {24}\log \frac 4 3  - \frac {13} 2 \log \Gamma\left (\frac 2 3\right) + \frac {13} 4 \log \pi. 
$$

Let us also note, that in the case $\beta=-1/2$ we get the tessellation of the singular sphere $(\overline{\Bbb C}_x, f^*m_{\pmb\beta}|_{\beta=-1/2})$ by the hyperbolic $(2,3,8)$-triangle. The surface  $(\overline{\Bbb C}_x, f^*m_{\pmb\beta}|_{\beta=-1/2})$, where $\pmb\beta$ is the divisor~\eqref{ocDiv},  is isometric to a regular hyperbolic octahedron with conical singularities of angle $\pi$. This is remarkable, as a double of  $(\overline{\Bbb C}_x, f^*m_{\pmb\beta}|_{\beta=-1/2})$ is the Bolza curve, known as the most symmetrical genus two smooth hyperbolic surface, see e.g.~\cite{KW}.  To the best of our knowledge, the exact value of the spectral determinant of the Bolza curve, endowed with the smooth Gaussian curvature $-1$ metric, is not yet known. For a numerical study see~\cite{StrUs}.

Finally, as $\beta\to -1^+$, the determinant grows without any bound in accordance with the asymptotics
\be\label{IdealOctahedron}
\log  \det \Delta^{4\pi}_{Octahedron}= \left(-\log(\beta+1) +\frac 1 2 \log 2 -1+12\zeta_R'(-1)     \right)\frac 1 {\beta+1}-3\log (\beta+1)+O(1)
\ee
of the right-hand side in~\eqref{DetOct1}.
 In the limit $\beta=-1$  we get a surface isometric to an ideal octahedron: a surface of Gaussian curvature $-2$ with six cusps, cf.~\cite{Judge,Judge2}.  The spectrum of the corresponding Laplacian is no longer discrete. 

}
\end{example}

\begin{example}[Spindles]\label{Spindles}{\rm
Let $m_{\pmb\beta}^S=e^{2\phi}|dz|^2$ be the metric of a spindle with two antipodal singularities~\cite{TroyanovSp}, where
$$
\phi(z)=\beta\log |z|+ \log 2+\log (\beta+1)-\log (1+|z|^{2\beta+2} ).
$$
The (regularized) Gaussian curvature of  $m_{\pmb\beta}^S$ is $1$, and the total area is $S= 4\pi(\beta+1)$. The metric represents the divisor $\pmb\beta=\beta_0\cdot 0+\beta_1\cdot 1+\beta_\infty\cdot\infty $ with $\beta_0=\beta_\infty=:\beta\in(-1,\infty)$ and $\beta_1=0$. The spindle $(\overline{\Bbb C}_z,m_{\pmb\beta}^S )$
 is isometric to the spherical double triangle glued from two copies of a spherical triangle (a lune) with internal angles $\bigl(\pi(\beta+1), \pi, \pi(\beta+1)\bigr)$, cf. Fig.~\ref{Cyclic triangulation}. 
 
  For the asymptotics of the metric potential
 we have
 $$
 \ba
 \phi(z)&=\beta\log |z|+\phi_0   +o(1),\quad z\to 0,\quad \phi_0=\log2(\beta+1),
 \\
  \phi(z)&=-(\beta+2)\log|z|  +\phi_\infty  +o(1), \quad z\to \infty,\quad \phi_\infty=\log 2(\beta+1).
 \ea
 $$
Clearly, for the pullback of $m_{\pmb\beta}^S$  by the Shabat polynomial  $f(x)=x^\ell$  we have
$$
f^*m_{\pmb\beta}^S=  \frac{4 \ell^2(\beta+1)^2|x|^{2(\ell(\beta+1)-1)}|d x|^2}{(1+|x|^{2\ell( \beta +1) } )^2},\quad f^*\pmb\beta= (\ell(\beta+1)-1)\cdot 0+ (\ell(\beta+1)-1)\cdot\infty. 
$$

The surface $(\overline{\Bbb C}_x,f^*m_{\pmb\beta}^S)$ is isometric to the surface glued from $\ell$ copies of the spindle $(\overline{\Bbb C}_z, m^S_{\pmb\beta})$ with a cut from the conical point at $z=0$ to the conical  point at $z=\infty$. Or, equivalently,  $(\overline{\Bbb C}_x,f^*m_{\pmb\beta}^S)$ is triangulated by $2\ell$ bicolored copies of a spherical  triangle with internal angles $\bigl(\pi(\beta+1), \pi, \pi(\beta+1)\bigr)$ and unit Gaussian curvature, cf. Fig.~\ref{Cyclic triangulation}. 
 The surface  $(\overline{\Bbb C}_x,f^*m_{\pmb\beta}^S)$ is again a spindle with two antipodal conical singularities:   the metric $m_{\pmb\beta}^S$ coincides with $f^*m_{\pmb\beta}^S$ after the replacement of $\beta$ by $\ell(\beta+1)-1$ (and $z$ by $x$).

 As is known~\cite{KalvinLast,KalvinJFA,Klevtsov,SpreaficoZerbini}, the spectral determinant of the Friedrichs Laplacian $\Delta^S_{\pmb\beta}$ on the spindle $(\overline{\Bbb C}_z,m^S_{\pmb\beta})$ satisfies 
 \be\label{spindledet}
 \log\frac {\Delta^S_{\pmb\beta}}{S}=\frac {\beta+1} 3-\frac 1 3 \left( \beta+1 +\frac 1 {\beta+1}  \right)\log 2(\beta+1)  -2\mathcal C(\beta)  
 +{\bf C},\quad \beta>-1.
 \ee

For the spectral determinant of the spindle $(\overline{\Bbb C}_x,f^*m_{\pmb\beta}^S)$  Theorem~\ref{PBdet}  (after an appropriate rescaling) gives
\be\label{cyclicsphere}
\ba
\log\frac {\det\Delta^{\ell S}_{f^*\pmb\beta}}{\ell S}=\ell \Bigl(\log\frac {\det\Delta^{S}_{\pmb\beta}}{4\pi(\beta+1)} +2\mathcal C(\beta)-{\bf C}\Bigr) +\frac 1 3\left(\ell-\frac 1 {\ell}\right)\frac {\log2 (\beta+1)}{\beta+1}
  \\
      -\frac 1 6 \left( 2\ell(\beta+1)+\frac 2 { \ell(\beta+1)}  \right)\log \ell
-2\mathcal C\left(\ell(\beta+1)-1\right)  +{\bf C}.
  \ea
 \ee
As expected, after the substitution~\eqref{spindledet} and the replacement of $\ell(\beta+1)-1$ by $\beta$, the equality~\eqref{cyclicsphere} reduces to the one in~\eqref{spindledet}.

In particular, for $\beta=\frac 1 \ell -1$ the pullback of the spindle metric $m_{\pmb\beta}^S$ by $f$ is the metric of the standard round sphere of total area $\ell S=4\pi$. In this case, the equality~\eqref{cyclicsphere} expresses the determinant of Laplacian of the standard sphere 
$$\det\Delta^{4\pi}_{f^*\pmb\beta}=\exp(1/2-4\zeta'_R(-1))
$$
 in terms of the determinants $\det\Delta^{4\pi/\ell}_{\pmb\beta}$ of the spherical double triangle (or, equivalently,  of the bicolored double lune) corresponding to the cyclic triangulation of the sphere via the Shabat polynomial $f(x)=x^\ell$,  cf. Fig.~\ref{Cyclic triangulation}. 
 }
\end{example}

\subsection{Determinant for dihedral triangulation}
 Consider the Belyi function 
\be\label{BelyiDihedral}
f(x)=1-\left(\frac {1-x^\ell}{1+x^\ell}\right)^2,\quad \ell\in\Bbb N. 
\ee
This is a ramified covering of degree $2\ell$. For the ramification divisor of $f$ we have
$$
{\pmb {f}} = (\ell-1)\cdot 0 +  1 \cdot \{\sqrt[\ell]{1} \} +  1 \cdot \{\sqrt[\ell]{-1} \} +(\ell-1)\cdot\infty,\quad |{\pmb {f}}|
=4\ell-2.
$$
This Belyi function defines a tessellation of the standard round sphere with  $(2,2,\ell)$-triangles, cf. Fig~\ref{DihTriang}. 
As we show in the proof of Theorem~\ref{Jan20} below, to this Belyi map there corresponds the constant
\be\label{CfDT}
C_f=\frac 2 3 \left(\ell-\frac 1 \ell\right)\log 2. 
\ee
The Belyi map $f$ sends the marked points listed in the ramification divisor $\pmb f$ 
 to the points $0,1,\infty$ as follows:
$$
f(0)=0,\quad f(\infty)=0, \quad f( \sqrt[\ell]{-1 })=\infty,  \quad f( \sqrt[\ell]{1 })=1.
$$
For the pullback divisor of $\pmb \beta=  \beta_0\cdot 0+\beta_1\cdot 1+\beta_\infty \cdot \infty$ by $f$ we obtain
$$
f^*\pmb\beta=  (\ell(\beta_0+1)-1)\cdot 0 + ( 2\beta_1+1) \cdot \{\sqrt[\ell]{1} \} +  (2\beta_\infty+1) \cdot \{\sqrt[\ell]{-1} \} +  (\ell(\beta_0+1)-1)\cdot\infty.
$$

\begin{figure}[h]
 \centering\includegraphics[scale=.3]{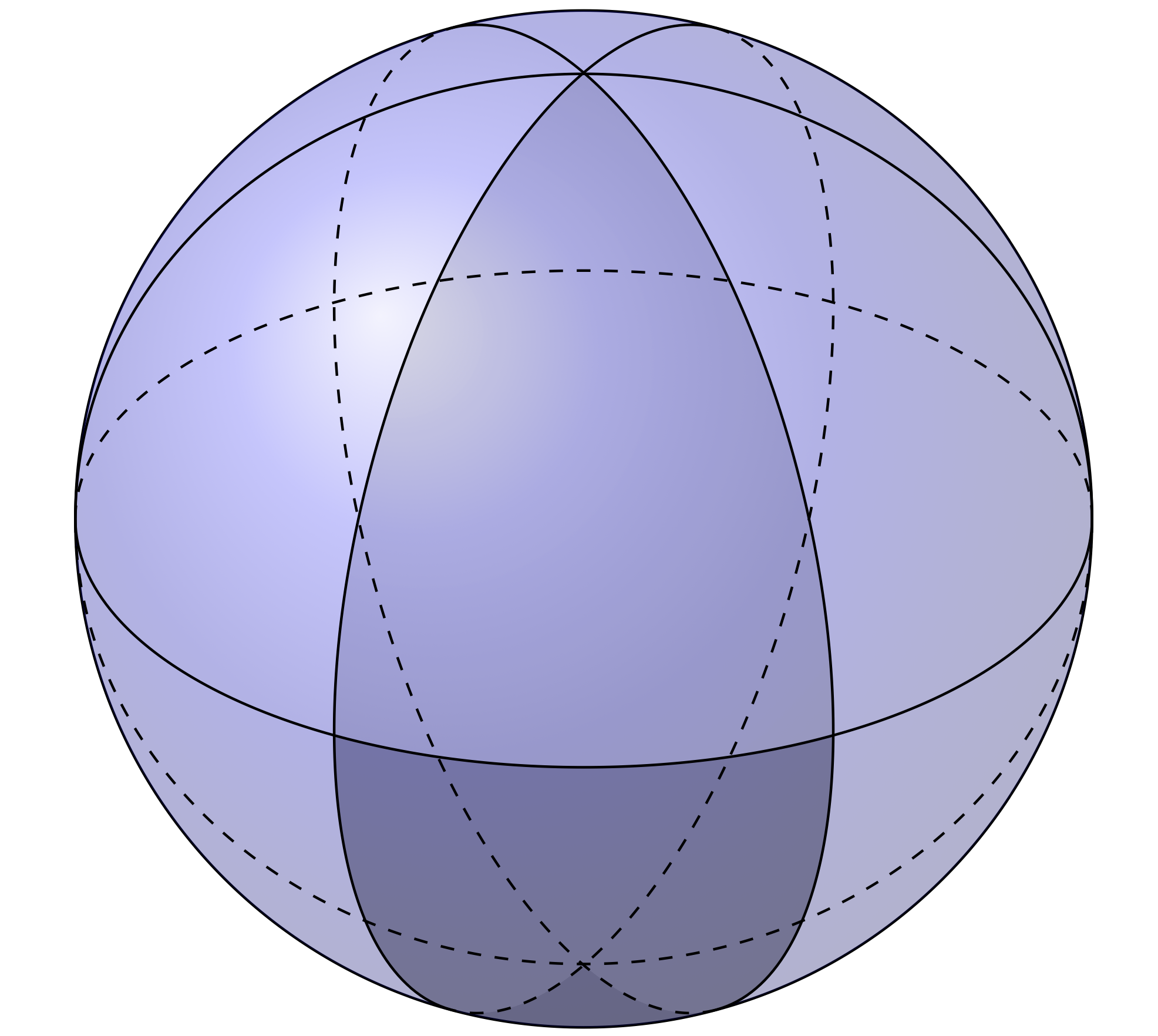}
  \caption{Dihedral triangulation}\label{DihTriang}
\end{figure}

\begin{theorem}[Dihedral triangulation]\label{Jan20} 
  Let $m_{\pmb\beta}$ be the unit area Gaussian curvature $2\pi(|\pmb\beta|+1)$ metric of $S^2$-like double triangle, see Section~\ref{PrelimTriangulation}. For the spectral zeta-regularized determinant of  the Friedrichs Laplacian $\Delta_{f^*\pmb\beta}$ corresponding to the area $2\ell$ pullback metric $f^*m_{\pmb\beta}$ on the Riemann sphere $\overline{\Bbb C}_x$, where $f$ is the dihedral Belyi map~\eqref{BelyiDihedral}, we have the following explicit expression:
$$
\ba
\log & \frac { \det\Delta_{f^*\pmb\beta}} {2\ell} =   2\ell\Bigl( \log \det \Delta_{\pmb\beta}   +\mathcal C_{\pmb\beta}   -{\bf C}\Bigr)+C_f
\\
&+\frac 1 3  \left(\ell-\frac 1 {\ell}\right)\frac {\Psi(\beta_0,\beta_1,\beta_\infty)}{\beta_0+1} +\frac \ell 4\left(\frac {\Psi(\beta_1,\beta_0,\beta_\infty)}{\beta_1+1} +\frac { \Psi(\beta_\infty,\beta_1,\beta_0)}{\beta_\infty+1} \right)
\\
    &    -\frac 1 3   \left( \ell(\beta_0+1)+\frac 1 { \ell(\beta_0+1)}  \right)\log \ell -\frac \ell 3 \left(   \beta_1+\beta_\infty+2+\frac 1{4\beta_1+4}+\frac 1{4\beta_\infty+4}\right)   \log 2
 \\
 &  -2 \mathcal C\left(  \ell(\beta_0+1)-1 \right)-\ell\Bigl(\mathcal C\left(  2\beta_1+1\right)+\mathcal C\left(  2\beta_\infty+1\right)\Bigr)
   +{\bf C},
\ea
$$
where $\mathcal C_{\pmb\beta} =\mathcal C(\beta_0)+\mathcal C(\beta_1)+\mathcal C(\beta_\infty)$ and $C_f$ is the same as in~\eqref{CfDT}.

Recall that  $\det \Delta_{\pmb\beta}$ stands for an explicit function~\eqref{CalcVar}, whose values are the determinants of the unit area $S^2$-like double triangles of Gaussian curvature $2\pi(|\pmb\beta|+2)$ tessellating the singular sphere $(\overline{\Bbb C}_x,f^*m_{\pmb\beta})$. The function $\beta\mapsto \mathcal C(\beta)$ is defined in~\eqref{cbeta}, the function $\Psi$ is the same as in~\eqref{eqn_Psi}, and $\bf C$ is the constant introduced in~\eqref{bfC}. 
\end{theorem}
\begin{proof}
For the derivative of the Belyi function~\eqref{BelyiDihedral} we have
$$
f'(x)=\frac{4\ell x^{\ell-1}(1-x^\ell)}{(x^\ell+1)^3}.
$$
As a result,  we immediately obtain  the asymptotics in vicinities of the critical points of the form~\eqref{asymp_log_f1}. The first non-zero coefficients $c_k$ in the Taylor expansions  (cf. Proposition~\ref{PBdet}) satisfy
$$
|c_k|=\left\{
\begin{array}{cc}
4 , & x_k=0,    \\
 \ell^2/4, &   x_k\in\{    \sqrt[\ell]{1 }\},   \\
4,  &    x_k=\infty,\quad  k=n=2\ell+2.     
\end{array}
\right.
$$

Similarly,  we obtain the asymptotics in vicinities of the poles of the form~\eqref{asymp_log_f2} with the first non-zero Laurent coefficients  $c_k$ satisfying 
$$
 |c_k|=\ell^2/4, \quad  x_k\in\{\sqrt[\ell]{- 1 }\}.
$$

Now the expression~\eqref{C_f} for $C_f$ implies~\eqref{CfDT}. As a result, the assertion is an immediate consequence of Theorem~\ref{main} together with the asymptotics~\eqref{asympt_phi}. 
\end{proof}

\begin{example}[Dihedra]{\rm  Here we deduce an alternative formula for the determinant of the regular dihedron of Gaussian curvature $K=2\pi(\beta+1/\ell)$ and area $2\ell$ with $2\ell$ conical singularities of order $\beta$: In Theorem~\ref{Jan20} we take 
$$
\pmb\beta=\left(\frac 1 \ell -1 \right)\cdot 0+\frac{\beta-1} 2 \cdot 1+\frac{\beta-1} 2\cdot \infty,
$$
and obtain 
$$
\ba
\log & \,{ \det\Delta^{2\ell}_{Dihedron}}  =   2\ell\Bigl( \log \det \Delta_{\pmb\beta}   +\mathcal C_{\pmb\beta}       -\mathcal C\left(  \beta\right) \Bigr) +    \frac 1 3  \log \ell 
   +(1-2\ell){\bf C}
\\
&+\frac {\ell^2-1 } 3   {\Psi\left ( \frac 1 \ell -1,\frac{\beta-1} 2,\frac{\beta-1} 2\right)} 
+\frac \ell { \beta+1}{ \Psi\left (\frac{\beta-1} 2,\frac{\beta-1} 2,\frac 1  \ell -1\right)}
\\
    &    +\left(\frac 2 3 \left(\ell-\frac 1 \ell\right) -\frac \ell 3 \left(  \beta+1+\frac 1{\beta+1}\right)   +1    \right)  \log 2.
\ea
$$

}

\end{example}

\begin{example}[Octahedron]{\rm
We obtain an alternative formula for the determinant of Laplacian on a regular octahedron of Gaussian curvature $3\beta+1$ with (six) conical singularities of order $\beta$: In Theorem~\ref{Jan20} we take  $\ell=2$ and 
$$
\pmb\beta=\frac{\beta-1} 2\cdot 0+\frac{\beta-1} 2\cdot 1  +\frac{\beta-1} 2\cdot\infty.$$
This together with the standard rescaling property of determinants implies
$$
\ba
\log\, &  { \det\Delta^{4\pi}_{Octahedron}}  =   4\log \det \Delta_{\pmb\beta}    -   \left( \beta+1+\frac 1 { \beta +1}  \right)\left(\log 2+\frac 1 2 \log \pi \right)
+\frac 5 3 \log \pi
\\
&+\frac 3 {{\beta+1} }    {\Psi\left(\frac{\beta-1} 2,\frac{\beta-1} 2,\frac{\beta-1} 2\right)}    +3\log 2+12\mathcal C\left(\frac{\beta-1} 2\right)-6\mathcal C\left( \beta\right)  -3{\bf C},
\ea
$$
where the right hand side is an explicit function of $\beta$; this expression is equivalent to the one in~\eqref{DetOct1}. A graph of this function is depicted in Fig.~\ref{Platonic} as a dotted line. }
\end{example}

\subsection{Determinant for tetrahedral triangulation}

The tetrahedral Belyi map is given by the function
\be\label{BelyiTetrahedral}
f(x)=-64\frac{(x^3+1)^3}{(x^3-8)^3x^3},\quad \deg f=12.
\ee
For the ramification divisor, we obtain
$$
\pmb f= 2\cdot\left\{0,\sqrt[3]{8}\right\}+1\cdot\left\{ \sqrt[3]{-10\pm6\sqrt  3}\right\} + 2\cdot\left\{\sqrt[3]{-1},\infty\right\} ,\quad  |\pmb f|=22. 
$$
The Belyi map sends the marked points listed in the divisor $\pmb f$ to the points $0,1,\infty$ of the target Riemann sphere $\overline{\Bbb C}_z$ as follows: 
$$
f(0)=\infty, \quad f(\sqrt[3]{-1})=0,\quad f\left(\sqrt[3]{-10\pm6\sqrt  3}\right)=1, \quad f(\sqrt[3]{8})=\infty, \quad f(\infty)=0.
$$
Here $ f(x)=1$ are the edge midpoints of a regular tetrahedron. The poles  $f(x)=\infty$ correspond to its vertices. 
The zeros $f(x)=0$, i.e. the roots of the numerator and $x=\infty$, are the centers of the faces. This defines a tessellation of the standard round sphere with spherical $(2,3,3)$-triangles, cf.~Fig~\ref{tetrahedral}.  A picture of the corresponding {\it dessin d'enfant} can be found e.g. in~\cite[Fig. 2]{Margot Zvonkin}. In the proof of Theorem~\ref{T_t} below we show that to the Belyi function~\eqref{BelyiTetrahedral} there corresponds the constant
\be\label{CfTT}
C_f=\log 2 +\frac 9 4 \log 3. 
\ee

For the pullback of the divisor $\pmb \beta=  \beta_0\cdot 0+\beta_1\cdot 1+\beta_\infty \cdot \infty$ by $f$ we obtain
$$
f^*\pmb\beta=     (3\beta_0+2)\cdot\left\{\sqrt[3]{-1},\infty\right\}+(2\beta_1+1)\cdot\left\{ \sqrt[3]{-10\pm6\sqrt  3}\right\}+(3\beta_\infty+2)\cdot \left\{0,\sqrt[3]{8}\right\}.
$$

\begin{figure}[h]
 \centering\includegraphics[scale=.3]{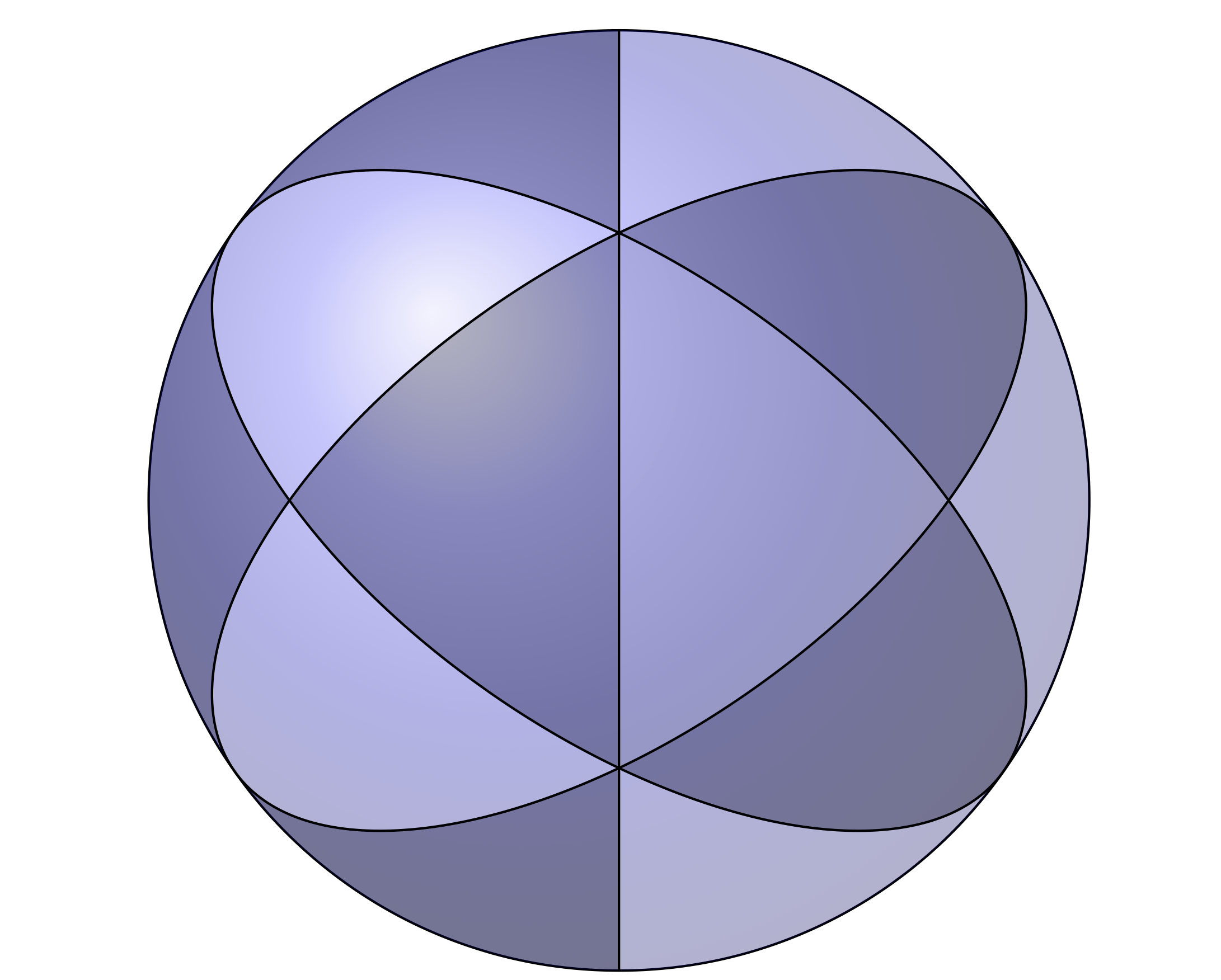}
 \caption{Tetrahedral triangulation}
 \label{tetrahedral}
\end{figure}

\begin{theorem}[Tetrahedral triangulation]\label{T_t}   Let $m_{\pmb\beta}$ be the unit area Gaussian curvature $2\pi(|\pmb\beta|+1)$ metric of $S^2$-like double triangle, see Section~\ref{PrelimTriangulation}.  For the determinant of the Friedrichs Laplacian $\det \Delta_{f^*\pmb\beta}$ corresponding to the pullback metric $f^*m_{\pmb\beta}$ of area $12$, where $f$ is the tetrahedral Belyi map~\eqref{BelyiTetrahedral}, we have
$$
\ba
\log & \frac { \det \Delta_{f^*\pmb\beta}} {12} =   12\Bigl( \log \det \Delta_{\pmb\beta}   +\mathcal C_{\pmb\beta}   -{\bf C}\Bigr)+ C_f
\\
&+ \frac {16} {9}\frac {\Psi(\beta_0,\beta_1,\beta_\infty)}{\beta_0+1} + \frac3 2 \frac {\Psi(\beta_1,\beta_0,\beta_\infty)}{\beta_1+1} + \frac {16} {9}\frac {\Psi(\beta_\infty,\beta_1,\beta_0)}{\beta_\infty+1}
\\
    &    -\frac 2 3   \left(3(\beta_0+\beta_\infty+2)+\frac 1 { 3(\beta_\infty+1)} +\frac 1 { 3(\beta_0+1)}  \right)\log 3 - \left(   2(\beta_1+1)+\frac 1{2(\beta_1+1)}\right)   \log 2
 \\
 &  -4\mathcal C\left(  3\beta_0+2 \right) -6\mathcal C\left(  2\beta_1+1\right)-4\mathcal C\left(  3\beta_\infty+2\right)
  +{\bf C},
\ea
$$
where  $\mathcal C_{\pmb\beta} =\mathcal C(\beta_0)+\mathcal C(\beta_1)+\mathcal C(\beta_\infty)$ and $C_f$ is the same as in~\eqref{CfTT}.

Recall that  $\det \Delta_{\pmb\beta}$ stands for an explicit function~\eqref{CalcVar}, whose values are the determinants of the unit area $S^2$-like double triangles of Gaussian curvature $2\pi(|\pmb\beta|+2)$ tessellating the singular sphere $(\overline{\Bbb C}_x,f^*m_{\pmb\beta})$. The function $\beta\mapsto \mathcal C(\beta)$ is defined in~\eqref{cbeta}, the function $\Psi$ is the same as in~\eqref{eqn_Psi}, and $\bf C$ is the constant introduced in~\eqref{bfC}. 
\end{theorem}
\begin{proof}
For the derivative of the Belyi function in~\eqref{BelyiTetrahedral} we have
$$
 f'(x)=\frac{192(x^3+1)^2(x^6+20x^3-8)}{x^4(x^3-8)^4}.
 $$
 In exactly the same way as in the proof of Theorem~\ref{Jan20} we obtain
 $$
 |c_k|=\left\{
\begin{array}{cc}
 2^6/3^3, &   x_k\in\{\sqrt[3]{-1}\},   \\
 2\sqrt 3 \pm 3, &   x_k\in\left\{\sqrt[3]{-10\pm 6\sqrt  3}\right\},   \\
 2^6,  &      x_k=\infty,\  k=n,\\
 2^3, &    x_k=0,\\
 2^3/3^3, & x_k\in \left\{\sqrt[3]8\right\}.
\end{array}
 \right. 
 $$
This together with the expression~\eqref{C_f} for $C_f$ implies the value stated in~\eqref{CfTT}. The remaining part of the assertion is a direct consequence of Theorem~\ref{main}.
\end{proof}

\begin{example}[Tetrahedron]{\rm
Here we deduce an alternative formula for the spectral determinant of a regular tetrahedron of Gaussian curvature $ K=2\beta+1 $: in Theorem~\ref{T_t} we substitute
$$
\pmb\beta=\left(-\frac 2 3 \right)\cdot 0+\left(-\frac 1 2 \right)\cdot 1+\frac{\beta-2}{3}\cdot\infty.
$$
As a result, after some rescaling we obtain
$$
\ba
\log & \det \Delta^{4\pi}_{Tetrahedron}  =   12\Bigl( \log \det \Delta_{\pmb\beta}   +\mathcal C_{\pmb\beta}  \Bigr)+\log 2 +\frac {7}{12}\log 3      +\frac 4 3 \log\pi
\\
&+ \frac {16} {3} {\Psi\left(-\frac 2  3,-\frac 1 2,\frac{\beta-2}{3}\right) + 3 \Psi\left(-\frac 1 2,-\frac 2 3,\frac{\beta-2}{3}\right)} + \frac {16} {3 (\beta+1)  } \Psi\left(\frac{\beta-2}{3},-\frac 1 2,-\frac 2 3\right)
\\
    &    -\frac 1 3   \left(\beta+1+\frac 1 { \beta+1}  \right)   \log 3\pi   -4\mathcal C\left(  \beta\right) -11{\bf C},
\ea
$$
where the right hand side is an explicit function of $\beta$. A graph of this function is depicted in Fig.~\ref{Platonic} as a solid line. 

In the case $\beta=0$ the surface $(\overline{\Bbb C}_x, f^*m_{\pmb\beta})$ is isometric to a unit sphere in $\Bbb R^3$. The sphere  $(\overline{\Bbb C}_x, f^*m_{\pmb\beta}|_{\beta=0})$ is tessellated by the double of $(2,3,3)$-triangle, and the above formula for the determinant is a representation for the determinant $ \det \Delta^{4\pi}_{Tetrahedron}|_{\beta=0}$ of the unit sphere in terms of the determinant of the double of spherical $(2,3,3)$-triangle. 
}
\end{example}

\subsection{Determinant for octahedral triangulation}
To the octahedral triangulation, there corresponds the Belyi function
\be\label{cube}
f(x)=-2^2 \cdot3^3\frac{(x^4+1)^4 x^4}{(x^8-14 x^4+1)^3},\quad \deg f=24.
\ee
In particular,  thanks to the identification of the Riemann sphere $\overline{\Bbb C}_x$  with the standard round sphere in $\Bbb R^3$ via the stereographic projection, $f$ describes the tessellation of the standard round sphere with spherical $(2,3,4)$-triangles, cf. Fig~\ref{Oct Triang}.
The ramification divisor of $f$ is 
\be\label{divisorF}
\pmb f=3\cdot\left\{0,\infty,\sqrt[4]{-1}\right\}+1\cdot\left\{\sqrt[4]{1}, \sqrt[4]{-17\pm 3\cdot 2 ^{5/2}}\right\}+2\cdot\left\{\sqrt[4]{(2\pm\sqrt 3)^2}\right\},\quad |\pmb f|=46.
\ee
Here $\{\sqrt[4]{1}, \sqrt[4]{-17\pm 3\cdot 2 ^{5/2}}\}$ are the edge midpoints of a cube, for any $x$ in this set we have $ f(x)=1$. The poles $\sqrt[4]{(2\pm\sqrt 3)^2}$  of $f$ correspond to the vertices of the cube.  The zeros $\{0,\infty,\sqrt[4]{-1}\}$ of $f$ are the centers of the cube faces.

For the octahedron dual to the cube: the points  $\{\sqrt[4]{1}, \sqrt[4]{-17\pm 3\cdot 2 ^{5/2}}\}$ correspond to the edge midpoints, the poles of $f$ correspond to the centers of the faces, and the zeros of $f$ are the vertices.  A picture of the corresponding {\it dessin d'enfant} can be found e.g. in~\cite[Fig. 3]{Margot Zvonkin}.

In the proof of Theorem~\ref{TOT} below we show that to the Belyi function~\eqref{cube} there corresponds the constant
\be\label{CfOT}
C_f=\frac 9 4 \log 3 +\frac{119}{18}\log 2.
\ee
For the pullback of the divisor $\pmb \beta=  \beta_0\cdot 0+\beta_1\cdot 1+\beta_\infty \cdot \infty$ by $f$ we get
$$
\ba
f^*\pmb\beta=(4\beta_0+3)\cdot\left\{0,\infty,\sqrt[4]{-1}\right\}+& (2\beta_1+1)\cdot\left\{\sqrt[4]{1}, \sqrt[4]{-17\pm 3\cdot 2 ^{5/2}}\right\}
\\
&\qquad\qquad + (3\beta_\infty+2)\cdot\left\{\sqrt[4]{(2\pm\sqrt 3)^2}\right\}.
\ea
$$
\begin{figure}[h]
 \centering\includegraphics[scale=.3]{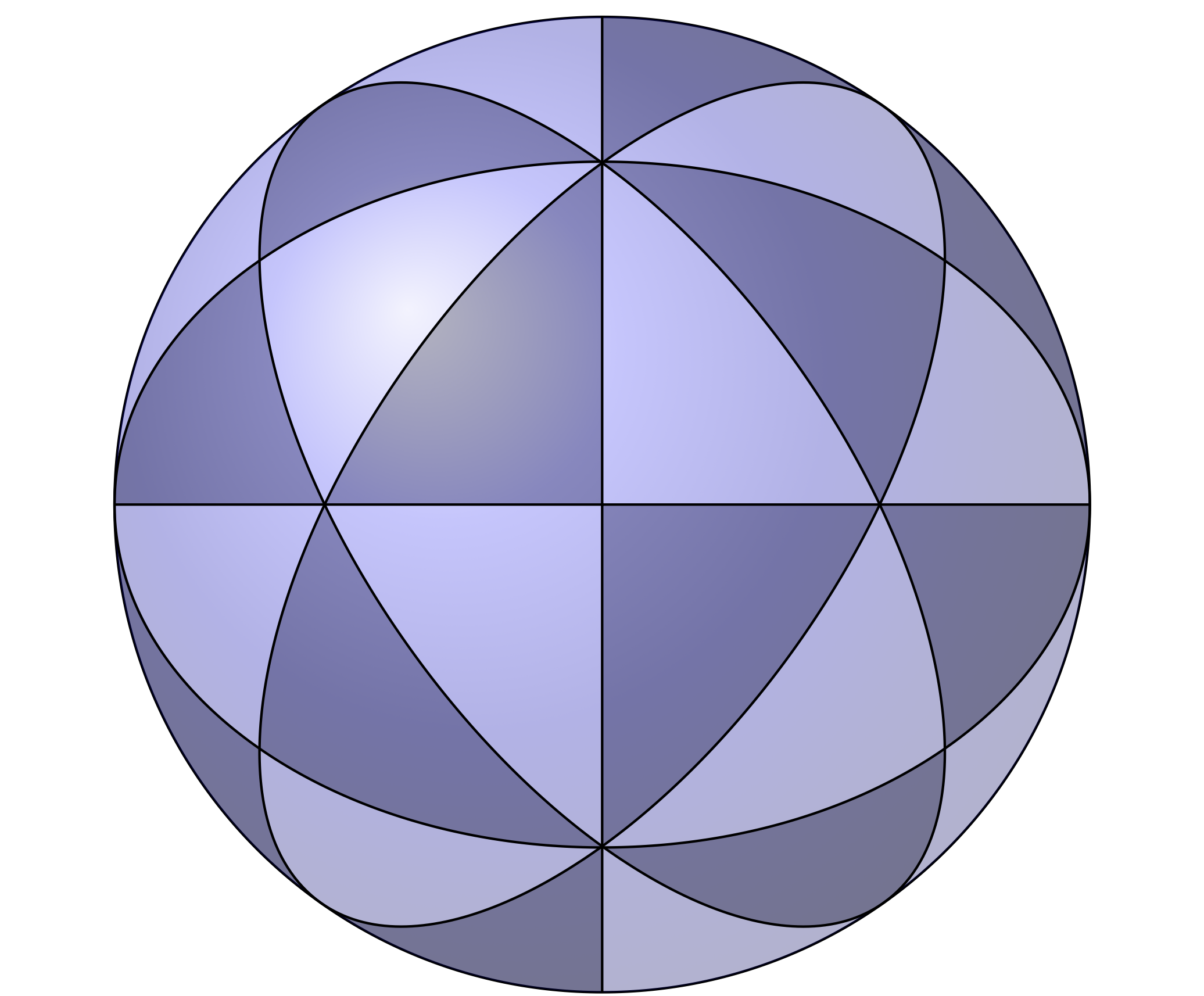}
\caption{Octahedral triangulation}\label{Oct Triang}
\end{figure}

\begin{theorem}[Octahedral triangulation]\label{TOT}  Let $m_{\pmb\beta}$ be the unit area Gaussian curvature $2\pi(|\pmb\beta|+1)$ metric of $S^2$-like double triangle, see Section~\ref{PrelimTriangulation}. For the determinant of Laplacian corresponding to the pullback metric $f^*m_{\pmb\beta}$ of area $24$, where $f$ is the  Belyi map~\eqref{cube}, we have
$$
\ba
\log & \frac { \det \Delta_{f^*\pmb\beta}} {24} =   24\Bigl( \log \det \Delta_{\pmb\beta}   +\mathcal C_{\pmb\beta}   -{\bf C}\Bigr)+ C_f
\\
&+ \frac {15} {4}\frac {\Psi(\beta_0,\beta_1,\beta_\infty)}{\beta_0+1} + 3\frac {\Psi(\beta_1,\beta_0,\beta_\infty)}{\beta_1+1} +\frac {32} 9 \frac {\Psi(\beta_\infty,\beta_1,\beta_0)}{\beta_\infty+1}
\\
    & - \left(4(2\beta_0+\beta_1+3)+\frac 1 {2(\beta_0+1)} +\frac 1 {\beta_1+1}\right)\log 2  -4\left(\beta_\infty+1+\frac 1 {9(\beta_\infty+1)}\right)\log 3 
    \\&
     -8\mathcal C\left(  3\beta_\infty+2 \right) -6\mathcal C\left(  4\beta_0+3\right)-12\mathcal C\left(  2\beta_1+1\right)
  +{\bf C},
\ea
$$
where $\mathcal C_{\pmb\beta} =\mathcal C(\beta_0)+\mathcal C(\beta_1)+\mathcal C(\beta_\infty)$  and $C_f$ is the same as in~\eqref{CfOT}.
  
Recall that  $\det \Delta_{\pmb\beta}$ stands for an explicit function~\eqref{CalcVar}, whose values are the determinants of the unit area $S^2$-like double triangles of Gaussian curvature $2\pi(|\pmb\beta|+2)$ tessellating the singular sphere $(\overline{\Bbb C}_x,f^*m_{\pmb\beta})$. The function $\beta\mapsto \mathcal C(\beta)$ is defined in~\eqref{cbeta}, the function $\Psi$ is the same as in~\eqref{eqn_Psi}, and $\bf C$ is the constant introduced in~\eqref{bfC}. 
\end{theorem}

\begin{proof}  
Here we find $C_f$  by using Proposition~\ref{B_C_f}.    
Namely, we  first write the potential of the pullback of $m_{\pmb\beta}=c|z|^{-4/3}|z-1|^{-4/3}|dz|^2$ by $f$ in the form
$$
\bigl(f^*\phi\bigr)(x)=\frac 1 3 \log |x|+\frac 1 3 \log |x^4+1| -\frac 1 3 \log |x^4-1|-\frac 1 3 \log |(x^4+17)^2-9\cdot 2^5|+\log cA,
$$
cf.~\eqref{pb2}. It is easy to see that as $x\to x_k\in\{\sqrt[4]{-1}\}$ the metric potential $f^*\phi$ satisfies the asymptotics
$$
\ba
\bigl(f^*\phi\bigr)(x)& = \frac 1 3 \log |x-x_k| +\frac 1 3 \log\left |(x^4+1)'\restriction_{x=\sqrt[4]{-1}}\right | -\frac 1 3 \log |-2|
\\
& \qquad-\frac 1 3 \log |(-1+17)^2-9\cdot 2^5|+\log cA+o(1)
\\
& =\frac 1 3 \log |x-x_k| +\log cA-\frac 4 3 \log 2 +o(1).
\ea
$$
This together with~\eqref{a_s_m} implies that for any $k$, such that $x_k\in\{\sqrt[4]{-1}\}$, we have
$$
\sum_{\ell: k\neq \ell<n}  \frac {\ord_\ell f-2} 3 \log |x_k-x_\ell|=-\frac 4 3 \log 2.
$$

Similarly, we obtain
$$
\sum_{\ell: k\neq \ell<n}  \frac {\ord_\ell f-2} 3 \log |x_k-x_\ell|=
\left\{
\begin{array}{cc}
 0 ,&   x_k=0,  \\
 -\log2 
 -\frac 2 3 \log 3, &  x_k\in\{\sqrt[4]{1}\},   \\
  \frac 1 3 \log \frac {3\pm2\sqrt 2}{144}, &    x_k\in\left\{\sqrt[4]{-17\pm3\cdot 2^{5/2}}\right\}.
\end{array}
\right.
$$

As a consequence,
by using the expression~\eqref{divisorF} for the ramification divisor $\pmb f$, for the first line in~\eqref{C_F} we obtain
\be\label{eqy1}
\ba
\frac 1 {18} \sum_{ k:k<n}\sum_{\ell: k\neq \ell<n}\frac{(\ord_k f-2)(\ord_\ell f -2)}{\ord_k f+1} \log|x_k-x_\ell|= \log2 
 +\frac 2 3 \log 3.
\ea
\ee

Thanks to~\eqref{divisorF}, it is also easy to verify that for the second line in~\eqref{C_F} one has 
\be\label{eqy2}
\frac 1 6 \sum_{k\leq n}  \left( \frac{\ord_k f +1} 3 +\frac 3 { \ord_k f+1}  \right)\log (\ord_k f+1)=
 \frac{17} {2}\log 2+\frac 8 3   \log 3.
\ee

Finally, as the scaling constant $A$ satisfies~\eqref{Exp_A}, we get
$$
A=\frac{|f(x)|^{-2/3}|f(x)-1|^{-2/3}|f'(x)|}{|x|^{1/3}|x^4+1|^{1/3}|x^4-1|^{-1/3}|(x^4+17)^2-9\cdot 2^5|^{-1/3}}=12\cdot 2^{2/3}.
$$
This allows us to calculate the value of the last line  in~\eqref{C_F}: 
\be\label{eqy3}
\frac 1 {6} \left( n-2 - \sum_{k\leq n}  \frac 3 { \ord_k  f+1}    \right)\log A= -  \frac {13} {12} \log (12\cdot 2^{2/3}).
\ee
Adding the values~\eqref{eqy1},~\eqref{eqy2},~\eqref{eqy3} of the lines in~\eqref{C_F} together, we arrive at~\eqref{CfOT}.
\end{proof}

\begin{example}[Cube]{\rm
Here we find the spectral determinant of a cube of Gaussian curvature $4\beta+1$ with (eight) conical singularities of order $\beta$:  in Theorem~\ref{TOT} we substitute
$$
\pmb\beta=\left(-\frac 3 4\right)\cdot 0+  \left(-\frac 1 2\right)\cdot 1+ \frac {\beta-2} 3\cdot \infty.
$$
After rescaling, for the determinant of  a cube of (regularized) Gaussian curvature $4\beta+1$ with  eight conical singularities of order $\beta$ we obtain
\be\label{DetCube}
\ba
\log &\,  { \det \Delta^{4\pi}_{Cube}}  =   24\Bigl( \log \det \Delta_{\pmb\beta}   +\mathcal C_{\pmb\beta} \Bigr)+ \frac {5} 4 \log 3 -\frac{7}{18}\log 2+2\log \pi
\\
&+15  {\Psi\left(-\frac 3 4,-\frac 1 2, \frac {\beta-2} 3\right)} + 6 {\Psi\left(-\frac 1 2,-\frac 3 4, \frac {\beta-2} 3\right)} +\frac {32} {3(\beta+1)}  {\Psi\left(\ \frac {\beta-2} 3,-\frac 1 2,-\frac 3 4\right)}
\\
    &  -\frac 2 3\left( \beta+1+\frac 1 { \beta+1}\right) \log \frac {3\pi} 2 
     -8\mathcal C\left(  \beta \right)   -23{\bf C},
\ea
\ee
where the right hand side is an explicit function of $\beta$. A graph of this function is depicted in Fig.~\ref{Platonic} as a dashed line. 
As $\beta\to -1^+$, the determinant of Laplacian grows without any bound in accordance with the asymptotics
\be\label{IdealCube}
\log { \det \Delta^{4\pi}_{Cube}} =\left(-\frac 4 3 \log(\beta+1) +\frac 2 3 \log 3 -\frac 4 3 +16\zeta'_R(-1)    \right)\frac 1 {\beta+1} -4\log (\beta+1) + O(1) 
\ee
of the right hand side in~\eqref{DetCube}.
In the limit $\beta=-1$  we get a surface isometric to an ideal cube: a  surface of Gaussian curvature $-3$ with eight cusps, cf.~\cite{Judge,Judge2}.  The spectrum of the corresponding Laplacian is no longer discrete.

To the case of a Euclidean cube there corresponds the value  $\beta=-1/4$. The formula~\eqref{DetCube} for the determinant reduces to  
$$
\ba
\log   { \det \Delta^{4\pi}_{Cube}}|_{\beta=-1/4}  =& \frac {32} 3 \zeta_R'(-1) - \frac {37}{18}\log 2  + \frac {25}{12}\log 3
\\
& +   \frac {16} 3 \log \Gamma\left(\frac 2 3\right) - \frac {86} 9 \log \Gamma\left (\frac 3 4\right)  + \frac{19} 9 \log \pi.
\ea
$$
}
\end{example}

\begin{example}[Octahedron]{\rm  Now we can deduce yet another formula for the determinant of Laplacian on a regular octahedron of Gaussian curvature $3\beta+1$ with (six) conical singularities of order $\beta$.
In Theorem~\ref{TOT} we take
$$
\pmb\beta=\frac {\beta-3} 4 \cdot 0+  \left(-\frac 1 2\right)\cdot 1+  \left(-\frac 2 3\right)\cdot \infty.
$$
As a result, after an appropriate rescaling we obtain
$$
\ba
\log & \,{ \det \Delta^{4\pi}_{Octahedron}}  =   24\Bigl( \log \det \Delta_{\pmb\beta}   +\mathcal C_{\pmb\beta}  \Bigr)- \frac {13}{12} \log 3 +\frac{71}{18}\log 2
\\
&+ \frac {15}{\beta+1}\Psi\left(\frac {\beta-3} 4,-\frac 1 2,-\frac 2 3\right ) + 6{\Psi\left(-\frac 1 2,\frac {\beta-3} 4,-\frac 2 3\right)} +\frac {32} 3 \Psi\left(-\frac 2 3,-\frac 1 2,\frac {\beta-3} 4\right)
\\
    & - \frac 1 2 \left(\beta+1+\frac 1 {\beta+1} \right)    \log \frac {8\pi} 3
    -6\mathcal C\left( \beta\right) -23{\bf C}+\frac 5 3 \log \pi.
\ea
$$
}
\end{example}

\subsection{Determinant for icosahedral triangulation}

The icosahedral Belyi function is given by 
\be\label{BMid}
f(x)=1728 \frac{     x^5(x^{10}-11 x^5-1)^5}  {(x^{20}+228(x^{15}-x^{5})+494 x^{10}+1)^3},\quad \deg f =60. 
\ee
The ramification divisor of $\pmb f$ is
$$
\pmb f= 2\cdot  \left\{20 \text{ poles of } f\right\}  + 4\cdot\left\{12 \text{ zeros of } f\right\}+1\cdot\left\{30 \text{ zeros of } f-1\right\},\quad|\pmb f|=118.
$$
This defines tessellation of a standard round sphere with bicolored spherical $(2,3,5)$ double triangle, cf. Fig~\ref{Icos Triang Pic}.  The $20$ poles of $f$ are the coordinates of  the centers of the faces, the $30$ solutions to $f(x)=1$ are the edge midpoints, and the $12$ zeros of  $f$   ($x=\infty$ is also a zero of $f$) are the vertices of a regular icosahedron inscribed into the sphere. In terms of the  dodecahedron that is dual to the  icosahedron: The poles of $f$ are the coordinates of the $20$ vertices,   the $12$ zeros of  $f$ are the centers of its faces,   the $30$ solutions to the equation $f(x)=1$ correspond to the edge midpoints. A picture of the corresponding {\it dessin d'enfant} can be found e.g. in~\cite[Fig. 1]{Margot Zvonkin}.

For the icosahedral Belyi function evaluation of the right hand side in~\eqref{C_F} gives
\be\label{CfIT}
C_f=\frac{139} {15} \log 2 + \frac {63} {10}\log 3 + \frac{125}{36}\log 5.
\ee
For the pullback of the divisor $\pmb\beta$ by $f$ we have
$$
f^*\pmb\beta= (5\beta_0+4)\cdot\left\{12 \text{ zeros of } f\right\}+(2\beta_1+1)\cdot\left\{30 \text{ zeros of } f-1\right\}+(3\beta_\infty+2)\cdot  \left\{20 \text{ poles of } f\right\}. 
$$

\begin{theorem}[Icosahedral triangulation]\label{IcoTr}  Let $m_{\pmb\beta}$ be the unit area Gaussian curvature $2\pi(|\pmb\beta|+1)$ metric of $S^2$-like double triangle, see Section~\ref{PrelimTriangulation}. Then for the determinant of Laplacian corresponding to the area $60$ pullback metric $f^*m_{\pmb\beta}$, where $f$ is the icosahedral Belyi map~\eqref{BMid}, we have
$$
\ba
\log & \frac { \det \Delta_{f^*\pmb\beta}} {60} =   60 \Bigl( \log \det \Delta_{\pmb\beta}   +\mathcal C_{\pmb\beta}   -{\bf C}\Bigr)+ C_f
\\
&+\frac {48} 5 \frac {\Psi(\beta_0,\beta_1,\beta_\infty)}{\beta_0+1} +  \frac {15} 2 \frac {\Psi(\beta_1,\beta_0,\beta_\infty)}{\beta_1+1} + \frac {80} 9  \frac {\Psi(\beta_\infty ,\beta_1,\beta_0)}{\beta_\infty+1}
\\
    &   -2  \left(5\beta_0+5+\frac 1 {5\beta_0+5}\right)\log 5  -5  \left(2\beta_1+2+\frac 1 {2\beta_1+2}\right)\log 2 
    \\&
     -\frac {10} 3  \left(3\beta_\infty+3+\frac 1 {3\beta_\infty+3}\right)\log 3
    \\&
     -12\mathcal C\left(  5\beta_0+4\right)-30\mathcal C\left(  2\beta_1+1\right) -20\mathcal C\left(  3\beta_\infty+2 \right)
   +{\bf C},
\ea
$$
where $\mathcal C_{\pmb\beta} =\mathcal C(\beta_0)+\mathcal C(\beta_1)+\mathcal C(\beta_\infty)$.  
Recall that  $\det \Delta_{\pmb\beta}$ stands for an explicit function~\eqref{CalcVar}, whose values are the determinants of the unit area $S^2$-like double triangles of Gaussian curvature $2\pi(|\pmb\beta|+2)$. The function $\beta\mapsto \mathcal C(\beta)$ is defined in~\eqref{cbeta}, the function $\Psi$ is the same as in~\eqref{eqn_Psi}, and $\bf C$ is the constant introduced in~\eqref{bfC}. 
\end{theorem}
\begin{proof} For this tessellation we find the value~\eqref{CfIT} of $C_f$ in exactly the same way as for the octahedral one in the proof of Theorem~\ref{TOT}. The calculations are a bit tedious but straightforward. We omit the details.   
\end{proof}

\begin{example}[Icosahedron]\label{EIcosahedron}
{\rm 
Here we find an explicit expression for the spectral determinant of a regular icosahedron of area $4\pi$ and Gaussian curvature $6\beta +1$ with $12$ conical points of order $\beta$. In Theorem~\ref{IcoTr} we take
\be\label{divIcos}
\pmb\beta=\frac{\beta-4}5 \cdot 0 +\left(-\frac 1 2\right)\cdot 1 +\left(-\frac 2 3 \right) \cdot \infty.
\ee
In particular, in the case $\beta=0$ all conical singularities disappear and we obtain a surface isometric to the standard round sphere $x_1^2+x_2^2+x_3^2=1$ in $\Bbb R^3$.

As a consequence of Theorem~\ref{IcoTr} and the standard rescaling property, for the divisor~\eqref{divIcos} we obtain
\be\label{DetIcos}
\ba
\log\,  &  { \det \Delta^{4\pi}_{Icosahedron}} =   60 \Bigl( \log \det \Delta_{\pmb\beta}   +\mathcal C_{\pmb\beta} \Bigr)+ \frac{19} {15} \log 2 - \frac {61} {30}\log 3   + \frac{65}{36}\log 5
\\
&+\frac {48} {\beta+1}  {\Psi\left( \frac{\beta-4}5,-\frac 1 2,-\frac 2 3\right)}+  15 {\Psi\left(-\frac 1 2,\frac{\beta-4}5,-\frac 2 3\right)} + \frac {80} 3  {\Psi\left(-\frac 2 3,-\frac 1 2,\frac{\beta-4}5\right)}
\\
    &   -  \left(\beta+1+\frac 1 {\beta+1}\right)\log \frac {5\pi}{3} -12\mathcal C\left(  \beta\right)   -59{\bf C}   +\frac 8 3 \log \pi,
\ea
\ee
where the right-hand side is an explicit function of $\beta$. A graph of this function is depicted in Fig.~\ref{Platonic} as a dash-dotted line. 

As $\beta\to -1^+$ the determinant increases without any bound in accordance with the asymptotics
\be\label{IdealIcosahedron}
\log  { \det \Delta^{4\pi}_{Icosahedron}} =\Bigl( -2\log(\beta+1)+\log 5 -2+24\zeta'_R(-1)  \Bigr)\frac 1 {\beta+1}-6\log(\beta+1)+O(1)
\ee
of the right-hand side in~\eqref{DetIcos}. In the limit $\beta=-1$ we obtain an ideal icosahedron, cf.~\cite{Judge,Judge2,KimWilkin}. The spectrum of the corresponding Laplacian is no longer discrete.

In the case $\beta=-1/6$ we obtain a flat regular icosahedron: 
the pullback of the flat metric  $m_{\pmb\beta}^S=|z|^{-4/3}|z-1|^{-1}|dz|^2$ by $f$ is the metric
$$
f^*m_{\pmb\beta}^S=  3\cdot 2^2\cdot 5^2   |     x(x^{10}-11 x^5-1)|^{-1/3}|dx|^2
$$
of a flat regular icosahedron.  The equality~\eqref{DetIcos} reduces to 
$$
\log  { \det \Delta^{4\pi}_{Icosahedron}} |_{\beta=-1/6}= \frac {96}{5} \zeta_R'(-1)+  \frac {18} 5  \log(\sqrt5-1) + \frac 6 5\log( \sqrt 5+1 )+\frac {23} 2 \log \pi
$$
$$
+\frac {214}{45} \log 2 -\frac {917}{60} \log 3+\frac {251}{36}\log 5    +\frac{72} 5  \log \Gamma\left(\frac 4 5\right)-\frac {211}{5} \log \Gamma\left(\frac 2 3\right)+\frac {24}{5} \log \Gamma\left(\frac 3 5 \right).
$$

Let us also note that for the Klein's tessellation by $(2,3,7)$-triangles, one should take $\beta=-2/7$ in~\eqref{divIcos} and~\eqref{DetIcos}.  This is related to Klein's quartic~\cite{SSW,KW,ShVo}: the genus $3$ surface with the highest possible number of authomorphisms~\cite{harts}, it is also the lowest genus Hurwitz surface.  It is an open longstanding problem to find the exact value of the spectral determinant of Klein's quartic; for a numerical study see~\cite{StrUs}. Klein's quartic is also conjectured to be a stationary point of the spectral determinant, cf.~Theorem~\ref{ExtDet}.

}
\end{example}

\begin{example}[Dodecahedron]{\rm Here we find an explicit expression for the spectral determinant of a regular dodecahedron of area $4\pi$ and Gaussian curvature $K= 10\beta+1$ with twenty conical singularities of order $\beta$. With this aim in mind, in Theorem~\ref{IcoTr}  we take
\be\label{divDod}
\pmb\beta= \left(-\frac 4 {5}\right)\cdot 0 +\left(-\frac 1 2\right)\cdot 1 +\frac {\beta-2} 3 \cdot \infty.
\ee
Then after an appropriate rescaling we obtain
\be\label{DODEXPL}
\ba
\log\, & { \det \Delta^{4\pi}_{Dodecahedron}}  =   60 \Bigl( \log \det \Delta_{\pmb\beta}   +\mathcal C_{\pmb\beta}  \Bigr)+ \frac{19} {15} \log 2 + \frac {33} {10}\log 3 - \frac{127}{36}\log 5
\\
&+ {48}  {\Psi\left(-\frac 4 {5},-\frac 1 2,\frac {\beta-2} 3\right)} +  {15}  {\Psi\left(-\frac 1 2,-\frac 4 {5},\frac {\beta-2} 3\right)} + \frac {80} {3(\beta+1)}   {\Psi\left(\frac {\beta-2} 3,-\frac 1 2,-\frac 4 {5}\right)}
\\
    &  
     -\frac {5} 3  \left(\beta+1+\frac 1 {\beta+1}\right)\log \frac {3\pi}{5} 
    -20\mathcal C\left( \beta \right)-59{\bf C}     +4\log  {\pi},
\ea
\ee
where the right-hand side is an explicit function of $\beta$. A graph of this function is depicted in Fig.~\ref{Platonic} as a long-dashed line. 

As $\beta\to -1^+$,  the determinant increases without any bound in accordance with the asymptotics 
\be\label{IdealDodecahedron}
\ba
\log { \det \Delta^{4\pi}_{Dodecahedron}}=\Bigl(  - \frac {10} 3 (\log (\beta+1)  - \log 3+1) + & 40\zeta'_R(-1)  \Bigr)\frac 1 {\beta+1}
\\
& -10 \log(\beta+1)+O(1)
\ea
\ee
of the right-hand side in~\eqref{DODEXPL}.  In the limit $\beta=-1$ the surface $(\overline{\Bbb C}_x,f^*m_{\pmb\beta})$ becomes isometric to an ideal dodecahedron~\cite{Judge,Judge2}.

In the case of a Euclidean dodecahedron, the equality~\eqref{DODEXPL} reduces to 
$$
\ba
\log\, { \det \Delta^{4\pi}_{Dodecahedron}}|_{\beta=-1/10}=\frac{83}{180} \log 2 - \frac 7{135} \log 3 - \frac {19}{72} \log5 + \frac {19}{108}\log( \sqrt 5-1 ) 
\\
 -\frac {19}{27}\log \Gamma\left(\frac 7{10}\right) -\frac{19} {27}\log\Gamma\left(\frac 4 5 \right)+\frac {19}{27} \log \pi+\frac 1 6-  4\zeta_R'(-1)-20\mathcal C\left(-\frac 1{10}\right),
 \ea
$$
where $\mathcal C\left(-\frac 1{10}\right)$ can be expressed in terms of Riemann zeta and gamma functions~\cite{KalvinJFA}.
}
\end{example}

We end this section with a remark that the spectral determinants of the surfaces of Archimedean solids can be calculated in the same way as above thanks to the Belyi maps found in~\cite{Margot Zvonkin}. Let us also notice that it would be interesting to study the cone to cusp degeneration~\cite{Judge,Judge2,KimWilkin} and obtain, as a consequence of our results, some explicit formulae for the relative spectral determinant~\cite{Mueller} of hyperbolic surfaces with cusps. We hope to address this elsewhere.

\vspace{1cm}

\noindent{ \bf Acknowledgements} I would like to thank Andrea Malchiodi  for correspondence, Iosif Polterovich for discussions of spectral invariants, and  Bin Xu for correspondence and drawing my attention to the reference~\cite{KimWilkin}.

\end{document}